\documentclass[11pt,reqno]{amsart}
\pdfoutput = 1
\usepackage{amsaddr}
\usepackage{amssymb,epsfig,float,color}
\usepackage{siunitx}
\usepackage{upgreek}
\usepackage{natbib}
\usepackage{geometry}
\allowdisplaybreaks 

\newtheorem*{remark}{Remark}
\newtheorem{proposition}{Proposition}
\newtheorem{lemma}{Lemma}
\newtheorem{theorem}{Theorem}

\DeclareMathAlphabet{\mathcalold}{OMS}{cmsy}{m}{n}
\DeclareMathAlphabet{\bmathcalold}{OMS}{cmsy}{b}{n}

\newcommand{\fracd}[2]{\displaystyle
{\frac{{\displaystyle{#1}}}{{\displaystyle{#2}}}}}
\newcommand{\rmd}{\mathrm{d}}
\newcommand{\pp}[2]{\fracd{\partial {#1}}{\partial {#2}}}
\newcommand{\dd}[2]{\fracd{\rmd {#1}}{\rmd {#2}}}

\newcommand{\vfl}{v}
\newcommand{\calH}{\mathcalold{H}}
\newcommand{\Qe}{Q_{\mathrm{e}}}
\newcommand{\Qf}{Q_{\mathrm{f}}}
\newcommand{\Qu}{Q_{\mathrm{u}}}

\newcommand{\Xmax}{X_{\mathrm{max}}}
\newcommand{\bC}{\boldsymbol{C}}
\newcommand{\bI}{\boldsymbol{I}}
\newcommand{\bJ}{\boldsymbol{J}}
\newcommand{\bPhi}{\boldsymbol{\Phi}}
\newcommand{\bR}{\boldsymbol{R}}
\newcommand{\bSf}{\boldsymbol{S}_{\mathrm{f}}}

\newcommand{\bS}{\boldsymbol{S}}
\newcommand{\bzero}{\boldsymbol{0}}
\newcommand{\jf}{j_{\mathrm{f}}}
\newcommand{\mumax}{\mu_{\mathrm{max}}}
\newcommand{\qjph}{q_{j+1/2}}
\newcommand{\vhs}{v_{\mathrm{hs}}}

\newcommand{\bCf}{\bC_{\rm f}}

\newcommand{\QCphm}{\vfl^{X,-}_{j+1/2}}
\newcommand{\QCphn}{\vfl^{X,n}_{j+1/2}}
\newcommand{\QCphp}{\vfl^{X,+}_{j+1/2}}
\newcommand{\QCmhnm}{\vfl^{X,n,-}_{j-1/2}}
\newcommand{\QCphnm}{\vfl^{X,n,-}_{j+1/2}}
\newcommand{\QCmhnp}{\vfl^{X,n,+}_{j-1/2}}
\newcommand{\QCphnp}{\vfl^{X,n,+}_{j+1/2}}

\DeclareMathAlphabet{\mathbcal}{OMS}{cmsy}{b}{n}
\DeclareMathAlphabet{\mathpzc}{OT1}{pzc}{b}{n}

\makeatletter
\renewcommand\@biblabel[1]{}
\makeatother

\title[Reactive settling with varying cross-sectional area]{A method-of-lines formulation for a model of reactive settling in tanks with varying cross-sectional area}
\author[B\"urger]{Raimund B\"urger}
\address[A1]{\vspace{-0.1cm}CI${}^{\mathrm{2}}$MA and Departamento de Ingenier\'{\i}a Matem\'{a}tica, Facultad de Ciencia\'{i}sicas y Matem\'{a}ticas, Universidad de Concepci\'{o}n, Casilla 160-C, Concepci\'{o}n, Chile}
\author[Careaga]{\vspace{-0.5cm}Julio Careaga$^{*}$}
\address[A2]{\vspace{-0.1cm}Centre for Mathematical Sciences, Lund University, P.O.\ Box 118, S-221 00 Lund, Sweden}
\author[Diehl]{\vspace{-0.5cm}Stefan Diehl}
\address[A2]{\vspace{-0.3cm}Centre for Mathematical Sciences, Lund University, P.O.\ Box 118, S-221 00 Lund, Sweden\vspace{0.3cm}
 {\small $^*$Corresponding author: \textsf{julio.careaga@math.lth.se}}}

 \usepackage[font=small,labelfont=bf]{caption}
 
\begin{document}

\maketitle

\begin{abstract}
{Reactive settling denotes the  process of sedimentation of small solid particles dispersed 
 in a viscous fluid with simultaneous reactions between the  components that  constitute 
  the solid and liquid phases. This process is of particular importance for the simulation and 
   control of secondary settling tanks (SSTs) in water resource recovery facilities (WRRFs), formerly known as wastewater treatment plants. 
    A spatially one-dimensional model of reactive settling in an SST is formulated by combining a mechanistic model of sedimentation 
     with compression with a model of biokinetic reactions. In addition, the cross-sectional area of the tank 
      is allowed  to vary as a function of height. The final model is a system 
      of   strongly degenerate parabolic, nonlinear partial differential equations (PDEs) 
       that include discontinuous coefficients to describe the feed, underflow and overflow mechanisms, 
        as well as singular source terms that model the feed mechanism.  
         A  finite difference scheme for the final model is derived  by first deriving a 
          method-of-lines formulation (discrete in space, continuous in time), and then 
           passing to a fully discrete scheme by a time discretization. The advantage 
            of this formulation is its  compatibility with common practice 
             in development of software for WRRFs. 
               The main mathematical result is an invariant-region property, which implies that physically relevant numerical solutions are produced. Simulations of denitrification in SSTs in wastewater treatment illustrate the model and its discretization.} 
{Secondary settling tank,  multi-component flow,  wastewater treatment, 
 degenerate parabolic equation, method-of-lines formulation, finite-difference method.}
\\
2000 Math Subject Classification: 65M06, 35K57, 35Q35  
\end{abstract}

\section{Introduction}

\subsection{Scope}\label{sect:scope}

Reactive settling denotes the combined process of sedimentation of small solid particles, each consisting of several components, dispersed in a viscous fluid with simultaneous reactions between the solids and soluble components in the fluid.
This process 
is of particular importance in secondary settling tanks (SSTs)  within the activated sludge process in water resource recovery facilities (WRRFs), formerly known as wastewater treatment plants. 
   The primary purpose of an SST (see Figure~\ref{fig:vesselA}) is to allow the biomass (essentially, bacteria) to settle out from 
    the process effluent of a bioreactor. The overflow produced by the SST should ideally be water, while most of the sediment (activated sludge) leaves the unit through the underflow and is recycled to the bioreactor. On the other 
      hand, significant  biokinetic reactions are going on in an SST, in particular denitrification, which  is 
       the conversion  of nitrate ($\mathrm{NO}_3$) into nitrogen ($\mathrm{N}_2$).  
   An excessive production of nitrogen, however, may led to bubbles that cause  biomass particles 
    to float and leave the SST with the effluent; this situation is highly  undesirable. 
We referto \cite{metcalf} and \cite{droste} for the background of wastewater treatment. 
    Mathematical models able to capture reactive settling, that is 
     the mechanical sedimentation process in an SST in combination with biological reactions 
     are urgently needed to allow for the simulation of operational scenarios.

There are two purposes of this work.
One is to extend the model  of reactive settling by \cite{SDm2an_reactive} by  including 
  dispersion (mixing effects) and  tanks with  a varying cross-sectional area.
The other purpose is to advance  a new numerical scheme, which is the main contribution of this work.
The scheme is easy to implement and to include in commercial simulation softwares for ordinary differential equations (ODEs), which are frequently used for the simulation of biological reactions in WRRFs and require method-of-lines (MOL) form for the simulation of partial differential equations (PDEs).

In contrast to \cite{SDm2an_reactive}, 
    here the main system of PDEs is formulated 
     in terms  of the  concentrations of solid particles and soluble components as unknowns instead 
      of using  percentages. 
By including a cross-sectional-area function, the model 
can be seen as a quasi-one-dimensional approach that allows simulation in more realistic tanks. 
The governing model can be stated as the following system of convection-diffusion-reaction equations, 
where $z\in \mathbb{R}$ is depth and $t \geq 0$ is time: 
\begin{align} \label{eq:model} \begin{split} 
&A(z)\pp{\bC}{t} + \pp{}{z}\big(A(z) \mathcalold{F}_{\boldsymbol{C}} (z,t,X)  \bC\big)  = \pp{}{z}\left(A(z)\gamma(z)\pp{D_{\bC}(X)}{z} \boldsymbol{C} \right) + \bmathcalold{B}_{\bC}(\boldsymbol{C},\bS,z,t),  
\\
& A(z)\pp{\bS}{t} + \pp{}{z}\big( A(z) \mathcalold{F}_{\boldsymbol{S}} (z,t,X)  \bS\big)  =  \pp{}{z}\left(A(z)\gamma(z)\mathbcal{D}\pp{\bS}{z}\right)   + \bmathcalold{B}_{\bS}(\boldsymbol{C},\bS,z,t), 
\\
& \boldsymbol{C} = \bigl( C^{(1)}, \dots, C^{(k_{\boldsymbol{C}})} \bigr)^{\mathrm{T}}, \quad 
 \boldsymbol{S} = \bigl( S^{(1)}, \dots, S^{(k_{\boldsymbol{S}})} \bigr)^{\mathrm{T}}, \quad 
  X= X(z,t) = C^{(1)} (z,t) + \dots + C^{(k_{\boldsymbol{C}})} (z,t), 
\end{split} 
\end{align}
 The unknowns are the vectors of solid concentrations $\bC = \bC(z,t)$ and of concentrations of soluble components $\bS = \bS(z,t)$, and $X$~denotes  the total concentration of solids. The function~$A= A(z)$ is the (variable) cross-sectional area, and $\gamma$~is a characteristic function which   equals one  inside the vessel and zero otherwise. 
 The scalar functions $\mathcalold{F}_{\boldsymbol{C}}$ and $\mathcalold{F}_{\boldsymbol{S}}$ depend 
  discontinuously on~$z$ and non-linearly on~$X$ and represent portions of the solid and liquid phase velocity, respectively, 
   in different parts of the tank.  The scalar 
    function $D_{\bC}$ models sediment compressibility, and  $\bmathcalold{D}$ 
     is a diagonal matrix  of  diffusion coefficients for the equations of the soluble components. 
      The terms $\bmathcalold{B}_{\boldsymbol{C}}$ and $\bmathcalold{B}_{\boldsymbol{S}}$ involve the feed and reaction terms for the solids and soluble components, respectively. 
        All ingredients are specified in detail in Section~\ref{sec:two}. The model \eqref{eq:model} is
   supplied with  a suitable initial condition; no boundary conditions are required. 
   
The main difficulties for the mathematical and numerical treatment of \eqref{eq:model} arise partly from the discontinuous dependence of 
$\mathcalold{F}_{\boldsymbol{C}}$, $\mathcalold{F}_{\boldsymbol{S}}$, and the diffusion terms (via the presence of $\gamma(z)$) on spatial position~$z$, partly from the presence of singular source terms 
      (within $\bmathcalold{B}_{\boldsymbol{C}} $ and $\bmathcalold{B}_{\boldsymbol{S}}$), and partly from strong type degeneracy; the function $D_{\bC}$ is zero for $X$-values on an interval of positive length. 
       The background of these properties is  provided 
        in Section~\ref{sec:two}.  

We present a new numerical scheme for \eqref{eq:model} that  handles all these difficulties, and that produces 
 approximate solutions that satisfy  certain bounds under a convenient Courant-Friedrichs-Lewy (CFL) condition. In particular, the scheme 
  is positivity preserving.  
The numerical scheme by \cite{SDm2an_reactive} for an equivalent model, but with constant cross-sectional area and no dispersion effect, 
was based on solving within each 
    time step first the scalar PDE for the total solids concentration~$X$, and then inserting the result into the discretized 
    PDEs for the percentages of solid and liquid components. 
We denote that numerical method by Method~XP and the new method presented here by Method~CS, since it computes the concentrations $\bC$ the $\bS$ directly. 
Contrary to Method~XP, Method~CS is compatible with the practice of commercial WRRF simulation 
      software packages that are based on method-of-lines (spatially discretized, 
       continuous in time) formulations for all submodels defined in terms of spatio-temporal PDEs. 
        This principle of simulator design is useful, for example, when simulating a WRRF with biological reactors coupled with sedimentation tanks, for which the  entire model is then a system of coupled ordinary differential equations (ODEs) and PDEs~\citep{SDwatres1}.

     \begin{figure}[t]
\centering 
 \includegraphics[scale=0.55]{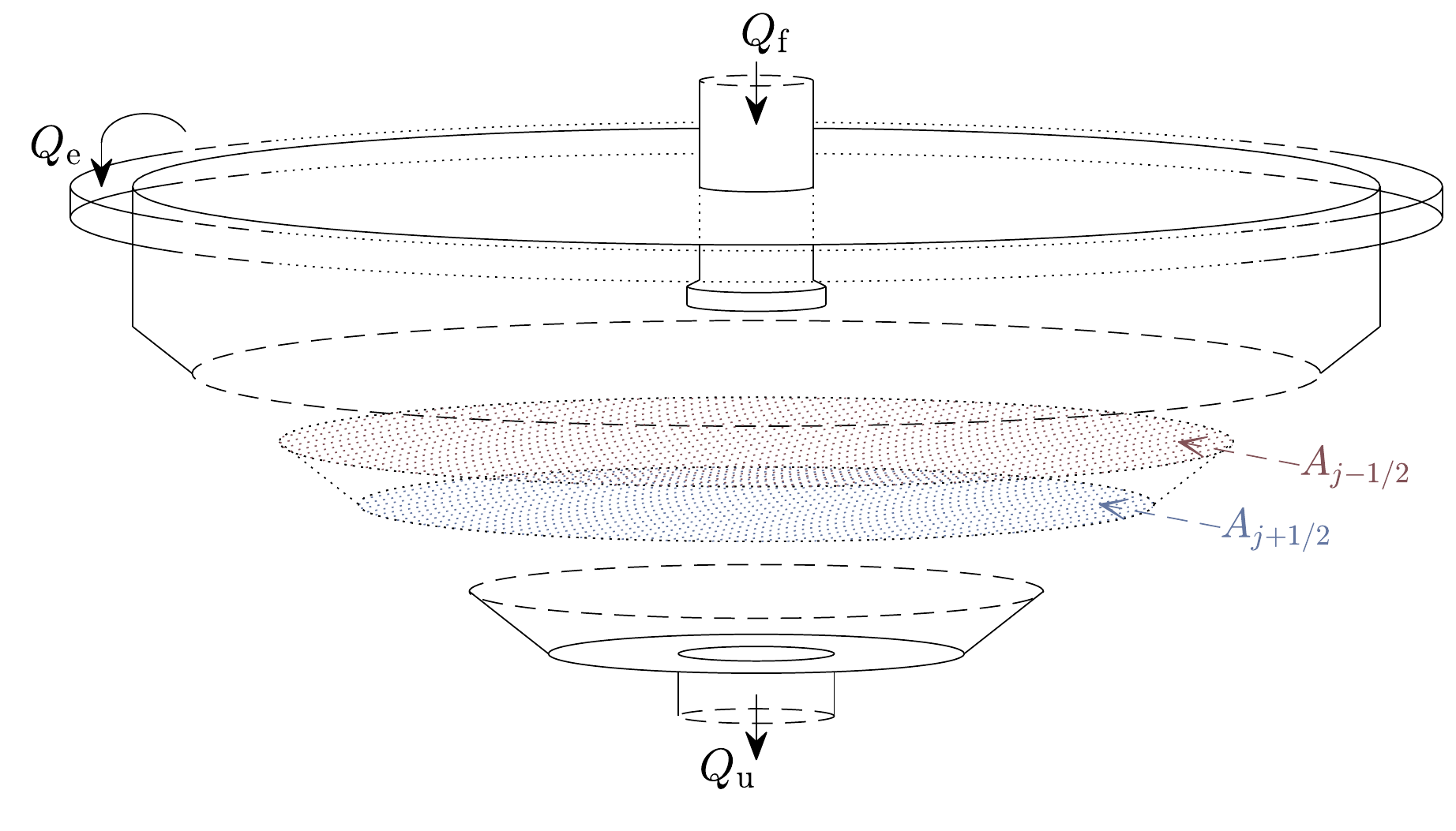}
\caption{Illustration of an axisymmetric secondary settling tank (SST). We assume a quasi-one-dimensional model of the sedimentation tank by letting the cross-sectional area $A = A(z)$ depend on  depth~$z$. The volumetric flows of the feed~$\Qf$, effluent~$\Qe$ and underflow~$\Qu$ are shown, and a volume element of the numerical method is shown. Its centre is located at  depth~$z_j$ and it is bounded above and below by circular discs of areas $A_{j-1/2}$ and $A_{j+1/2}$.} \label{fig:vesselA}
\end{figure}

\subsection{Related work}

References to one-dimensional PDE models for the simulation of continuous sedimentation of solid particles in WRRFs include \cite{Anderson1981, Chancelier1994, SDsiam2, DeClercq2003, Burger&K&T2005a};  and \cite{DeClercq2008}.
In parallel to the PDE development, several ad hoc simulation models have been presented,  of which the one by~\cite{Takacs1991} has been most widely used. That model is based on the subdivision of an SST into layers between which numerical flows are specified. The resulting
 simulation model is, however, not a valid numerical scheme for a PDE model \citep{SDcace1}. 
 
A key difficulty within the framework of one-dimensional PDEs is the nonlinear flux function, which also varies discontinuously with depth due to the inlet and outlets \citep{SDsiam2}.
Another difficulty is a nonlinear, strongly degenerate diffusion term to account for sediment compressibility \citep{Burger&K&T2005a}.
These two publications laid the foundation for the B\"{u}rger-Diehl (BD) model~\citep{SDwatres3, SDCMM2}, which has improved realism in simulations of entire WRRFs \citep{SDwst_sludge_inventory,Li2016_practical}, but above all, given the physically correct numerical solutions with discontinuities satisfying the entropy condition.
The reliability of numerical schemes to handle all the mathematical problems of the PDE model are discussed by \cite{SDcace1}.
Extensions  of the one-dimensional PDE models 
 to include a variable cross-sectional area were made by \cite{Chancelier1994}, \cite{SDsiam3} and \cite{SDcec_varyingA} (see also the references cited in these works).  

The need to model biological reactions occurring in the sedimentation tank has been addressed, for instance,  by \cite{Hamilton1992, Gernaey2006, Alex2011, Flores2012, Ostace2012, Guerrero2013}; and \cite{Li2013_oxidation}. 
A common idea has been to use the layered simulation model by \cite{Takacs1991} and to add to each layer   a system of ODEs modelling the biological reactions.
\cite{SDcace_reactive} employed a PDE batch settling model to simulate  denitrification in an SST. 
The model consists of 
  two  solid and three soluble components, where the latter are modelled by advection-diffusion equations with a constant diffusion/dispersion coefficient for all components.
\cite{Kirim2019} added the biokinetic ASM1 model~\citep{Henze2000ASMbook} to the BD model and included a varying cross-sectional area for the simulation and comparison with real data.
A PDE model and numerical scheme for continuous settling with reactions was presented by~\cite{SDm2an_reactive} and the differences between that and the present work is described in Section~\ref{sect:scope}.

\subsection{Outline of the paper} 

The remainder of this work is organized as follows. The governing model is developed in Section~\ref{sec:two} in the following steps.
 The model consists of two  phases, solid and liquid, each of  which in turn consists of components.
  The assumptions underlying these components, and the reactions between them, are specified in 
   Section~\ref{subsec:ass}. Next, in Section~\ref{subsec:balance}, we outline the balance equations of the solid and liquid components. 
    To turn these balance equations into a solvable PDE model, we utilize 
      in Section~\ref{subsec:vels} various relations between phase velocities and given feed and discharge flows
       so that the 
       unique velocity that remains to be specified to close the model is the solid-liquid relative velocity. 
        The latter is done in Section~\ref{sect:constitutive}, where we recall the expression known from available treatments of sedimentation with compression \citep{Burger&K&T2005a,SDcace1,SDCMM2} that involves constitutive assumptions for the hindered settling velocity and the effective solid stress. Combining all ingredients, we derive in Section~\ref{subsec:gov} the model equations in final form, 
         including explicit formulas of the quantities $\mathcalold{F}_{\boldsymbol{C}}$, $\mathcalold{F}_{\boldsymbol{S}}$,
          $\bmathcalold{B}_{\boldsymbol{C}}$ and~$\bmathcalold{B}_{\boldsymbol{S}}$ arising 
           in~\eqref{eq:model}. Properties of the final PDE system related to hyperbolicity and parabolicity are given in 
            Section~\ref{subsec:proppde}. Section~\ref{sec:scheme} is devoted to the description of the novel numerical scheme, 
             starting  with the spatial discretization in Section~\ref{subsec:spatdisc}, which requires the definition 
of numerical fluxes associated with boundaries of computational cells (Section~\ref{subsec:numflux}). These considerations lead 
 to spatially discrete, continuous in time MOL  formulation that is described in Section~\ref{subsec:mol}. 
  Based on the MOL formulation, we describe in Section~\ref{subsec:scheme} a fully  discrete  scheme (Method~CS), which is the main contribution of this work.  Then, in Section~\ref{subsec:propscheme} we prove that 
    under a suitable CFL condition the numerical solutions  assume  values in a certain 
      invariant region, that is, assume physically relevant values only and are in particular  non-negative. 
        Numerical examples are presented in Section~\ref{sec:num},  where we simulate the denitrification process carried out in the SSTs in wastewater treatment. Examples~1 and~2 (Sections~\ref{subsec:ex1} and~\ref{subsec:ex2}) present complete simulations with various 
         changes of feed conditions to illustrate the spatio-temporal dynamics of the reactive settling process. 
          Examples~3 to~6, presented in Section~\ref{subsec:ex3to6}, illustrate the effect of various choices of the parameters 
           describing the diffusion of soluble components.  Some conclusions are collected in Section~\ref{sec:conc}. 

\section{The model} \label{sec:two}

\subsection{Assumptions} \label{subsec:ass}

The solid phase consists of flocculated particles (biomass consisting of bacteria) 
 that belong to $k_{\bC}$ different species.
These species  have the concentrations~$\smash{C^{(k)}}$,  $k=1,\dots,k_{\boldsymbol{C}}$, which are collected in the vector~$\bC$.
The liquid phase consists of water of concentration~$W$ and $k_{\bS}$~soluble components of concentrations~$\smash{S^{(k)}}$, $k=1,\dots, k_{\bS}$, which make up the vector~$\bS$.
The total concentrations of solid particles~$X$ and liquid~$L$ are
\begin{align}\label{eq:XLdef}
 X :=  C^{(1)} + \dots + C^{(k_{\boldsymbol{C}})},\quad L := W +  S^{(1)} + \dots + S^{(k_{\boldsymbol{S}})}.
\end{align}
All these concentrations depend on depth~$z$ and time~$t$.
The vectors $\bC$ and $\bS$ contain all components in a bioreactor
model.

We let $\Xmax$ denote the maximum concentration of solids and assume that the density of all solids is the same~$\rho_X>\Xmax$.
 The density of the liquid phase is assumed to be~$\rho_L<\rho_X$, typically the density of water, irrespectively of the concentrations of the soluble components.
If $\phi$ denotes the volume fraction of the solid phase, then 
$X=\rho_X\phi$ and $L=\rho_L(1-\phi)$. 
Eliminating $\phi$ one obtains the fundamental relation 
\begin{gather}\label{eq:L}
L=\rho_L-r{X}, \quad  r:=\rho_L/\rho_X 
\quad\Leftrightarrow\quad
\frac{X}{\rho_X}+\frac{L}{\rho_L}=1.
\end{gather}
The bound $0\leq X\leq\Xmax$ implies the bound $\rho_L-r\Xmax\leq L\leq\rho_L$.

The flocculated particles, and hence all the solid components, are assumed to have the same phase velocity~$v_X$, whereas the velocities of the soluble components are $\smash{v^{(k)}}$, $k=1,\ldots,k_{\bS}$.
The volume fractions of the soluble components are~$\smash{\phi^{(k)}:=\bS^{(k)}/\rho_L}$, $k=1,\ldots,k_{\bS}$, and the liquid average velocity is
\begin{align*}
v_L:=\phi^{(1)}v^{(1)}+\cdots+\phi^{(k_{\bS})}v^{(k_{\bS})}.
\end{align*}
(Since all the liquid subphases have the same density, $v_L$ is both the volume and the mass average velocity.)
The concentrations of the feed inlet $\bCf$ and $\bSf$ are assumed to satisfy (see \eqref{eq:XLdef})
\begin{align*}
 X_{\mathrm{f}} =  C^{(1)}_{\mathrm{f}} + \dots + C^{(k_{\boldsymbol{C}})}_{\mathrm{f}},\quad L_{\mathrm{f}} = W_{\mathrm{f}} +  S^{(1)}_{\mathrm{f}} + \dots + S^{(k_{\boldsymbol{S}})}_{\mathrm{f}}.
\end{align*}
These concentrations and the volumetric flows~$\Qf\geq\Qu>0$ of the feed inlet and the underflow outlet (see Figure~\ref{fig:vesselA}) are assumed to be given functions of~$t$.
The effluent volumetric flow~$\Qe$ will generally depend on $\Qf$, $\Qu$ and the unknown concentrations 
 since  the reactions  may cause a volume change; however, we assume that the tank is always filled and $\Qe\geq 0$.
The feed concentrations are assumed to satisfy~\eqref{eq:L}, i.e., ${X_{\mathrm{f}}}/{\rho_X}+{L_{\mathrm{f}}}/{\rho_L}=1$.

The reaction terms for all particulate and soluble components are collected in the vectors $\bR_{\bC}(\bC,\bS)$ and $\bR_{\bS}(\bC,\bS)$ of lengths $k_{\bC}$ and $k_{\bS}$, respectively, which model the respective increase of bacteria and soluble components.
We assume that the water concentration $W$ does not influence (or is influenced by) any reaction.
Without bacteria there is no growth; $\bR_{\bC}(\bzero,\bS)=\bzero$, and when there is no soluble components, the bacteria cannot consume any such, however, concentrations of soluble components may increase due to decay of bacteria; hence, we assume $\bR_{\bS}(\bC,\bzero)\geq\bzero$.
If one sort of bacteria is not present; no more such can vanish, i.e.,  the  functional form of each component is chosen such that
\begin{align}\label{eq:assumptionRC}
R^{(k)}_{\bC}(\bC,\bS)\bigr|_{C^{(k)}=0} \geq 0. 
\end{align}
We define
\begin{align*}
\tilde{R}_{\bC}(\bC,\bS)&:= R_{\bC}^{(1)}(\bC,\bS) + \dots + R_{\bC}^{(k_{\bC})}(\bC,\bS)   \quad\text{and}\quad
\tilde{R}_{\bS}(\bC,\bS):= R_{\bS}^{(1)}(\bC,\bS) +  \dots +   R_{\bS}^{( k_{\bS})}(\bC,\bS).
\end{align*}
Furthermore,  it is assumed that there is no reaction in the effluent and underflow regions (which model outflow pipes), and 
  that 
the relative velocity between the solid and liquid phases
\begin{align}\label{eq:vrel}
v_X-v_L=:v_{\rm rel}=v_{\rm rel}(X,\partial X/\partial z,z)
\end{align}
is  given by a constitutive function of~$X$ and~$\partial X/ \partial z$.  This function models hindered and compressive settling inside the tank
 (see Section~\ref{sect:constitutive}), whereas
outside the tank, all components have the same velocity; hence,
\begin{align*}
v_{\rm rel}:=0\quad\text{for $z\leq -H$ and $z\geq B$.}
\end{align*}
For the relative motion of the soluble components within the liquid inside the vessel, we assume diffusion of each component relative to the average liquid velocity:
\begin{align}\label{eq:Sdiff}
\big(v^{(k)}-v_L\big)S^{(k)}=-d^{(k)}\frac{\partial S^{(k)}}{\partial z},\quad k=1,\ldots,k_{\bS},
\end{align}
where $d^{(k)}>0$ are diffusion coefficients. Strictly speaking,   several  mechanisms are ``lumped'' into the diffusion  coefficient, 
 namely Fickian (molecular) diffusion, as well as hydrodynamic dispersion (``mixing''). 

To be able to simulate the complex reality with the present  model, we have to make some further technical assumptions to be able to prove an invariant-region property:
\begin{align}\label{eq:tech}
\left. \bR_{\bC}(\bC,\boldsymbol{S})\right|_{X=\Xmax}=0,\qquad v_{\rm rel}(X_{\max},\partial X/\partial z,z)=0.
\end{align}
Simulations with realistic parameter values indicate, however, that the extreme concentrations when these assumptions are in effect never or hardly ever occur.
 Conditions~\eqref{eq:tech} state that when the maximum concentration of biomass is reached ($X=\Xmax$), the biomass cannot grow any more and its relative velocity to the liquid phase is zero.

\subsection{Balance equations} \label{subsec:balance} 

The balance law for each particulate and soluble/liquid component together with the fundamental relationships~\eqref{eq:XLdef} and \eqref{eq:L} gives $k_{\bC}+k_{\bS}+1$ equations for the unknowns $\bC$, $\bS$ and $W$; see~\eqref{eq:govC}--\eqref{eq:govW}.
These equations contain also the unknown velocities $v_X$ and $v^{(k)}$, $k=1,\ldots,k_{\bS}$.
The model is closed with the constitutive relations~\eqref{eq:vrel} and~\eqref{eq:Sdiff}.
Hence, the model equations are the following for $z\in\mathbb{R}$:
\begin{subequations}\label{eq:mod}
\begin{align}
\pp{}{t} \bigl(A(z)\bC \bigr)+\pp{}{z} \bigl(A(z)v_X\bC \bigr) & = \delta(z)\bCf\Qf + \gamma(z)A(z)\bR_{\bC},     \label{eq:govC} \\
\pp{}{t} \bigl(A(z)\bS \bigr)+\pp{}{z}\big(A(z)\mathbcal{D}_v\bS\big) & = \delta(z)\bS_\mathrm{f}\Qf + \gamma(z)A(z)\bR_{\bS}, \label{eq:govS}\\
W&=\rho_L-r{X} -  \bigl(S^{(1)}  + \dots + S^{(k_{\bS})} \bigr),\label{eq:govW}\\
v_X-v_L &=v_{\rm rel}(X,\partial X/\partial z,z),\quad\text{where}\quad v_L= v^{(1)}+\cdots+v^{(k_{\bS})}\label{eq:gov_vrel},\\   
\mathbcal{D}_v\bS-v_L \bS&=-\gamma(z)\mathbcal{D}\frac{\partial \bS}{\partial z},\label{eq:govSdiff}
\end{align}
\end{subequations}
where $\delta(z)$ is the delta function, $\smash{\gamma(z)=\chi_{\{-H<z<B\}}}$, where $\chi_I$ is the indicator function which equals one if and only if $I$ is true, $v_{\rm rel}$ a constitutive function that is specified in Section~\ref{sect:constitutive},  and  the  matrices   $\smash{\mathbcal{D}_v :=  {\rm diag} ( v^{(1)},\dots,v^{(k_{\bS})})}$  and $\smash{\mathbcal{D} := {\rm diag} ( d^{(1)},\dots,d^{(k_{\bS})})}$   come from the vectorized version of \eqref{eq:Sdiff}.

\subsection{Relations between phase, bulk and relative velocities} \label{subsec:vels} 

The next step is to express the solid and liquid velocities~$v_X$ and~$v_L$ in terms 
of   known variables and eliminate~$\mathbcal{D}_v$. To this end, we first write  the average bulk velocity
\begin{align}\label{eq:qdef}
q:=\phi v_X+(1-\phi)v_L
\end{align}
 as a function of the other variables. Summing  
  all equations of~\eqref{eq:govC}, doing the same for \eqref{eq:govS}, and using~\eqref{eq:XLdef}, $X=\rho_X\phi$ and 
  $L=\rho_L(1-\phi)$ (see  Section~\ref{subsec:ass}) 
  we obtain 
\begin{align*}
\pp{}{t} \bigl(A(z)\rho_X\phi \bigr)+\pp{}{z} \bigl(A(z)\rho_X\phi v_X \bigr) &=\delta(z)X_{\mathrm{f}}\Qf+\gamma(z)A(z)\tilde{R}_{\bC},\\
\pp{}{t} \bigl(A(z)\rho_L(1-\phi) \bigr)+\pp{}{z} \bigl(A(z)\rho_L(1-\phi) v_L\bigr) &=\delta(z)L_{\mathrm{f}}\Qf+\gamma(z)A(z)\tilde{R}_{\bS}.
\end{align*}
Dividing  the respective equation by $\rho_X$ and $\rho_L$ and then adding  them yields
\begin{align}\label{eq:temp1}
\pp{}{z}(A(z)q)&= \delta(z)\Qf + \gamma(z)A(z)\mathcalold{R},
\quad\mbox{where}\quad \mathcalold{R} := \frac{\tilde{R}_{\bC}}{\rho_X} + \frac{\tilde{R}_{\bS}}{\rho_L}. 
\end{align}
Integrating \eqref{eq:temp1} from $z$ and $B$ and using~\eqref{eq:qdef}, we get 
\begin{align*}
A(z)q(z,t) 
& = A(B)q(B,t)-\Qf(t)\chi_{\{z\leq 0\}} - 
Q_{\mathrm{reac}}(z,t;\bC,\bS),
\end{align*}
where
\begin{align*}
Q_{\mathrm{reac}}(z,t;\bC,\bS) := \int_{z}^{B}\gamma(\xi)A(\xi) \mathcalold{R}\big(\bC(\xi,t),\bS(\xi,t)\big)\,\rmd\xi,
\end{align*}
and $A(B)q(B,t) = A(B)v_L(B,t) = \Qu(t)$, since $v_{\rm rel}=0$ for $z=B$.
Hence, $q$~is expressed in terms of the other given or unknown variables by 
\begin{align*}
A(z)q(z,t) = \Qu(t) - \Qf(t)\chi_{\{z\leq 0\}} - Q_{\mathrm{reac}}(z,t;\bC,\bS) 
\end{align*}
(however; see the remark below). Consequently, a general definition of the effluent volumetric flow is \begin{align*}
\Qe(t;\bC,\bS):=-A(-H)q(-H,t)=\Qf(t)-\Qu(t) + Q_{\mathrm{reac}}(-H,t;\bC,\bS).
\end{align*}

Introducing $v:=(1-\phi)v_{\rm rel}$, one gets from~\eqref{eq:vrel} and \eqref{eq:qdef}:
\begin{align}
v_X &= q+(1-\phi)v_{\rm rel} = q +v,\label{eq:vX} \\
v_L &= q-\phi v_{\rm rel} = q - \dfrac{\phi}{1-\phi}v.\label{eq:vL}
\end{align}
The next step is to express~$v$ in terms of the other variables by constitutive assumptions.

\begin{remark} 
The dependence of $Q_{\mathrm{reac}}$ on the functions $\bC$ and $\bS$ via an integral means that the dependence is not local.
This is problematic for the analysis of numerical schemes, which otherwise are three-point schemes.
In the application to wastewater treatment, the term $Q_{\mathrm{reac}}$ is negligible (see \cite{SDm2an_reactive}).
In the proof of an invariant-region property of the numerical solution, we have to set $Q_{\mathrm{reac}}:=0$, and we assume this is the case from now on.
Then  $q$  is  defined via
\begin{align}\label{eq:qdef2}
A(z)q(z,t) = \Qu(t) - \Qf(t)\chi_{\{z\leq 0\}}.
\end{align}
\end{remark}

\subsection{Constitutive functions for hindered and compressive settling}\label{sect:constitutive}

We assume that the relative velocity is given by $v_{\rm rel}=v/(1-\phi)=v/(1-X/\rho_x)$, where $v$ is given by
\begin{equation*}
v = 
v(X,\partial X/\partial z,z) = \gamma(z)\vhs(X)\left(
1-\dfrac{\rho_X\sigma_{\mathrm{e}}'(X)}{Xg\Delta\rho}\pp{X}{z}
\right).
\end{equation*}
Here, $\vhs$ is the hindered-settling velocity function, $\sigma_{\mathrm{e}}$ the effective
solids stress, $\Delta\rho:=\rho_X-\rho_L$, and $g$ is the acceleration of gravity. Constitutive functions are needed for $\vhs$
and $\sigma_{\mathrm{e}}$. We require that $\vhs$ is decreasing, 
\begin{align} \label{vhsass} 
\vhs(\Xmax)=0
\end{align} 
 and that the derivative $\sigma_{\mathrm{e}}'(X)$ of the effective solid stress function~$\sigma_{\mathrm{e}}$ satisfies 
\begin{align} \label{smeass} 
\sigma_{\mathrm{e}}(X)
	\begin{cases}
		=0 & \text{for $X \leq X_{\mathrm{c}}$,} \\
		>0 & \text{for $X>X_{\mathrm{c}}$,}
	\end{cases}
\end{align}
where $X_{\mathrm{c}}$ is a critical concentration above which the particles touch each other and form a network that can bear a certain stress.

\subsection{Model equations in final form}\label{subsec:gov}

With the functions
\begin{align*}
d_{\bC}(X) := \frac{\vhs(X)\rho_X\sigma_{\mathrm{e}}'(X)}{Xg\Delta\rho},\qquad 
D_{\bC}(X) := \int_{X_c}^{X}d_{\bC}(s)\,{\rm d}s,
\end{align*}
we can write \eqref{eq:vX} as
\begin{align}
\vfl_X & :=\vfl_X(X,\partial X/\partial z,z,t) 
:= q(z,t) + \gamma(z)\left(\vhs(X) - \pp{D_{\bC}(X)}{z}\right).\label{eq:vlf_C}
\end{align}
Notice that the properties~\eqref{vhsass}   and~\eqref{smeass} imply that 
\begin{align}  \label{dcprop} 
d_{\bC}(X) \begin{cases} >0 & \text{for $X_{\mathrm{c}} < X < X_{\max}$,} \\
=0 & \text{for $X \leq X_{\mathrm{c}}$ and $X= X_{\max}$,}  
\end{cases} 
\end{align}
so the first PDE in \eqref{eq:model} is strongly degenerate  since it degenerates into a first-order PDE 
 on an $X$-interval of positive length (namely, on $[0, X_{\mathrm{c}}]$). 
On the other hand~\eqref{eq:govSdiff} and \eqref{eq:vL} imply
\begin{align*}
\mathbcal{D}_v\bS 
&= v_L \bS-\gamma(z)\mathbcal{D}\frac{\partial \bS}{\partial z}
= \left(q-\frac{X/\rho_X}{1-X/\rho_X}\right)(v_X-q)\bS- \gamma(z)\mathbcal{D}\frac{\partial \bS}{\partial z} 
= \frac{\rho_Xq-\vfl_X X}{\rho_X-X}\bS - \gamma(z)\mathbcal{D}\frac{\partial\bS}{\partial z}.
\end{align*}
Hence, the total mass fluxes of the balance equations~\eqref{eq:govC}--\eqref{eq:govS} can be written as
\begin{align}
\bPhi_{\bC} & := \bPhi_{\bC}(\bC,X,\partial X/\partial z,z,t) := A(z)\vfl_{X}(X,\partial X/\partial z,z,t)\bC,\label{eq:PhiC}\\
\bPhi_{\bS} & := \bPhi_{\bS}(\bS,X,\partial X/\partial z,z,t) := A(z)\left(\frac{\rho_Xq-\vfl_X X}{\rho_X-X}\bS - \gamma(z)\mathbcal{D}\frac{\partial\bS}{\partial z}\right).\label{eq:PhiS}
\end{align} 
Collecting the available results, we see that the model equations~\eqref{eq:mod} can be written as  \eqref{eq:model} for 
\begin{align}   \label{eq:modelterms} \begin{split} 
& \mathcalold{F}_{\boldsymbol{C}} (z,t,X)  =  q(z,t) + \gamma(z) \vhs(X), \quad 
 \mathcalold{F}_{\boldsymbol{S}} (z,t,X) = \frac{ \rho_X q(z,t) - ( q(z,t) + \gamma(z) \vhs(X) )  X }{\rho_X - X},  \\
&  \bmathcalold{B}_{\bC}(\boldsymbol{C},\bS,z,t)   = \delta(z)\bCf\Qf + \gamma(z)A(z)\bR_{\bC}, \quad 
  \bmathcalold{B}_{\bS}(\boldsymbol{C},\bS,z,t) = \delta(z)\bSf\Qf + \gamma(z)A(z)\bR_{\bS},  \end{split} 
\end{align} 
supplied with Equation~\eqref{eq:govW} to calculate  the water concentration $W$ whenever required; note that~$W$ is not present in \eqref{eq:model},  \eqref{eq:modelterms}.  For the development of the numerical method, however, it will be useful to rewrite the  governing PDEs in terms of the total fluxes \eqref{eq:PhiC} and \eqref{eq:PhiS}. We then obtain \begin{subequations}\label{finalmod}
\begin{align}
\pp{(A(z)\bC)}{t}+\pp{\bPhi_{\bC}}{z} & = \delta(z)\bCf\Qf + \gamma(z)A(z)\bR_{\bC},     \label{eq:govCfinal} \\
\pp{(A(z)\bS)}{t}+\pp{\bPhi_{\bS}}{z} & = \delta(z)\bSf\Qf + \gamma(z)A(z)\bR_{\bS} \quad \text{for $z\in\mathbb{R}$ and $t>0$.} \label{eq:govSfinal}
\end{align}
\end{subequations}
No boundary condition is needed. The initial values are the concentrations of solid and liquid components:
\begin{align*}
\bC^0 (z) &= \big(C^{(1),0}(z) ,C^{(2),0}(z) ,\dots, C^{(k_{\bC}),0}(z) \big)^{\mathrm{T}},\quad
\bS^0 (z) = \big(S^{(1),0}(z) ,S^{(2),0}(z) ,\dots, S^{(k_{\bS}),0}(z) \big)^{\mathrm{T}}, \quad z \in \mathbb{R}. 
\end{align*}
Clearly, the corresponding initial total solids and water concentrations are obtained by
\begin{align*}
X^0(z) = C^{(1),0}(z) + \dots +  C^{(k_{\bC}),0}(z) ,\quad W^0 (z) = \rho_L-rX^0 (z) 
 - \bigl(S^{(i),0} (z) + \dots + S^{(k_{\bS}),0} (z) \bigr).  
\end{align*}
We define the solution vector $\boldsymbol{U}(z,t)  :=  (\bC(z,t),\bS(z,t),W(z,t))$ and 
 $\smash{\boldsymbol{U}^0(z)  :=  (\bC^0(z),\bS^0(z),W^0(z))}$. It is presupposed that 
 \begin{align}  \label{init-omega} 
   \boldsymbol{U}^0(z)  \in \Omega \quad \text{for all $z \in \mathbb{R}$,} 
   	 \end{align} 
    where we define the set 
  \begin{align}\label{eq:Omega} \begin{split} 
\Omega := \bigl\{ & \boldsymbol{U} = (\boldsymbol{C}, \boldsymbol{S}, W)  \in\mathbb{R}^{k_{\bC}+k_{\bS}+1}:  \\ 
 & 0\leq  C^{(1)} , \dots, C^{(k_{\bC})}  \leq\Xmax, \quad 
  0\leq  C^{(1)} + \dots + C^{(k_{\bC})}   \leq\Xmax, \quad 
  S^{(1)},  \dots,  S^{(k_{\bS})} \geq 0 \bigr\}.  \end{split} 
\end{align}
It will be shown that under  the condition \eqref{init-omega} the numerical solution assumes values in~$\Omega$. 

\subsection{Properties of the PDE system} \label{subsec:proppde} 

If one assumes that  $X \leq X_{\mathrm{c}}$  (cf.\ \eqref{dcprop}) and in addition sets  $\mathbcal{D}=\bzero$, 
 then the system~\eqref{eq:model}, or equivalently \eqref{finalmod}, reduces to a first-order  system of conservation laws away from source terms, and $\bPhi_{\bC}$ and $\bPhi_{\bS}$ depend only on $X$ (and $z$ and $t$, which we do not write out now). 
This system is recovered if all right-hand sides in \eqref{eq:model} are set to zero. 
As the following proposition implies, this system is non-strictly hyperbolic, which means that its solution for a Riemann initial datum is involved. 
This property is established by examining the eigenvalues of the $(k_{\boldsymbol{C}} + k_{\boldsymbol{S}}) \times (k_{\boldsymbol{C}} + k_{\boldsymbol{S}})$ Jacobian matrix of the conservation law of the conserved variable $A(z)(\bC^\mathrm{T},\bS^\mathrm{T})^\mathrm{T}$, which is
  \begin{align*} 
   \bJ:=\begin{bmatrix} \bJ_{11}&\bzero_{k_{\bC}\times k_{\bS}}\\[1mm] 
\bJ_{21}&\bJ_{22} \end{bmatrix}, 
\end{align*} 
with the sub-matrices (in obvious notation) 
\begin{align*} 
 \boldsymbol{J}_{11} := \pp{(\mathcalold{F}_{\boldsymbol{C}} (z,t, X) \boldsymbol{C} 
    )}{\boldsymbol{C}}    , \quad 
    \boldsymbol{J}_{21} := \pp{(\mathcalold{F}_{\boldsymbol{S}} (z,t, X) \boldsymbol{S} 
    )}{\boldsymbol{C}} , \quad  
    \boldsymbol{J}_{22} := \pp{(\mathcalold{F}_{\boldsymbol{S}} (z,t, X) \boldsymbol{S} 
    )}{\boldsymbol{S}}, 
    \end{align*} 
and where the $k_{\bC}\times k_{\bS}$ block of zeros $\smash{\bzero_{k_{\bC}\times k_{\bS}}}$ appears since $\mathcalold{F}_{\boldsymbol{C}} (z,t, X) \boldsymbol{C}$ does not depend on~$\boldsymbol{S}$. 
Note that the eigenvalues of~\eqref{eq:model} do not depend on $A(z)$. In what follows, we fix~$z$ and~$t$ and write $\mathcalold{F}_{\boldsymbol{C}} (X)$ and $\mathcalold{F}_{\boldsymbol{s}} (X)$ instead of $\mathcalold{F}_{\boldsymbol{C}} (z,t, X)$ and $\mathcalold{F}_{\boldsymbol{s}} (z,t, X)$.

\begin{proposition}\label{prop:eig}
The Jacobian matrix of the flux vector of the system~\eqref{eq:model} has two real eigenvalues:
\begin{align*}
\lambda_1&=q+\gamma f'(X),\quad\text{where}\quad f(X):=\vhs(X)X,\\
\lambda_2&=\mathcalold{F}_{\boldsymbol{S}}(X)
= q-\frac{\gamma(z)f(X)}{\rho_X-X}. 
\end{align*}
\end{proposition}

\begin{proof}
The Jacobian matrix of the flux vector is 
\begin{align*}
\bJ =
\begin{bmatrix}
\mathcalold{F}_{\boldsymbol{C}}(X)\bI_{k_{\bC}} + \mathcalold{F}_{\boldsymbol{C}}'(X)\bC\boldsymbol{1}_{k_{\bC}}^\mathrm{T} & \bzero_{k_{\bC}\times k_{\bC}}\\[1mm]
\mathcalold{F}_{\boldsymbol{S}}'(X)\bS\boldsymbol{1}_{k_{\bS}}^\mathrm{T} & \mathcalold{F}_{\boldsymbol{S}}(X)\bI_{k_{\bS}}
\end{bmatrix},
\end{align*}
where $\smash{\bI_{k_{\bC}}}$ is the identity matrix of size $k_{\bC}\times k_{\bC}$ and $\smash{\boldsymbol{1}_{k_{\bC}}}$ is a column vector of length $k_{\bC}$ full of ones, so that $\smash{\bC\boldsymbol{1}_{k_{\bC}}^\mathrm{T}}$ represents a tensor product.
The Jacobian is block lower triangular with eigenvalues those of $\bJ_{11}$ and $\bJ_{22}$.
The latter matrix is diagonal with the single real eigenvalue $\lambda_2= \mathcalold{F}_{\boldsymbol{S}}$ of multiplicity~$k_{\bS}$, while $\bJ_{11}$~is the sum of a diagonal matrix and a rank-one matrix.
Any eigenvector $\boldsymbol{Z}$ to $\bJ_{11}$ with eigenvalue $\lambda_1$ should satisfy
\begin{align*}
\bJ_{11}\boldsymbol{Z} = \lambda_1\boldsymbol{Z}\quad\Leftrightarrow\quad \mathcalold{F}_{\boldsymbol{C}}(X)\boldsymbol{Z} + \mathcalold{F}_{\boldsymbol{C}}'(X)\boldsymbol{1}_{k_{\bC}}^\mathrm{T}\boldsymbol{Z}\bC = \lambda_1\boldsymbol{Z}.
\end{align*}
Generally, $\mathcalold{F}_{\boldsymbol{C}}'(X)\boldsymbol{1}_{k_{\bC}}^\mathrm{T} \boldsymbol{Z}\bC\neq 0$, and then $\boldsymbol{Z}$ has to be parallel to $\bC$, say $\boldsymbol{Z}=\alpha\bC$, $\alpha\neq 0$.
Since $\boldsymbol{1}_{k_{\bC}}^\mathrm{T}\bC=X$, the corresponding eigenvalue (with multiplicity $k_{\bC}$) is
\begin{align*}
\lambda_1&=\mathcalold{F}_{\boldsymbol{C}}(X) +\mathcalold{F}_{\boldsymbol{C}}'(X)X = 
\frac{\rmd}{\rmd X}\big(\mathcalold{F}_{\boldsymbol{C}}(X)X\big) = 
\frac{\rmd}{\rmd X}\big((q+\gamma\vhs(X))X\big).
\end{align*}
We have thus found all the eigenvalues.
(If $\mathcalold{F}_{\boldsymbol{C}}'(X)=0$, then $\bJ_{11}$ is diagonal with the single real eigenvalue $\mathcalold{F}_{\boldsymbol{C}}(X) =\mathcalold{F}_{\boldsymbol{C}}(X) + \mathcalold{F}_{\boldsymbol{C}}'(X)X = \lambda_1$.
Similarly, if $\bC=\bzero$, then $X=0$ and $\bJ_{11}$ is diagonal with the single real eigenvalue $\mathcalold{F}_{\boldsymbol{C}}(0) =\mathcalold{F}_{\boldsymbol{C}}(0) + \mathcalold{F}_{\boldsymbol{C}}'(0)0 = \left.\lambda_1\right|_{X=0}$.)
\end{proof}

\smallskip 

To see that the eigenvalues are not distinct, we may calculate, for instance, 
\begin{align*}
\left.\lambda_1\right|_{X=0}&=q+\gamma f'(0) >q=\left.\lambda_2\right|_{X=0},\\
\left.\lambda_1\right|_{X=\Xmax}&=q+\gamma f'(\Xmax) \leq q =\left.\lambda_2\right|_{X=\Xmax}.
\end{align*}
Thus, there is some $X \in (0, X_{\max}]$ for which $\lambda_1 = \lambda_2$.

\begin{remark}

Finally, let us briefly comment on the diffusive parts of the PDEs in~\eqref{eq:model}. In order not to 
 complicate the argument, let us assume that $A = \mathrm{const.}$, so we may divide the PDEs by~$A$.
Furthermore, assume that $\gamma(z) =1$. 
In this case  the  diffusion term in the first equation can be written as 
 \begin{align} \label{eq:nice} \begin{split} 
\pp{}{z}\biggl( \pp{D_{\bC}(X)}{z} \boldsymbol{C} \biggr)
& = \pp{}{z}\biggl(  d_{\bC}(X) \pp{X}{z} \boldsymbol{C} \biggr) 
= \pp{}{z}\biggl(  d_{\bC}(X) \biggl(\boldsymbol{1}_{k_{\bC}}^{\mathrm{T}}  \pp{ \boldsymbol{C} }{z}\biggr)  \boldsymbol{C} \biggr)
 =\pp{}{z}\biggl( \boldsymbol{D} ( \boldsymbol{C}) \pp{\bC}{z} \biggr) 
\end{split} 
\end{align} 
with the $\smash{k_{\bC} \times k_{\bC}}$ diffusion matrix 
  $\smash{\boldsymbol{D} ( \boldsymbol{C})  =  
d_{\bC}(X)\boldsymbol{C}\boldsymbol{1}_{k_{\bC}}^{\mathrm{T}} 
=     d_{\bC} \bigl( \boldsymbol{1}_{k_{\bC}}^{\mathrm{T}}  \boldsymbol{C} \bigr) 
 \boldsymbol{C}\boldsymbol{1}_{k_{\bC}}^{\mathrm{T}}}$. 
Assume now that 
 $ X_{\mathrm{c}} < X < X_{\max}$.  Then $ \boldsymbol{D} (\boldsymbol{C})$ is a 
 rank-one matrix   whose only nonzero eigenvalue equals 
 \begin{align*} 
  \mu = \mu ( \boldsymbol{C}) =   d_{\bC} \bigl( \boldsymbol{1}_{k_{\bC}}^{\mathrm{T}}  \boldsymbol{C} \bigr)  
   \boldsymbol{1}_{k_{\bC}}^{\mathrm{T}}  \boldsymbol{C} =  d_{\bC} (X) X, 
  \end{align*}  
  with $\boldsymbol{C}$ (or a multiple of it) being the corresponding eigenvector. 
  Since $\mu ( \boldsymbol{C}) >0 $, the matrix 
   $ \boldsymbol{D} (\boldsymbol{C})$ is positive semidefinite, and therefore the corresponding system of PDEs is 
    parabolic in the sense of Petrovsky (or simply {\em parabolic}) \citep{fried64,lsu68,eidel69,tayl96}. 
     Furthermore, if $\boldsymbol{C}$ is a vector such that $C^{(i)} >0$ for $i  = 1, \dots , 
      k_{\mathrm{C}}$, then  $ \boldsymbol{D} ( \boldsymbol{C})$ is a full matrix with no zero entries, 
       so in principle the  model involves  cross diffusion 
        (that is, the  diffusive flux of any species $\smash{C^{(i)}}$ does not only depend 
         on $\smash{\partial  C^{(i)}/ \partial z}$, but on   $\smash{\partial  C^{(m)}/ \partial z}$
          for all $m= 1, \dots, k_{\boldsymbol{C}}$). 
           
  For the particular case $k_{\boldsymbol{C}} =1$ and $k_{\boldsymbol{S}} =0$ and 
 if no reactions take place ($\boldsymbol{R}_{\boldsymbol{C}}$ and $\boldsymbol{R}_{\boldsymbol{S}}$ are set to zero), 
  the model reduces to the well-known mechanistic B\"{u}rger-Diehl (BD) model of sedimentation with compression. 
 In this case we may identify $C:= C^{(1)} = \boldsymbol{C}$ and $X$. Calculations similar to \eqref{eq:nice} then reveal that the nonlinear diffusion term in the first PDE of \eqref{eq:model} satisfies 
   \begin{align*}
    \pp{D_{\boldsymbol{C}}  (X)}{z} X = d_{\boldsymbol{C}} \pp{X}{z} X = 
     \frac{v_{\mathrm{hs}} (X) \rho_X \sigma_{\mathrm{e}}' (X)}{X g \Delta \rho}
      \pp{X}{z} X =  \frac{v_{\mathrm{hs}} (X) \rho_X \sigma_{\mathrm{e}}' (X)}{g \Delta \rho} \pp{X}{z},
      \end{align*} 
which is the diffusion term accounting for sediment compressibility within the BD model
    \citep{SDcace1,SDCMM2,SDm2an_reactive}. The agreement  of  the convection term 
     ($\mathcalold{F}_{\boldsymbol{S}} (z, t, X) X$, in this case) and of 
      the terms describing the feed source with those of the BD model is  easily verified. 
      Furthermore, in this case the water concentration is $W=\rho_L- rX$ (see \eqref{eq:govW}). 
   \end{remark}

\section{Numerical scheme} \label{sec:scheme} 

\subsection{Spatial discretization} \label{subsec:spatdisc}

We divide the tank into $N$~internal computational cells,
or layers, so that each layer has the depth $\Delta z = (B+H)/N$. The location of layer~$j$ is such that its
midpoint has the coordinate $z_j$, hence the layer is the interval
$[z_{j-1/2},z_{j+1/2}]$. The top layer~1 in the
clarification zone is thus the interval $[z_{1/2},z_{3/2}]=[-H,-H+\Delta z]$,
and the bottom location is $z=z_{N+1/2}=B$. We define $j_{\mathrm{f}}: = \lceil
H/\Delta z \rceil$, which is equal to the smallest integer larger than or equal
to $H/\Delta z$. Then the feed inlet ($z=0$) is located in layer~$j_{\mathrm{f}}$ (henceforth, the ``feed layer''). We add a layer to both the top and bottom to extract the
correct effluent and underflow concentrations, respectively. 

The average values of the unknowns in each layer $j$ are denoted by
$\smash{\bC_{j}=\bC_{j}(t)}$, $\smash{\bS_j=\bS_j(t)}$, and $\smash{W_j=W_j(t)}$.
The unknown output functions at the effluent and underflow are defined as $\bC_\mathrm{e}(t):=\bC_{0}(t)$, $\bC_\mathrm{u}(t):=\bC_{N+1}(t)$, etc. 
Two outer variables appear in the formulas for the numerical scheme; however, their values are irrelevant, so we may set $\bC_{-1}:= \bzero$, $\bC_{N+2}:=\bzero$, and analogously for other variables. 
The computational domain is given by $N+2$ intervals and one needs to define numerical fluxes for $N+3$ layer boundaries.

To approximate the cross-sectional area and the corresponding cell volumes we define
\begin{align*}
 A_{j+1/2}  &:= \dfrac{1}{\Delta z}\int_{z_{j-1}}^{z_j} A(\xi)\,{\rm d} \xi \quad \mbox{and}\quad  A_j := \dfrac{1}{\Delta z}\int_{z_{j-1/2}}^{z_{j+1/2}} A(\xi)\,{\rm d }\xi.
\end{align*}
In case $A$ is continuous one can use $A_{j+1/2}:=A(z_{j+1/2})$ as an alternative. 

The unknwons are approximated   by piecewise constant functions in each layer, i.e.,
\begin{align*}  C^{(k)}(z,t)=C_j^{(k)} \quad \text{for 
$z\in(z_{j-1/2},z_{j+1/2}]$.} 
\end{align*}
We let $\gamma_{j+1/2}:=\gamma(z_{j+1/2})$ and define the approximate volume average velocity $\qjph:=q(z_{j+1/2},t)$ in accordance with \eqref{eq:qdef2} with $Q_{\mathrm{reac}}\equiv 0$:
\begin{align*}
A_{j+1/2}\qjph &:= \Qu(t) - \gamma^{\; \rm f}_{j+1/2}Q_{\rm f}(t),
 \quad \text{where} \quad \gamma^{\; \rm f}_{j+1/2} := \chi_{\{j<j_{\rm f}\}}.
\end{align*}

\subsection{Numerical fluxes} \label{subsec:numflux} 

The flux $\bPhi_{\bC}$ given by~\eqref{eq:PhiC} is discretized over the cell boundary $z=z_{j+1/2}$ in an upwind 
or downwind  fashion depending on the sign of the total velocity~$\vfl_{X}$.
 The flux  $\Phi_{\bS}$ in~\eqref{eq:PhiS} is handled in a similar way depending on the sign of $\rho_X q-\vfl_X X$. 
  The   diffusion term is discretized in a  standard way.
We start by approximating the velocity $\vfl_X$, which contains three terms. The first term is straightforward; $\qjph(t):=q(z_{j+1/2},t)$, and for the third term we use central finite differences, i.e.
\begin{align*}
  J_{j+1/2}^{\bC} = J_{j+1/2}^{\bC}(X_j,X_{j+1}) &:= \frac{D_{\bC}(X_{j+1})-D_{\bC}(X_j)}{\Delta z}.
\end{align*}
For the numerical implementation of $D_{\bC}(X_j)$, we refer to~\cite{SDwatres3} (see Algorithm~2 and~3 therein).
For the second term in \eqref{eq:vlf_C},   $\vhs(X_{j+1})\approx\vhs(X(z_{j+1/2},t))$ is chosen with the following motivation.
When $q=0$, $D_{\bC}=0$ and there is only one component of~$\bC$, the flux~\eqref{eq:PhiC} is $\bPhi_{\bC}=A(z)\vfl_X X=A(z)\vhs(X)X$ and a working numerical flux that gives a monotone numerical scheme is $A(z)\vhs(X_{j+1})X_{j}$; see \cite{Burger&G&K&T2008b}.
Thus, the velocity~\eqref{eq:vlf_C} between cells~$j$ and~$j+1$ is approximated by
\begin{align*}
\vfl^{X}_{j+1/2}&=\vfl^{X}_{j+1/2}(X_{j},X_{j+1},t)
:=\qjph + \gamma_{j+1/2}\big(\vhs(X_{j+1}) - J_{j+1/2}^{\bC}\big).
\end{align*}
In the case $q=0$, $D_{\bC}=0$ and there is only one component of~$\bC$, our choice of upwind total flux would be $A(z)F^X_{j+1/2}$ where
\begin{align} \label{eq:fluxFX}
 F^X_{j+1/2}:= F^X_{j+1/2}(X_{j},X_{j+1},t) :=  (\vfl_X X)_{j+1/2} := \QCphm X_{j+1} + \QCphp X_{j}, 
\end{align}
where we use the notation $a^-:=\min\{a,0\}$ and $a^+:=\max\{a,0\}$. 
A key point in obtaining a working numerical scheme is that this flux is used in the approximation of the flux~$\Phi_{\bS}$ in~\eqref{eq:PhiS}.
Summarizing, we approximate  the fluxes of \eqref{finalmod}  by
\begin{align*}
\bPhi^{\bC}_{j+1/2}&
:= A_{j+1/2}\bigl(\QCphm\bC_{j+1} + \QCphp\bC_{j}\bigr),\\
\bPhi^{\bS}_{j+1/2}&
:= A_{j+1/2}\biggl(\dfrac{(\rho_Xq_{j+1/2}-F^X_{j+1/2})^-\bS_{j+1}}{\rho_X-X_{j+1}} + \dfrac{(\rho_Xq_{j+1/2}-F^X_{j+1/2})^+\bS_{j}}{\rho_X-X_{j}} -\gamma_{j+1/2}\mathbcal{D}\frac{\bS_{j+1}-\bS_{j}}{\Delta z}\biggr).
 \end{align*}
Note that the numerical  flux vector $\smash{\bPhi^{\bC}_{j+1/2}}$ is a function of $(\bC_j,\bC_{j+1},t)$ while  $\smash{\bPhi^{\bS}_{j+1/2}}$ depends on $(\bS_j,\bS_{j+1},X_j,X_{j+1},t)$. Moreover, the term 
 $\smash{F^X_{j+1/2}}$ in \eqref{eq:fluxFX}  results from summing  up the components of the vector~$\smash{\bPhi^{\bC}_{j+1/2}}$.
 
\subsection{Method of lines (MOL) formulation} \label{subsec:mol} 
We introduce the notation $\smash{[\Delta \bPhi]_j:=\bPhi_{j+1/2}-\bPhi_{j-1/2}}$ for the flux difference associated with cell~$j$ and let
$\smash{\delta_{j,j_{\mathrm{f}}}}$ denote the Kronecker delta, which is~1 if
$j=j_{\mathrm{f}}$ and zero otherwise. 
The conservation of mass for each layer,
corresponding to \eqref{eq:govCfinal}--\eqref{eq:govSfinal}, gives the following MOL equations (for $j=0,\ldots,N+1$):
\begin{align} \label{eq:MOL}
\begin{split}
\dd{\bC_j}{t}   &  = -\frac{[\Delta\bPhi^{\bC}]_j}{A_j\Delta z} +\delta_{j,j_{\mathrm{f}}}\frac{\bC_{\rm f}\Qf}{A_j\Delta z}
                           +\gamma_j\bR_{C,j},\\
\dd{\bS_j}{t}  &= -\frac{[\Delta\bPhi^{\bS}]_j}{A_j\Delta z} +\delta_{j,j_{\mathrm{f}}}\frac{\bS_{\rm f}\Qf}{A_j\Delta z} 
                           + \gamma_j\bR_{S,j},
\end{split}
\end{align}
The approximate water concentrations can be calculated  after the entire simulation via 
\begin{align*} 
W_j  & = {\rho_L} - rX_j  - \bigl(S_j^{(1)}+\cdots + S_j^{(k_{\bS})}\bigr).
\end{align*}

\subsection{Explicit fully discrete scheme} \label{subsec:scheme}

Let $t_n$, $n=0,1,\ldots, T$, denote the discrete time points and $\Delta t$
the time step that should satisfy a certain CFL condition depending on the
chosen time-integration method. For explicit 
schemes, the right-hand sides of equations~\eqref{eq:MOL} are evaluated at time $t_n$. 
The value of a variable at time $t_n$ is denoted by an upper index, e.g., $\smash{\bC_j^n}$.
The main restriction of the time step (for small $\Delta z$) is due to the
second-order spatial derivatives in the compression term (B\"urger et al.,
2005, 2012).
For explicit Euler, the time derivatives in \eqref{eq:MOL} are approximated by 
\begin{equation*}
\dd{\bC_j}{t}(t_n) \approx\frac{\bC_j^{n+1}-\bC_j^{n}}{\Delta t}.
\end{equation*}
We set
\begin{align*}
Q_\mathrm{f}^n:=\frac{1}{\Delta t}\int_{t_n}^{t_{n+1}}Q_\mathrm{f}(t)\,\rmd t
\end{align*}
and similarly for the time-dependent reaction terms.
Then we obtain the  explicit scheme 
\begin{subequations} \label{eq:numscheme}
\begin{align}
\bC_j^{n+1} & = \bC_j^n + \dfrac{\Delta t}{A_j\Delta z}\bigl(-[\Delta \bPhi^{\bC}]^n_j + \delta_{j,j_{\mathrm{f}}}\bC_{\rm f}^n Q_{\rm f}^n+\gamma_jA_j\Delta z\,\bR_{\bC,j}^n\bigr), \label{eq:timeC}\\
\bS_j^{n+1} & = \bS_j^n+\dfrac{\Delta t}{A_j\Delta z}\bigl(-[\Delta \bPhi^{\bS}]^n_j + \delta_{j,j_{\mathrm{f}}}\bS_{\rm f}^n Q_{\rm f}^n+\gamma_jA_j\Delta z\,\bR_{\bS,j}^n\bigr).\label{eq:timeS}
\end{align}
\end{subequations}

To establish some boundedness properties of Method~CS, we introduce the  CFL condition 
\begin{align} \label{cfl}\tag{CFL}
\Delta t  \max\{\beta_1,\beta_2\}
\leq 1,
\end{align}
where the $\beta$-values depend on $\Delta z$, $\Delta z^2$  and the constitutive functions by
\begin{align*}
\beta_1 &:= \frac{\|\Qf\|_{\infty,T}}{A_\mathrm{min}\Delta z}  + \frac{M_1}{\Delta z}
\big(
\|\vhs'\|_\infty\Xmax + \vhs(0)
\big) 
+ \frac{M_2}{\Delta z^2}
\big(
\|d_{\bC}\|_\infty\Xmax + D_{\bC}(\Xmax)
\big)
+ \max\{M_{\bC},\tilde{M}_{\bC}\}, \\
\beta_2 &:= 
\dfrac{\rho_X+\Xmax}{\rho_X-\Xmax} \frac{\|\Qf\|_{\infty,T}}{A_\mathrm{min}\Delta z} +\dfrac{\Xmax M_1}{\rho_X-\Xmax}\dfrac{\lVert \vhs \rVert_\infty}{\Delta z} + \dfrac{\Xmax M_2}{\rho_X-\Xmax}\dfrac{D_{\bC}(\Xmax)}{\Delta z^2} + \tilde{d}\dfrac{M_2}{\Delta z^2}+ M_{\bS},
\end{align*}
and the constants are given by
\begin{align*}
&M_{\bC}:=\sup_{\boldsymbol{U}\in\Omega, \atop 1\le k\le k_{\bC}}
 \left|\pp{R^{(k)}_{\bC}}{C^{(k)}} \right| , \qquad \tilde{M}_{\bC}:=\sup_{\boldsymbol{U}\in\Omega, \atop 1\le k\le k_{\bC}}
\left|\pp{\tilde{R}_{\bC}^{(k)}}{C^{(k)}}\right|, \qquad M_{\bS}:=\sup_{\boldsymbol{U} \in\Omega,  \atop
1\le k\le k_{\bS}} \left|\pp{R^{(k)}_{\bS}}{S^{(k)}} \right|, \\
&\|\xi\|_{\infty}:=\max\limits_{0\le X\le\Xmax}|\xi(X)|, \qquad\|Q\|_{\infty,T}:=\max_{0\le t\le T}\Qf(t),\quad  \tilde{d} = \max\{d^{(k)}:k = 1,\dots, k_{\bS} \},
\end{align*}
where $\xi$ represents $\vhs, \vhs'$ or $d_{\bC}$, and
\begin{align*}
  & M_1:= \underset{j=1,\dots,N}{\max}\left\{\dfrac{A_{j+1/2}}{A_j},\dfrac{A_{j-1/2}}{A_j}\right\},\quad
    M_2:= \underset{j=1,\dots,N}{\max}\left\{\dfrac{A_{j+1/2}+A_{j-1/2}}{A_j}\right\}.
\end{align*}

It is interesting to compare the eigenvalues of the flux Jacobian computed in Proposition~\ref{prop:eig} and the maximum speed given by condition~\eqref{cfl} in the case all diffusion and source terms are zero and if the area-dependent constant $M_1=1$ (corresponding to $A(z)\equiv$~constant):
\begin{align*}
\max\left|{\lambda_1}\right| 
&= \max\left|q + \gamma(z)\big(\vhs'(X)X+\vhs(X)\big)\right| 
\leq
\frac{\|\Qf\|_{\infty,T}}{A_\mathrm{min}} + \|\vhs'\|_\infty\Xmax + \vhs(0) \leq\beta_1\Delta z,\\
\max\left|{\lambda_2}\right|
&=\max\left|q-\frac{\gamma(z)f(X)}{\rho_X-X}\right|
\leq
\frac{\|\Qf\|_{\infty,T}}{A_\mathrm{min}} + \frac{\|\vhs\|_\infty\Xmax}{\rho_X-\Xmax}<\beta_2\Delta z.
\end{align*}
Note that the eigenvalues do not depend on $A(z)$, whereas the CFL condition for a numerical scheme may via $M_1$.

\subsection{Properties of the explicit numerical scheme} \label{subsec:propscheme} 

The aim is to show that the numerical solution stays in the set~$\Omega$, see~\eqref{eq:Omega}.
\begin{theorem}\label{thm}
If $\boldsymbol{U}_j^n:=(\bC_j^n,\bS_j^n,W_j^n)\in\Omega$ for all $j$, then under the condition~\eqref{cfl}, the scheme \eqref{eq:numscheme} implies $\boldsymbol{U}_j^{n+1}\in\Omega$ for all $j$.
\end{theorem}

We show this by proving that each scalar right-hand side of~\eqref{finalmod} is a monotone function of the concentrations in the cells~$j-1$, $j$ and~$j+1$. 
In the proofs below we use the estimate
\begin{align}\label{eq:JCbound}
 \bigl|J^{\bC}_{j+1/2} \bigr|
=\frac{\gamma_{j+1/2}}{\Delta z}\left|\int_{X_j}^{X_{j+1}}d_{\bC}(s)\,\rmd s\right|
\leq\frac{1}{\Delta z}\int_{X_{\mathrm{c}}}^{\Xmax}d_{\bC}(s)\,\rmd s
=\frac{D_{\bC}(\Xmax)}{\Delta z}.
\end{align}
It is convenient to define 
\begin{align*}
Q_{j+1/2}^n:=A_{j+1/2}q_{j+1/2}^n=
\begin{cases}
\Qu^n-\Qf^n  &\text{if $j<j_\mathrm{f}$,} \\
\Qu^n & \text{if $j\geq j_\mathrm{f}$.} 
\end{cases}
\end{align*}

\begin{lemma}\label{lem:Cjbound}
If $\smash{(\bC_j^n,\bS_j^n,W_j^n)\in\Omega}$ for all $j$ and \eqref{cfl} holds, then $\smash{0\leq\bC_j^{n+1}\leq\Xmax}$ for all~$j$.
\end{lemma}

\begin{proof}
We denote by~$\smash{\calH_{\bC}^{(k)}(\bC_{j-1}^{n},\bC_j^{n},\bC_{j+1}^{n})}$ the 
right-hand side of component~$k\in\{1,\ldots,k_{\bC}\}$ of~\eqref{eq:timeC}. 
We  show that $\smash{\calH_{\bC}^{(k)}}$ is a monotone function of each of its arguments by proving that
\begin{align}\label{eq:HCder}
\pp{\calH_{\bC}^{(k)}}{C_{j-1}^{(\ell),n}}\geq 0, \quad
\pp{\calH_{\bC}^{(k)}}{C_{j}^{(\ell),n}}\geq 0, \quad
\pp{\calH_{\bC}^{(k)}}{C_{j+1}^{(\ell),n}}\geq 0, \quad\ell=1,\ldots,k_{\bC}.
\end{align}
We  start with the most complicated case $\ell=k$.
The case $\ell\neq k$ will only have fewer terms in the estimations that will follow.
To avoid too heavy notation, we write~$\smash{C_j^{n}}$ instead of~$\smash{C_j^{(k),n}}$, etc.
With this convention, we first write out the following expression (of component $k$) of~\eqref{eq:timeC}:
\begin{align*}
[\Delta\Phi^{\bC}]^{n}_j & = \Delta\Phi_{C,j+1/2}^{n} - \Delta\Phi_{C,j-1/2}^{n}\\
& = A_{j+1/2}\bigl(\QCphnm{C_{j+1}^{n}} + \QCphnp{C_{j}^{n}}\bigr) - A_{j-1/2}\bigl(\QCmhnm{C_{j}^{n}} + \QCmhnp{C_{j-1}^{n}}\bigr)\\
&= A_{j+1/2}\QCphnm{C_{j+1}^{n}} + \bigl(A_{j+1/2}\QCphnp - A_{j-1/2}\QCmhnm\bigr){C_{j}^{n}} - A_{j-1/2}\QCmhnp{C_{j-1}^{n}}.
\end{align*}
We   use the shorter notation \begin{align}\label{eq:mp}
m_{\bC,j+1/2}^n&:=\chi_{\{\QCphn \leq 0\}} (A\gamma)_{j+1/2},\qquad
p_{\bC,j+1/2}^n:= \chi_{\{ \QCphn  \geq 0 \}} (A\gamma)_{j+1/2},
\end{align}
so that $\smash{m_{\bC,j+1/2}^n+p_{\bC,j+1/2}^n=(A\gamma)_{j+1/2}}$.
We calculate
\begin{align*}
A_{j+1/2}\pp{\QCphnm}{C_{j}^{n}} 
& = A_{j+1/2}\pp{}{C_{j}^{n}}\min \bigl\{\QCphn,0 \bigr\} 
= A_{j+1/2}\chi_{ \{   \QCphn \leq 0 \} } \pp{\QCphn}{C_{j}^{n}} \\
& = m_{\bC,j+1/2}^n\biggl( - \pp{J^{\bC}_{j+1/2}}{X_{j}^{n}}\biggr)
 = m_{\bC,j+1/2}^n\frac{D_{\bC}'(X_j^{n})}{\Delta z}
= m_{\bC,j+1/2}^n\frac{d_{\bC}(X_j^{n})}{\Delta z}\geq 0,  \\
A_{j+1/2}\pp{\QCphnm}{C_{j+1}^{n}}
& =  \chi_{\{ \QCphn  \leq 0\}} A_{j+1/2} \gamma_{j+1/2}\biggl(\vhs'(X_{j+1}^{n}) - \pp{J^{\bC}_{j+1/2}}{X_{j+1}^{n}}\biggr)\\
& = m_{\bC,j+1/2}^n
\biggl(\vhs'(X_{j+1}^{n})-\frac{d_{\bC}(X_{j+1}^{n})}{\Delta z}\biggr)\leq 0.
\end{align*}
Similarly, we get
\begin{align*}
A_{j+1/2}\pp{\QCphnp}{C_{j}^{n}} 
& = A_{j+1/2}\pp{}{C_{j}^{n}}\max \bigl\{\QCphn,0 \bigr\} 
 = p_{\bC,j+1/2}^n\frac{d_{\bC}(X_j^{n})}{\Delta z}\geq 0,\\
A_{j+1/2}\pp{\QCphnp}{C_{j+1}^{n}}
& = p_{\bC,j+1/2}^n
\biggl(\vhs'(X_{j+1}^{n}) - \frac{d_{\bC}(X_{j+1}^{n})}{\Delta z}\biggr)\leq 0.
\end{align*}
Now we differentiate $\calH_{\bC}$ to obtain
\begin{align*}
\pp{\mathcalold{H}_{\bC}}{C_{j-1}^{n}} 
& = -\dfrac{\Delta t}{A_j\Delta z}\pp{[\Delta\Phi^{\bC}]^{n}_j}{C_{j-1}^{n}} = -\dfrac{A_{j-1/2} \Delta t }{A_j\Delta z}
\biggl(
- \pp{\QCmhnm}{C_{j-1}^{n}}C_j^n - \pp{\QCmhnp}{C_{j-1}^{n}}C_{j-1}^{n} - \QCmhnp
\biggr)\geq 0,\\
\pp{\mathcalold{H}_{\bC}}{C_{j+1}^{n}} 
& = -\dfrac{\Delta t}{A_j\Delta z}\pp{[\Delta\Phi^{\bC}]^{n}_j}{C_{j+1}^{n}} = -\dfrac{A_{j+1/2} \Delta t}{A_j\Delta z}
\biggl(
\pp{\QCphnm}{C_{j+1}^{n}}C_{j+1}^{n} + \QCphnm + \pp{\QCphnp}{C_{j+1}^{n}}C_j^n
\biggr)\geq 0.
\end{align*}
With the help of the signs of the derivatives above, we estimate
\begin{align*}
\pp{[\Delta\Phi^{\bC}]^{n}_j}{C_{j}^{n}}
 &=  A_{j+1/2}\pp{\QCphnm}{C_{j}^{n}}{C_{j+1}^{n}} + 
\biggl(A_{j+1/2}\pp{\QCphnp}{C_{j}^{n}} - A_{j-1/2}\pp{\QCmhnm}{C_{j}^{n}}\biggr){C_{j}^{n}} \\
& \qquad + A_{j+1/2}\QCphnp - A_{j-1/2}\QCmhnm - A_{j-1/2}\pp{\QCmhnp}{C_{j}^{n}}{C_{j-1}^{n}}\\
& \leq 
\biggl(
 A_{j+1/2}\pp{\QCphnm}{C_{j}^{n}} +  A_{j+1/2}\pp{\QCphnp}{C_{j}^{n}} -
 A_{j-1/2}\pp{\QCmhnm}{C_{j}^{n}} -  A_{j-1/2}\pp{\QCmhnp}{C_{j}^{n}}
\biggr)\Xmax \\
&\qquad +  A_{j+1/2}\QCphnp -  A_{j-1/2}\QCmhnm\\
&=:\mathcalold{T}_1\Xmax + \mathcalold{T}_2,
\end{align*}
where we estimate
\begin{align*}
\mathcalold{T}_1 &= (A\gamma)_{j+1/2}\frac{d_{\bC}(X_j^{n})}{\Delta z} 
-(A\gamma)_{j-1/2}\left(\vhs'(X_{j}^{n}) - \frac{d_{\bC}(X_{j}^{n})}{\Delta z}
\right)\\
&\leq
A_j\left(
M_1\|\vhs'\|_\infty + M_2\frac{d_{\bC}(X_j^{n})}{\Delta z} 
\right)
\leq
A_j\left(
M_1\|\vhs'\|_\infty  + M_2\frac{\|d_{\bC}\|_\infty}{\Delta z}\right).
\end{align*}
For the term $\mathcalold{T}_2$, we use that $-a^-=(-a)^+$ and $(a+b)^+\leq a^++b^+$, so that $-(a+b)^-\leq -a^--b^-$, and~\eqref{eq:JCbound} to obtain
\begin{align*}
\mathcalold{T}_2& =  A_{j+1/2} \QCphnp -  A_{j-1/2}\QCmhnm\\
&\leq 
A_{j+1/2}\bigl(q_{j+1/2}^{n,+} + \gamma_{j+1/2}\bigl(\vhs(X_{j+1}^{n}) + \big(-J^{C,n}_{j+1/2}\big)^+\bigr)\bigr) 
- A_{j-1/2}\bigl(q_{j-1/2}^{n,-} 
+ (\gamma)_{j-1/2}\bigl(-J^{C,n}_{j-1/2}\bigr)^-\bigr)\\
& \leq
Q_{j+1/2}^{n,+} - Q_{j-1/2}^{n,-} + (A\gamma)_{j+1/2}\bigl(\vhs(0) + \big(-J^{C,n}_{j+1/2}\big)^+\bigr) + (A\gamma)_{j-1/2}J^{C,n,+}_{j-1/2} \\
&\leq
\Qu^n+\Qe^n + A_{j+1/2}\vhs(0) 
+ (A_{j+1/2} + A_{j-1/2})\frac{D_{\bC}(\Xmax)}{\Delta z} \\
&\leq 
A_j\biggl(
\frac{\|\Qf\|_{\infty,T}}{A_\mathrm{min}} +
M_1\vhs(0) + M_2\frac{D_{\bC}(\Xmax)}{\Delta z} 
\biggr).
\end{align*}
The condition~\eqref{cfl} now implies 
\begin{align*}
\pp{\mathcalold{H}_{\bC}}{C_j^{n}} 
& = 1-\dfrac{\Delta t}{A_j\Delta z}\pp{[\Delta\Phi^{\bC}]^{n}_j}{C_{j}^{n}} +\Delta t\,\gamma_j\pp{R_{\bC,j}^{n}}{C_{j}^{n}}\\
& \geq 1-\Delta t\biggl(\biggl(\frac{M_1\|\vhs'\|_\infty}{\Delta z}  + M_2\frac{\|d_{\bC}\|_\infty}{\Delta z^2}\biggr)\Xmax \\
& \qquad + \frac{\|\Qf\|_{\infty,T}}{A_\mathrm{min}\Delta z} +
\frac{M_1\vhs(0)}{\Delta z}
+ M_2\frac{D_{\bC}(\Xmax)}{\Delta z^2}  + {M}_{\bC}\biggr)\geq 0.
\end{align*}
The derivatives~\eqref{eq:HCder} in the case $\ell\neq k$ are obtained as above; however, with $\mathcalold{T}_2\equiv 0$.
For a given vector~${\bC_j^n}$ with ${X_j^n=C_j^{(1),n}+\cdots+C_j^{(k_{\bC}),n}\leq\Xmax}$  we let ${\bar{\bC}_j^n}$ denote any vector that satisfies ${C_j^{(k),n}\leq\bar{C}_j^{(k),n}}$, $k=1,\ldots,k_{\bC}$, and ${\bar{C}_j^{(1),n}+\cdots +\bar{C}_j^{(k_{\bC}),n}}=\Xmax$.
The monotonicity in each variable of ${\calH_{\bC}^{(k)}}$ and the assumptions~\eqref{eq:assumptionRC} and \eqref{eq:tech} are now used to obtain, for $j\neq j_{\rm f}$, 
\begin{align*}
 0 =   \mathcalold{H}_{\bC}^{(k)}(\bzero,\bzero,\bzero) \leq C_j^{n+1} 
   =   \mathcalold{H}_{\bC}^{(k)} \bigl(\bC_{j-1}^n,\bC_j^n,\bC_{j+1}^n \bigr) 
   \le \mathcalold{H}_{\bC}^{(k)} \bigl(\bar{\bC}_{j-1}^n,\bar{\bC}_j^n,\bar{\bC}_{j+1}^n \bigr) =\Xmax,
\end{align*}
and for the case when $j=\jf$, we have
\begin{align*}
0&\le\dfrac{\Delta t}{A_{j_\mathrm{f}}\Delta z}\,{C_\mathrm{f}\Qf} = \calH_{\bC}^{(k)}(\bzero,\bzero,\bzero)
\le C_j^{(k),n+1}
=\calH_{\bC}^{(k)} \bigl(\bC_{j_\mathrm{f}-1}^{n},\bC_{j_\mathrm{f}}^{n},
\bC_{j_\mathrm{f}+1}^{n} \bigr)\\
&\le\calH_{\bC}^{(k)}
\bigl(\bar{\bC}_{j-1}^n,\bar{\bC}_j^n,\bar{\bC}_{j+1}^n \bigr)
=\Xmax-\dfrac{\Delta t}{A_{j_\mathrm{f}}\Delta z}\big(\Qu\Xmax-(\Qu-\Qf)\Xmax-{C_\mathrm{f}\Qf}\big)\\
&=\Xmax-\dfrac{\Delta t}{A_{j_\mathrm{f}} \Delta z}\Qf(\Xmax-C_\mathrm{f})\le\Xmax,
\end{align*}
which proves the bound of~$\smash{C_j^{(k),n}}$.
\end{proof}

\begin{lemma}\label{lem:Xjbound}
If $\smash{(\bC_j^n,\bS_j^n,W_j^n)\in\Omega}$ for all~$j$ and \eqref{cfl} holds, then $\smash{0\leq X_j^{n+1}\leq\Xmax}$ for all~$j$.
\end{lemma}

\begin{proof}
Summing all components of~\eqref{eq:timeC}  yields the   update formula 
\begin{align}\label{eq:temp}
 X_j^{n+1} & = X_j^n + \dfrac{\Delta t}{A_j\Delta z}\bigl(-[\Delta\Psi]^{n}_j + \delta_{j,j_{\mathrm{f}}}X_{\rm f}^{n} Q_{\rm f}^n+ \gamma_jA_j\Delta z \tilde{R}_{\bC,j}^{n}\bigr)
\end{align}
or the total solids concentration~$X$, where
\begin{align*}
\Psi^{n}_j:=A_{j+1/2}\QCphm X_{j+1} + A_{j+1/2}\QCphp X_{j}.
\end{align*}
Since~\eqref{eq:temp} is similar to one component of~\eqref{eq:timeC}, this lemma can be proved 
 by following  the proof of Lemma~\ref{lem:Cjbound} with $\smash{C_j^{(k),n}}$ replaced by $\smash{X_j^n}$, 
  $\smash{{R}_{\bC,j}^{(k),n}}$ replaced by $\smash{\tilde{R}_{\bC,j}^{(k),n}}$, and hence, $\smash{M_{\bC}}$ replaced by~$\smash{\tilde{M}_{\bC}}$.
\end{proof}

\begin{lemma}\label{lem:Sjbound}
If $\smash{(\bC_j^n,\bS_j^n,W_j^n)\in\Omega}$ for all $j$ and \eqref{cfl} holds, then 
\begin{align*} 
 S_j^{(1),n+1} \geq 0 , \dots, S_j^{(k_{\boldsymbol{S}}),n+1} \geq 0 \quad \text{for all $j$}.
 \end{align*} 
\end{lemma}

\begin{proof} Let us denote by~$\smash{\calH_{\bS}^{(k)}(\bS_{j-1}^{n},\bS_j^{n},\bS_{j+1}^{n})}$ component $k\in\{1,\ldots,k_{\bS}\}$ of the right-hand side of~\eqref{eq:timeS}.
To show that $\smash{\calH_{\bS}^{(k)}}$ is a monotone function of  each scalar argument we prove  
\begin{align*}
&\pp{\calH_{\bS}^{(k)}}{S_{i}^{(\ell),n}}\geq 0, \quad i=j-1,j,j+1,\quad \ell=1,\ldots,k_{\bS}.
\end{align*}
We start with $\ell=k$, 
  do not write out the superscript $(k)$ and define $\smash{m_{\bS,j+1/2}^n}$ and $\smash{p_{\bS,j+1/2}^n}$ in analogy with~\eqref{eq:mp}.
 We introduce $\smash{\tilde{X}^n_j :=  \rho_X - X^n_j}>0$ and the flux
\begin{align*}
F^{L,n}_{j+1/2}(X_j,X_{j+1}):= \rho_Xq_{j+1/2}-F^{X,n}_{j+1/2}.
\end{align*}
Component $k$ of~\eqref{eq:timeS} contains the  expression 
\begin{align*}
[\Delta\Phi^{\bS}]^{n}_j &= A_{j+1/2}F^{L,n,-}_{j+1/2}\dfrac{S_{j+1}^{n}}{\tilde{X}^n_{j+1}} + \bigl(A_{j+1/2}F^{L,n,+}_{j+1/2} - A_{j-1/2}F^{L,n,-}_{j-1/2} \bigr)\dfrac{S_{j}^{n}}{\tilde{X}^n_j} - A_{j-1/2}F^{L,n,+}_{j-1/2}\dfrac{S_{j-1}^{n}}{\tilde{X}^n_{j-1}}\\
&\qquad - d\biggl( (A\gamma)_{j+1/2}\dfrac{S^n_{j+1}-S^n_j}{\Delta z} - (A\gamma)_{j-1/2}\dfrac{S^n_{j}-S^n_{j-1}}{\Delta z}\biggr). 
\end{align*}
Since $\smash{F^{L,n}_{j+1/2}}$ does not depend on $\bS$, we obtain
\begin{align*}
\pp{\mathcalold{H}_{\bS}}{S_{j-1}^{n}} 
& = -\dfrac{\Delta t}{A_j\Delta z}\pp{[\Delta\Phi^{\bS}]^{n}_j}{S_{j-1}^{n}} = \dfrac{A_{j-1/2}\Delta t}{A_j\Delta z} \dfrac{F^{L,n,+}_{j-1/2}}{\tilde{X}^n_{j-1}} 
+d\dfrac{\Delta t (A\gamma)_{j-1/2}}{A_j\Delta z^2} 
\geq 0,\\
\pp{\mathcalold{H}_{\bS}}{S_{j+1}^{n}} 
& = -\dfrac{\Delta t}{A_j\Delta z}\pp{[\Delta\Phi^{\bS}]^{n}_j}{S_{j+1}^{n}} = -\dfrac{A_{j+1/2}\Delta t}{A_j\Delta z} \dfrac{F^{L,n,-}_{j+1/2}}{\tilde{X}^n_{j+1}} 
+d\dfrac{\Delta t (A\gamma)_{j+1/2}}{A_j\Delta z^2} 
\geq 0.
\end{align*}
Now we estimate the following, using $(a+b)^+\leq a^++b^+$, $-(a+b)^-\leq (-a)^++(-b)^+$:
\begin{align*}
\dfrac{\partial[\Delta\Phi^{\bS}]^{n}_j}{\partial S_j^n} &= A_{j+1/2}\dfrac{F^{L,n,+}_{j+1/2}}{\tilde{X}^n_j} - A_{j-1/2}\dfrac{F^{L,n,-}_{j-1/2}}{\tilde{X}^n_j} +
d\dfrac{(A\gamma)_{j+1/2}+(A\gamma)_{j-1/2}}{\Delta z}\\
&\leq \dfrac{A_{j+1/2}\rho_X}{\tilde{X}_j^n}q^{n,+}_{j+1/2} +  \dfrac{A_{j-1/2}\rho_X}{\tilde{X}_j^n}(-q^n_{j-1/2})^+ 
+ \dfrac{A_{j+1/2}X_{j+1}^n}{\tilde{X}_j^n}\bigl( (-q^n_{j+1/2})^+ + \gamma_{j+1/2}J_{j+1/2}^{C,n,+}\bigr)\\
&\qquad + \dfrac{A_{j-1/2}X_{j}^n}{\tilde{X}_j^n}\bigl( q^{n,+}_{j-1/2}+ \gamma_{j-1/2}(\vhs(X^n_j))^++\gamma_{j-1/2}(-J_{j+1/2}^{C,n})^+\bigr)
+d A_j \dfrac{M_2}{\Delta z}\\
&\leq \dfrac{\rho_X}{\rho_X-\Xmax}\bigl(Q^{n,+}_{j+1/2}+(-Q^{n}_{j-1/2})^+\bigr) 
+\dfrac{\Xmax}{\rho_X-\Xmax}\bigl((-Q^n_{j+1/2})^+ +  Q^{n,+}_{j-1/2}\bigr) \\
&\qquad + A_j\left(\dfrac{\Xmax M_1}{\rho_X-\Xmax}\lVert \vhs \rVert_\infty + \dfrac{\Xmax M_2}{\rho_X-\Xmax}\dfrac{D_{\bC}(\Xmax)}{\Delta z} + d\dfrac{M_2}{\Delta z}\right)\\
&\leq A_j\left(\dfrac{\rho_X+\Xmax}{\rho_X-\Xmax} \frac{\|\Qf\|_{\infty,T}}{A_\mathrm{min}} +\dfrac{\Xmax M_1}{\rho_X-\Xmax}\lVert \vhs \rVert_\infty + \dfrac{\Xmax M_2}{\rho_X-\Xmax}\dfrac{D_{\bC}(\Xmax)}{\Delta z} + d\dfrac{M_2}{\Delta z}\right)
\end{align*}
The condition~\eqref{cfl} now implies 
\begin{align*}
\pp{\mathcalold{H}_{\bS}}{S_j^{n}} 
& = 1-\dfrac{\Delta t}{A_j\Delta z}\pp{[\Delta\Phi^{\bS}]^{n}_j}{S_{j}^{n}} +\Delta t\,\gamma_j\pp{R_{S,j}^{n}}{S_{j}^{n}}\\
& \geq  1- \Delta t \bigg( \frac{(\rho_X+\Xmax)\|\Qf\|_{\infty,T}}{(\rho_X-\Xmax)A_\mathrm{min}\Delta z} +\dfrac{\Xmax M_1\lVert \vhs \rVert_\infty}{(\rho_X-\Xmax)\Delta z} + \dfrac{\Xmax M_2D_{\bC}(\Xmax)}{(\rho_X-\Xmax)\Delta z^2} + \tilde{d}\dfrac{M_2}{\Delta z^2}+ M_{\bS}\bigg)\geq 0,
\end{align*}
where $\smash{\tilde{d} := \max\{d^{(1)},\dots, d^{(k_{\bS})} \}}$. 
Sine $\smash{\calH_{\bS}^{(k)}}$ is monotone 
 in each variable, it follows for $j\neq j_{\rm f}$ that  
\begin{align*}
0 = \mathcalold{H}^{(k)}_{\bS}(\bzero,\bzero,\bzero) \leq 
   \mathcalold{H}^{(k)}_{\bS} \bigl(\bS_{j-1}^n,\bS_j^n,\bS_{j+1}^n \bigr) = S_j^{n+1},
\end{align*}
and for the case $j=\jf$, we have
\begin{align*}
0&\le\dfrac{\Delta t}{A_{j_\mathrm{f}}\Delta z}\,{S_\mathrm{f}\Qf} = \calH^{(k)}_{\bS}(\bzero,\bzero,\bzero)
\le \calH^{(k)}_{\bS} \bigl(\bS_{j_\mathrm{f}-1}^{n},\bS_{j_\mathrm{f}}^{n}, \bS_{j_\mathrm{f}+1}^{n} \bigr) = S_{j_\mathrm{f}}^{n+1}.
\end{align*}
\end{proof}

\section{Numerical examples} \label{sec:num} 

We use the same model for denitrification as \cite{SDcace_reactive} with two solid components: ordinary heterotrophic organisms $X_{\rm OHO}$ and undegradable organics $X_{\rm U}$; and three soluble components: nitrate $S_{\rm NO_3}$, readily biodegradable substrate $S_{\rm S}$ and nitrogen $S_{\rm N_2}$, then the simulated variables are 
\begin{align*}
\bC = (
             X_{\rm OHO}, 
             X_{\rm U})^{\mathrm{T}} 
             \quad (k_{\boldsymbol{C}} =2), \quad 
\bS  =  (
             S_{\rm NO_3}, 
             S_{\rm S} , 
             S_{\rm N_2} 
              )^{\mathrm{T}}  \quad (k_{\boldsymbol{S}} =3). 
\end{align*}
The reaction terms for the solid and liquid phases used for all numerical examples are given by
\begin{align*}
 \bR_{\bC} = X_{\rm OHO}Z(X)\begin{pmatrix} 
             \mu(\bS)-b\\
            f_{\rm P}b
                   \end{pmatrix} 
                  ,\quad
\bR_{\bS} = X_{\rm OHO} 
          \begin{pmatrix}  
            - \bar{Y}\mu(\bS)\\
            (1-f_{\rm p})b-\mu(\bS)/Y\\
             \bar{Y} \mu(\bS)
                   \end{pmatrix} 
           ,\quad  \bar{Y} = \dfrac{1-Y}{2.86Y}
\end{align*}
where $Y = 0.67$ is a yield factor, $b = 6.94\times 10^{-6}\,\mathrm{s}^{-1}$ is the decay rate of heterotrophic organisms and $f_\mathrm{P} = 0.2$ is the portion of these that decays to undegradable organics.
The continuous function~$Z(X)$ should be equal to one for most concentration and satisfies $Z(\Xmax)=0$, so that the technical assumption~\eqref{eq:assumptionRC} is satisfied. 
The function~$Z(X)$ should not influence the condition~\eqref{cfl} and we have used $\Xmax=30\,\mathrm{kg}/\mathrm{m}^3$, a value our simulated solutions never reaches, despite we have simulated with $Z(X)\equiv 1$.
Moreover, 
\begin{align*}
 \mu(\bS) = \mumax \dfrac{S_{\rm NO_3}}{K_{\rm NO_3}+S_{\rm NO_3}} \dfrac{S_{\rm S}}{K_{\rm S}+S_{\rm S}}
\end{align*}
is the specific growth rate function 
with $\mumax = 5.56\times 10^{-5}\,\mathrm{s}^{-1}$, and saturation parameters $K_{\rm NO_3} = 5\times 10^{-4}\, \rm kg/m^3$ and $K_{\rm S} = 0.02\, \rm kg/m^3$. Adding the components of the reaction terms we get
\begin{align*}
 \tilde{R}_{\bC} = R_{\bC}^{(1)}+  R_{\bC}^{(2)}   = \bigl(\mu(\bS)-(1-f_{\rm P})b\bigr)X_{\rm OHO}Z(X), \quad 
   \tilde{R}_{\bS} =  R_{\bS}^{(1)} +  R_{\bS}^{(2)}  + R_{\bS}^{(3)} = R_{\bS}^{(2)}.
\end{align*}
The constitutive functions used in all simulations are 
\begin{align*}
\vhs(X) := \dfrac{v_0}{1 + (X/ \bar{X})^\eta}, \qquad
\sigma_{\mathrm{e}}(X)=\alpha \chi_{\{ X \geq X_{\mathrm{c}} \}} (X-X_{\mathrm{c}}),
\end{align*}
with the constants $v_0 = 1.76\times 10^{-3}\, \rm m/s$, $\bar{X} = 3.87\, \rm kg/m^3$, $\eta = 3.58$, $X_{\mathrm{c}} = 5\, \rm kg/m^3$ and $\alpha = 0.2\,\rm m^2/s^2$. 
Other parameters are $\rho_X = 1050\, \rm kg/m^3$, $\rho_L = 998\, \rm kg/m^3$ and $g = 9.81 \,\rm m/s^2$. The feed concentrations of soluble components in all examples are $\bSf  = (6.00\times 10^{-3}, 9.00\times 10^{-4},0)^{\rm T}\, \rm kg/m^3$, which 
 are chosen  constant with respect to time.

\begin{figure}[t]
 \begin{tabular}{ccc}
 \includegraphics[scale=0.4]{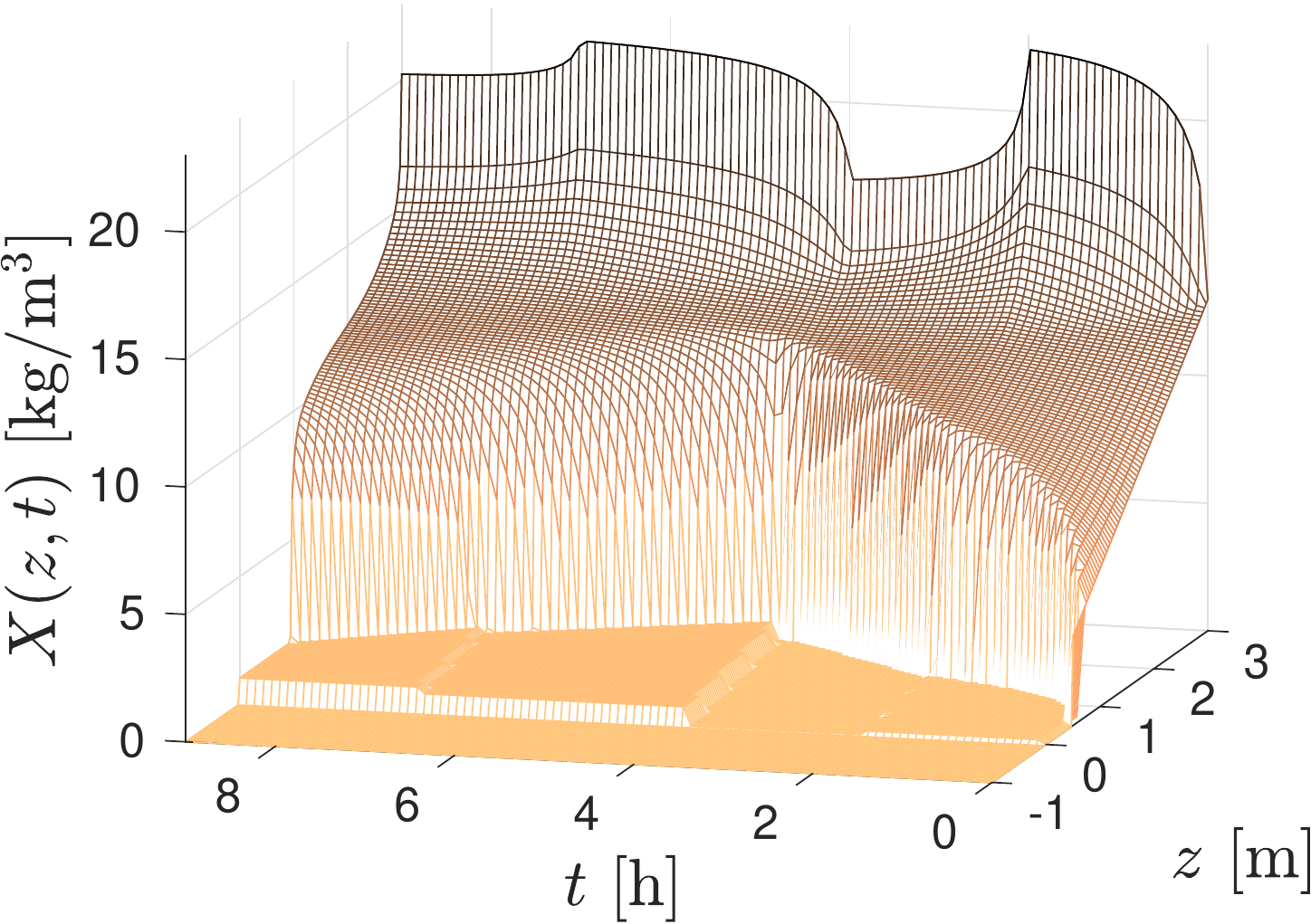} &
 \hspace{-1cm}\includegraphics[scale=0.4]{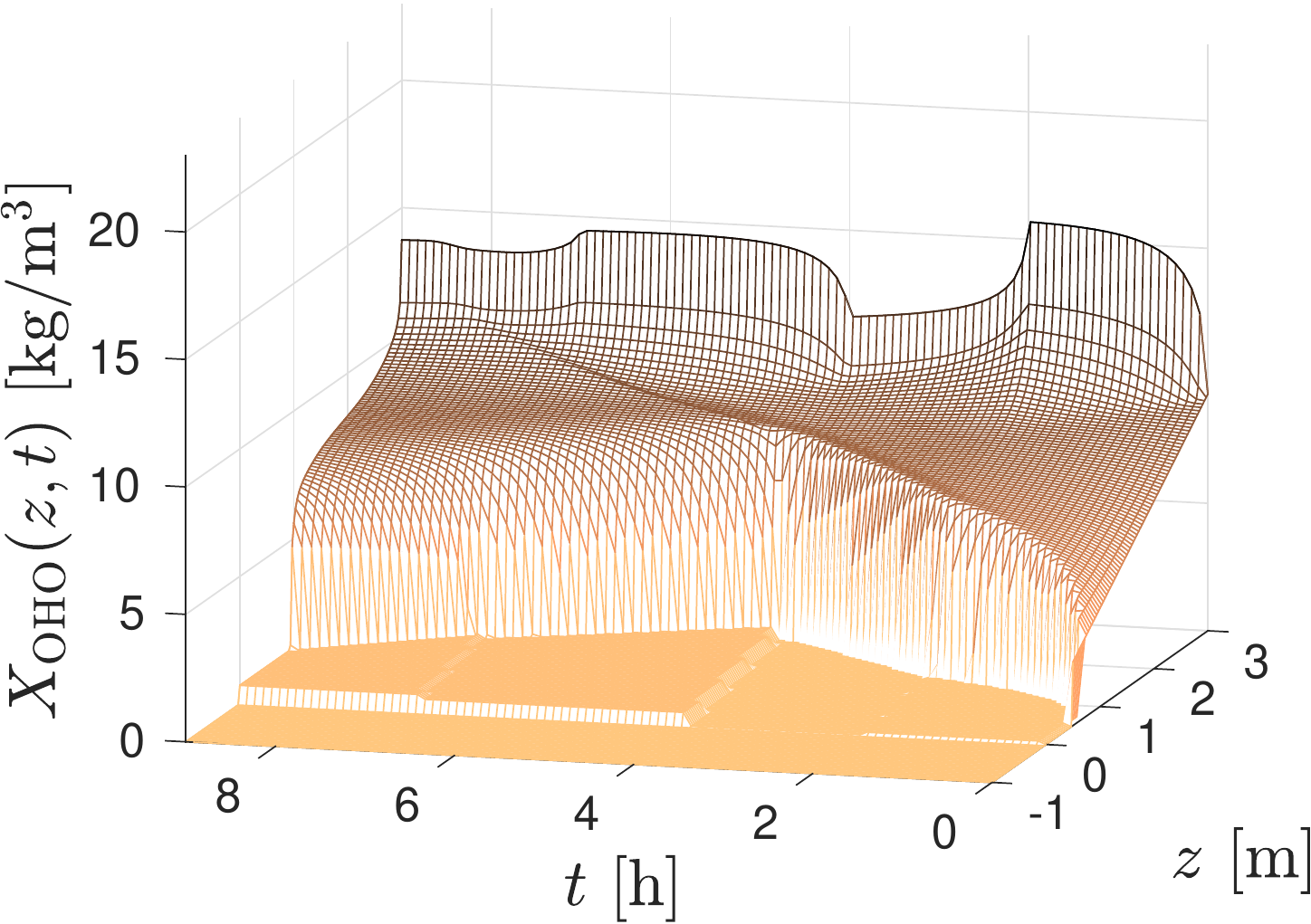} \\
 \includegraphics[scale=0.4]{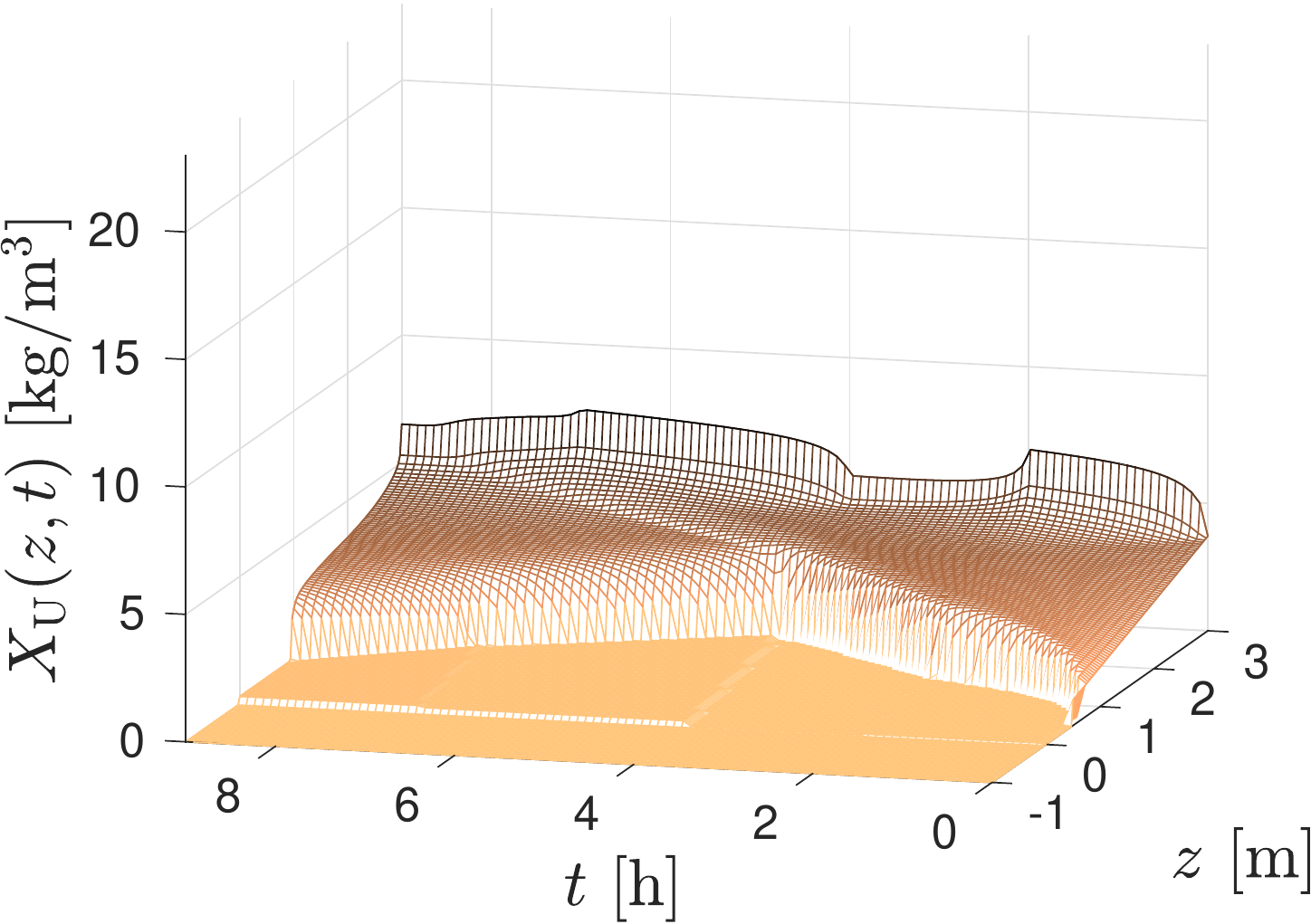} &
 \hspace{-1cm} \includegraphics[scale=0.4]{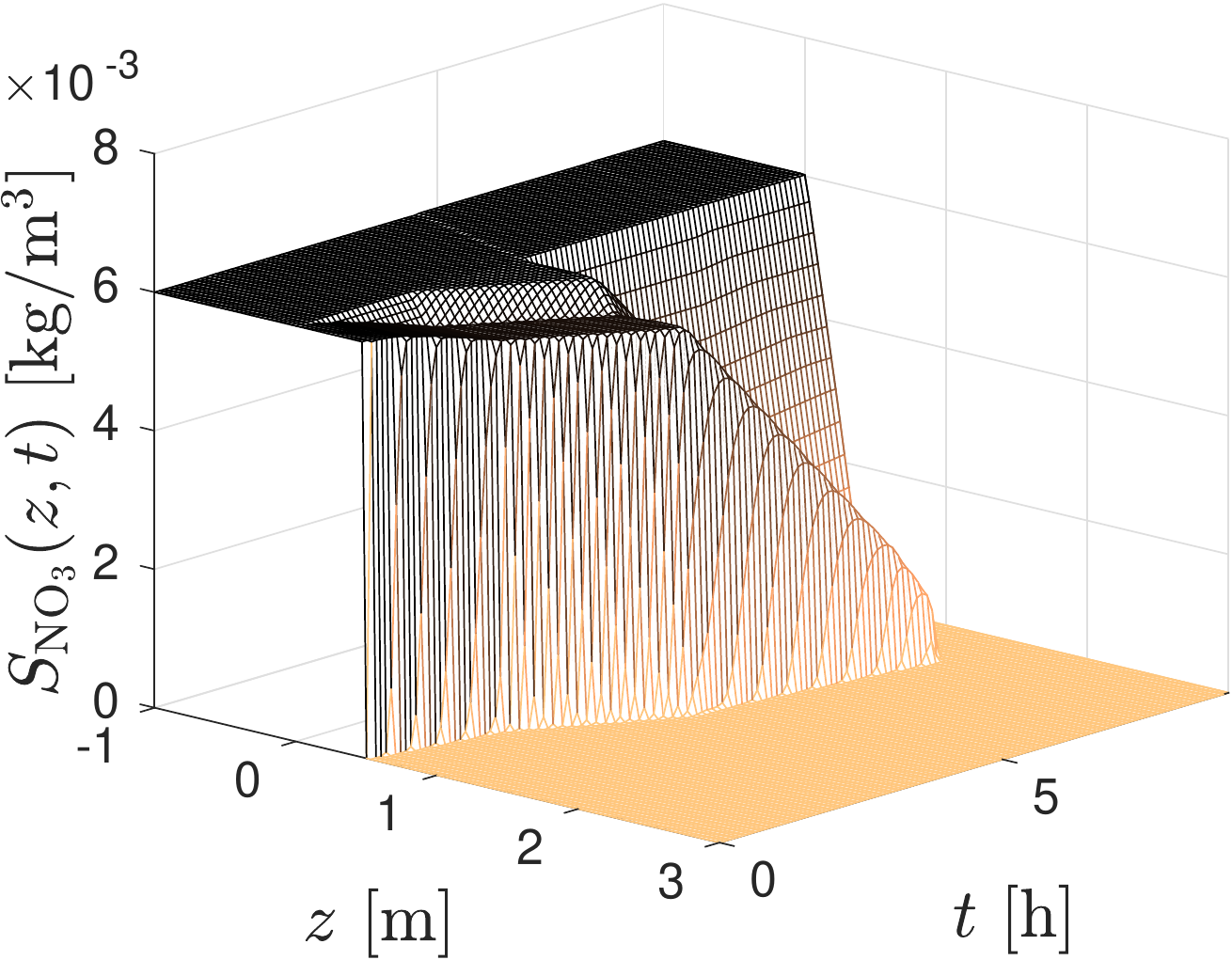} \\
 \includegraphics[scale=0.4]{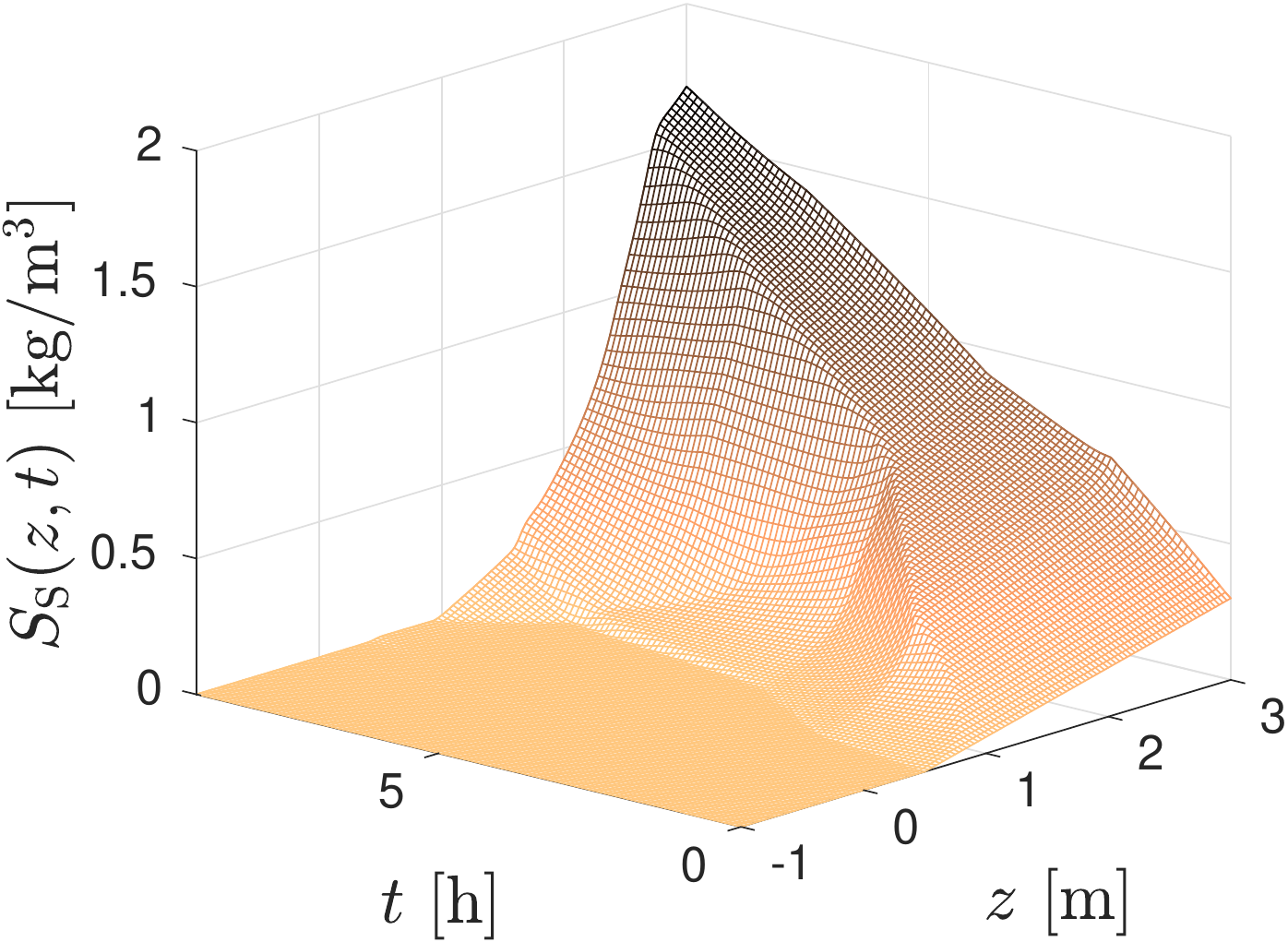} &
 \hspace{-1cm}\includegraphics[scale=0.4]{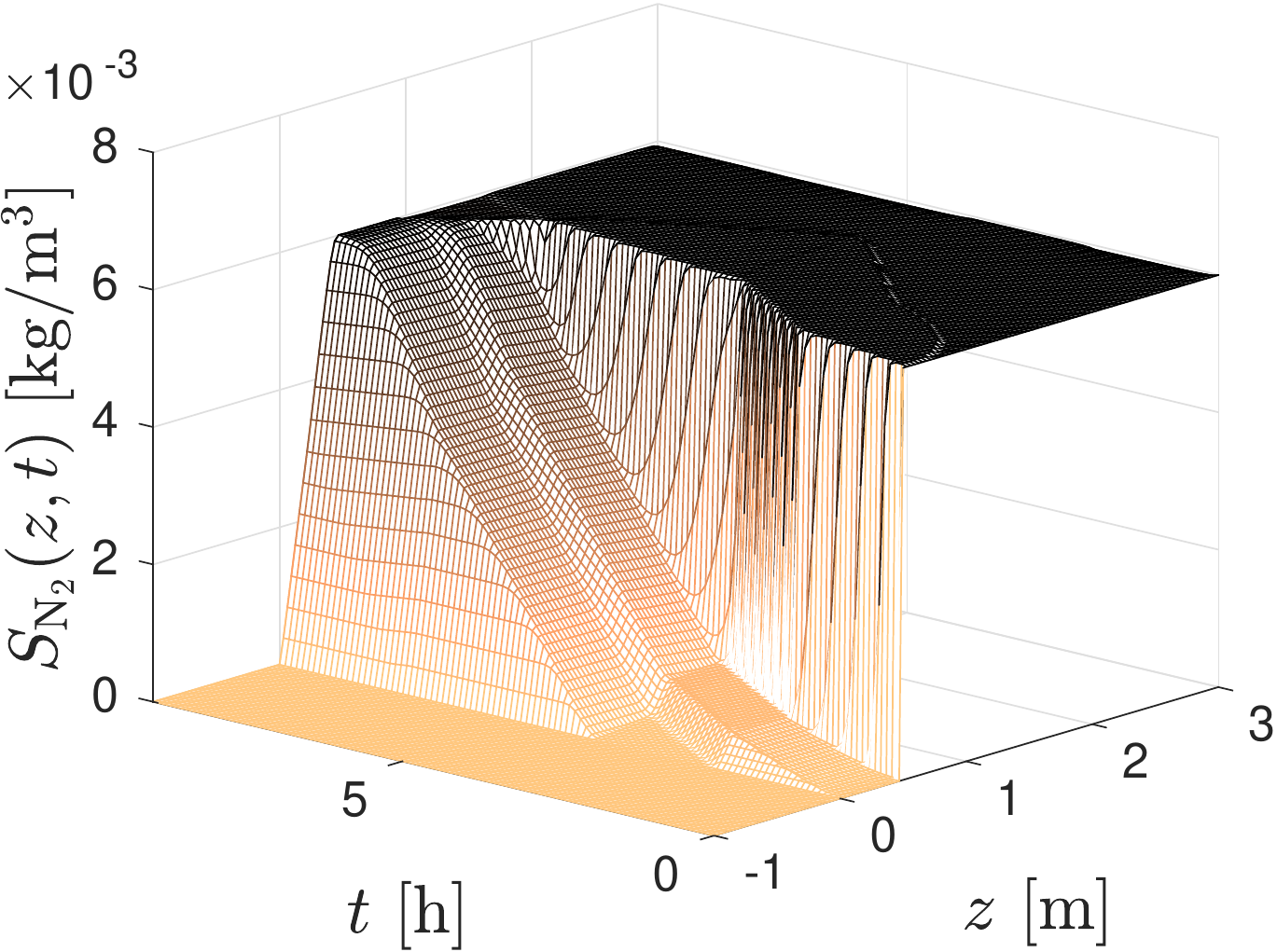}
 \end{tabular}
 \caption{Example 1: Reference solution  with $N = 4096$ and $T = 20\,\rm h$. The solution has been projected onto a coarse grid.} \label{fig:refsol}
\end{figure}%

 \begin{table}[t]
\caption{Example 1: Errors~$e_N^{\textrm{rel}}$ \eqref{eq:error}, approximate order of convergence~$\theta$ \eqref{eq:theta} and CPU times at simulated times $3,6$ and $9$ hours. The errors have been computed with the reference solution obtained by Method~CS with $N = 4096$.} \label{table:error}
\begin{center}
\begin{tabular}{|c|ccc|ccc|} \hline 
 & \multicolumn{3}{c|}{Method CS} & \multicolumn{3}{c|}{Method XP} \\
 \hline
 & \multicolumn{6}{c|}{$t = 3{\rm \, h}$}\\
 \hline
 $N$ & $e_N^{\textrm{rel}}(t)$ & $\theta(t)$ &  CPU $[\mathrm{s}]$ & $e_N^{\textrm{rel}}(t)$ & $\theta(t)$ &  CPU $[\mathrm{s}]$ \\ \hline 
16 &  0.7239 & --- & 0.2047 &  0.5868 & --- & 0.2577 \\ 
32 & 0.4042 & 0.8407 & 0.3675 & 0.3413 & 0.7819 & 0.4687 \\ 
64 & 0.2471 & 0.7100 & 0.6834 & 0.2086 & 0.7101 & 0.8867 \\ 
128 & 0.1487 & 0.7326 & 1.3370 & 0.1271 & 0.7154 & 1.7144 \\ 
256 & 0.0868 & 0.7763 & 2.6357 & 0.0747 & 0.7664 & 3.3940 \\ 
512 & 0.0481 & 0.8514 & 6.6872 & 0.0415 & 0.8462 & 6.9696 \\ 
\hline
 & \multicolumn{6}{c|}{$t = 6{\rm \,  h}$}\\
\hline
16 &  1.1278 & --- & 0.3939 &  0.8704 & --- & 0.5137 \\ 
32 & 0.6411 & 0.8149 & 0.7164 & 0.5116 & 0.7668 & 0.9316 \\ 
64 & 0.3840 & 0.7396 & 1.3411 & 0.3074 & 0.7347 & 1.7577 \\ 
128 & 0.2304 & 0.7369 & 2.6078 & 0.1843 & 0.7382 & 3.3995 \\ 
256 & 0.1319 & 0.8049 & 5.1365 & 0.1052 & 0.8087 & 6.6752 \\ 
512 & 0.0710 & 0.8934 & 12.9595 & 0.0563 & 0.9015 & 13.5693 \\ 
\hline
 & \multicolumn{6}{c|}{$t = 9{\rm \, h}$}\\
 \hline
16 &  0.8363 & --- & 0.5929 &  0.6182 & --- & 0.7721 \\ 
32 & 0.4675 & 0.8390 & 1.0663 & 0.3599 & 0.7803 & 1.3779 \\ 
64 & 0.2735 & 0.7738 & 2.0404 & 0.2056 & 0.8078 & 2.6182 \\ 
128 & 0.1535 & 0.8331 & 3.9169 & 0.1131 & 0.8626 & 5.0563 \\ 
256 & 0.0829 & 0.8895 & 7.7370 & 0.0593 & 0.9308 & 10.0362 \\ 
512 & 0.0425 & 0.9624 & 19.5205 & 0.0289 & 1.0350 & 20.4181 \\  
\hline
\end{tabular}
\end{center}
\end{table}%

\begin{figure}[t]
\begin{center}
 \begin{tabular}{cc}
                  \includegraphics[scale=0.44]{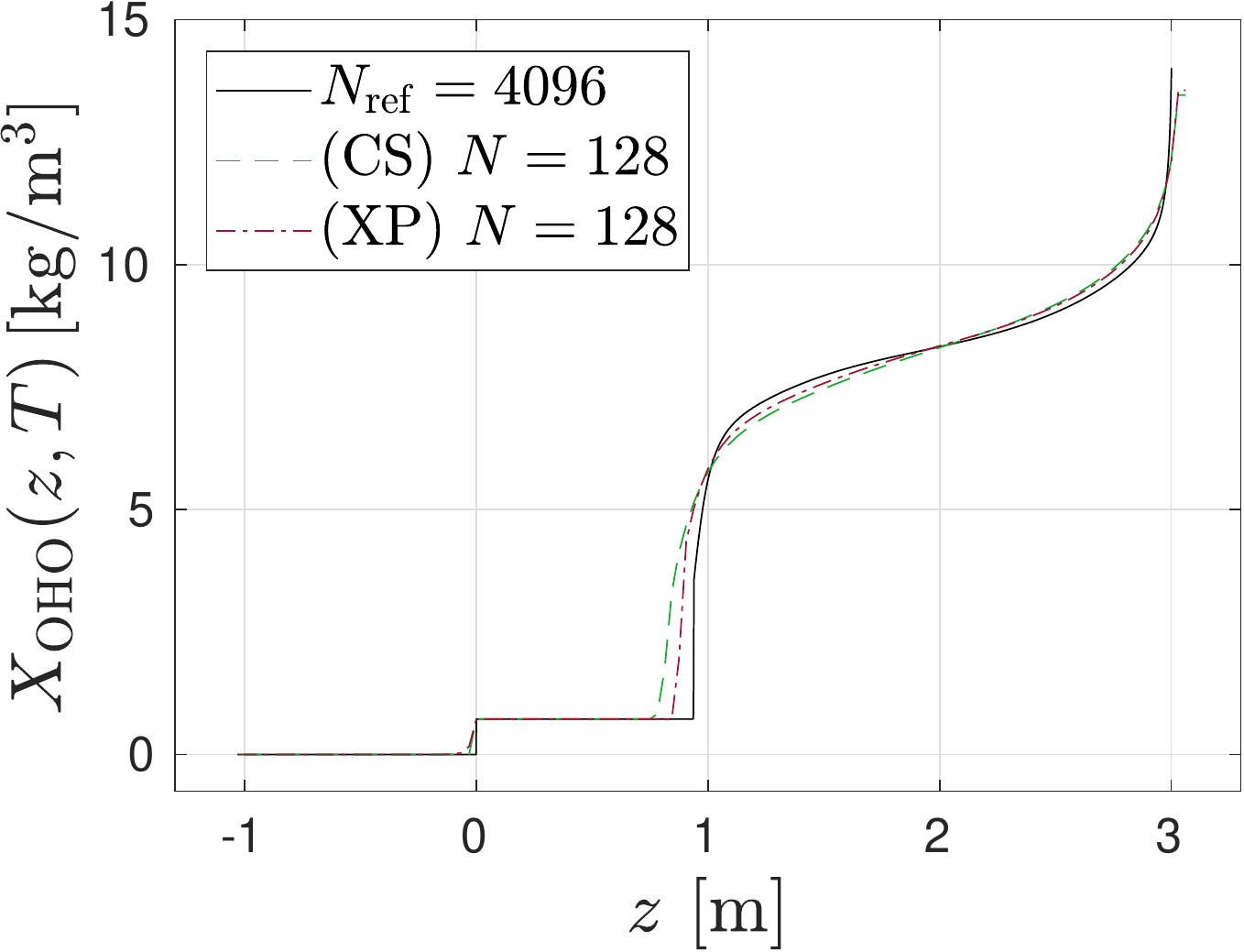} &
 \hspace{-0.9cm}  \includegraphics[scale=0.44]{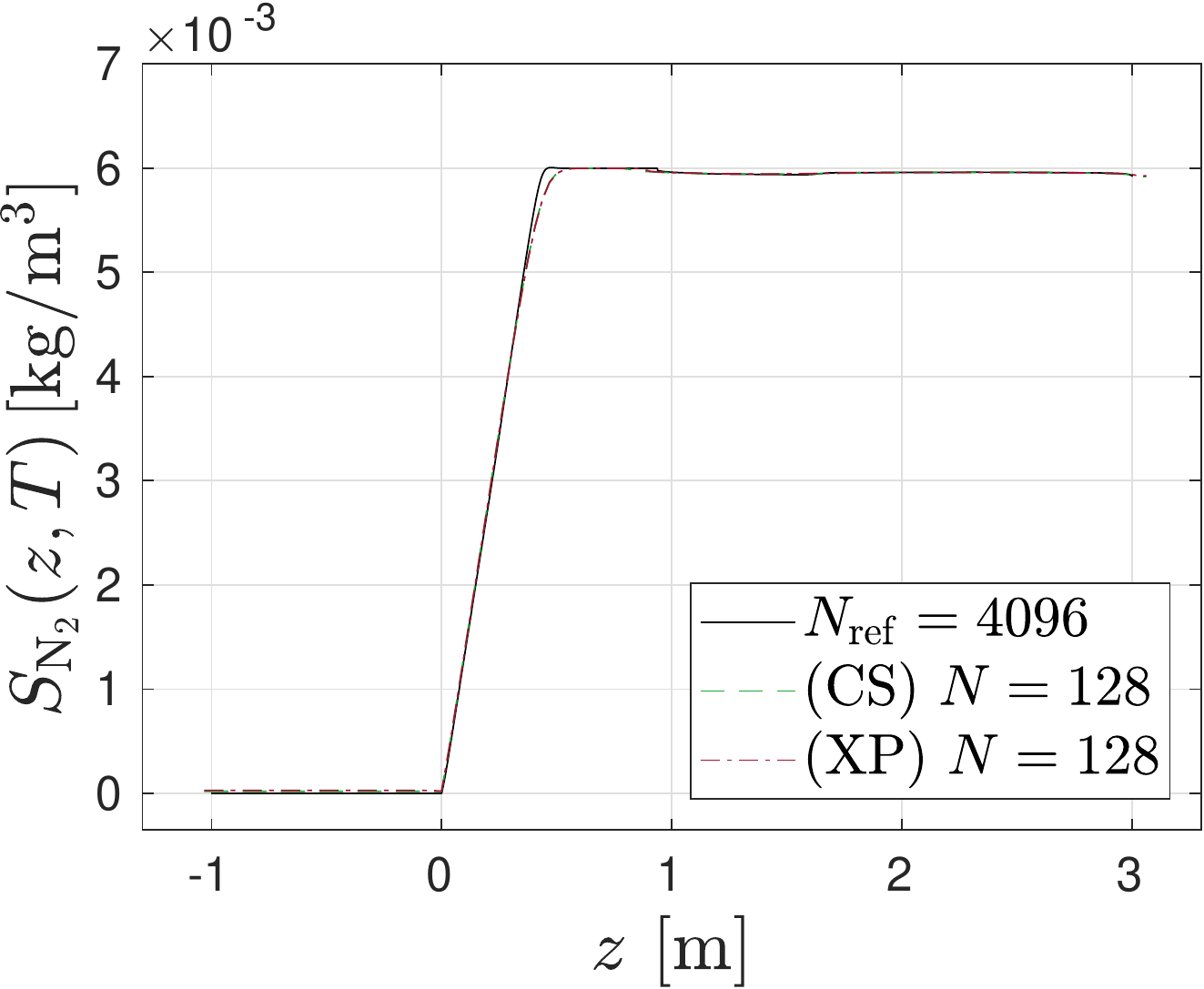} \\
                  \includegraphics[scale=0.44]{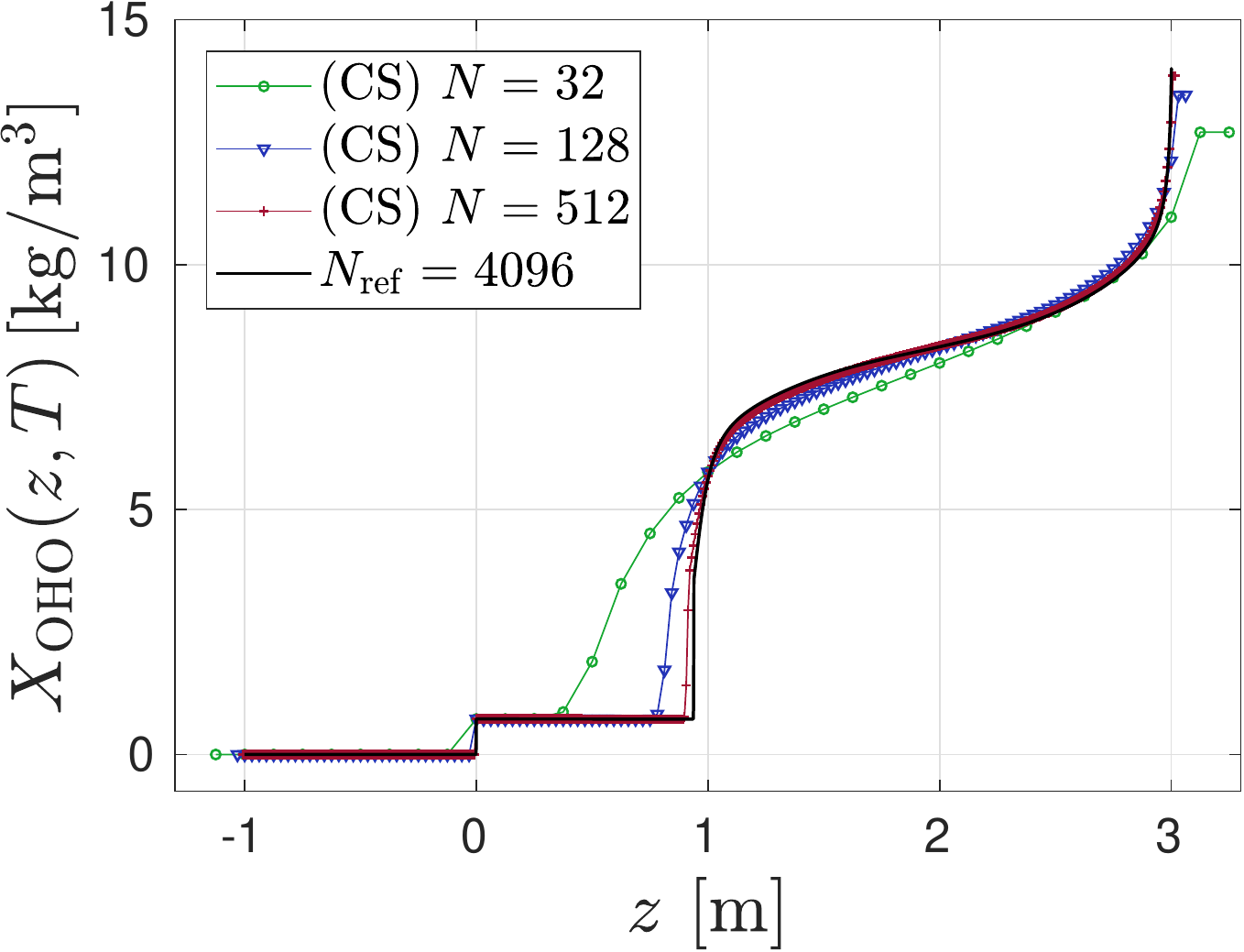}  &
  \hspace{-0.9cm} \includegraphics[scale=0.44]{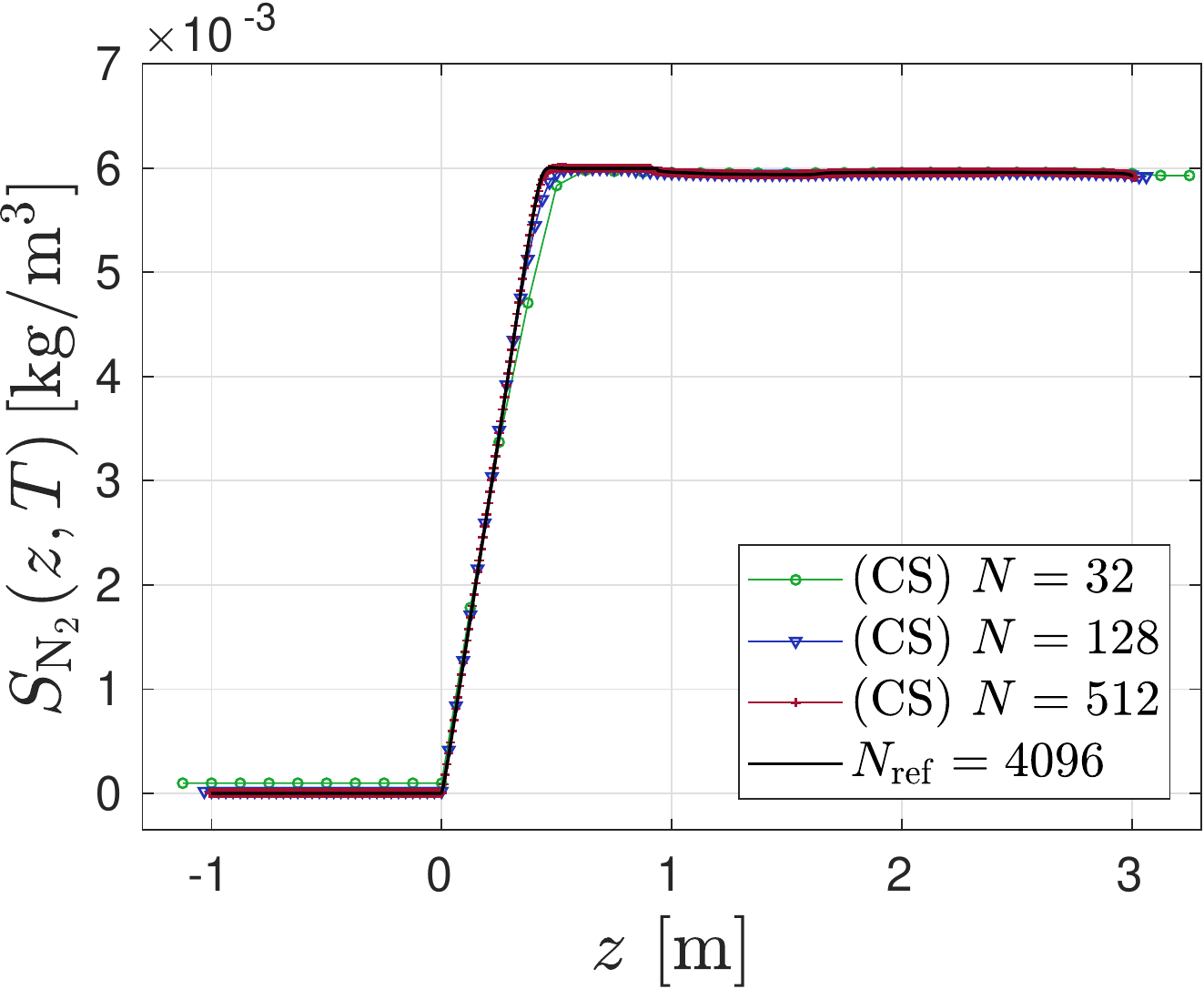}
 \end{tabular}
\end{center}
\caption{Example~1: First row: comparison of $X_{\rm OHO}$ (left) and $S_{\rm N_2}$ (right) obtained with 
Methods~XP and~CS   at simulated time $T = 9\,\rm h$ with $N = 128$. Second row: convergence of the 
Method~CS  at $T=9\,\rm h$. The reference solutions are shown in solid black.} \label{fig:profiles}
\end{figure}%

\subsection{Example 1} \label{subsec:ex1} 
In this example, we compare the new Method~CS with Method~XP of~\cite{SDm2an_reactive}.
Since the latter method only handles a constant cross-sectional area, we choose $A = 400\,\rm m^2$.
The depth parameters are $H = 1\,\rm m$ and  $B = 3\,\rm m$. 
The bulk flows are given by
\begin{align*}
 \Qf(t) = \begin{cases}
            450\,{\rm m^3/h} & \text{if $0\,{\rm h}\leq t < 2 \,{\rm h}$,}  \\
            130\,{\rm m^3/h} & \text{if $2\,{\rm h}\leq t < 4 \,{\rm h}$,} \\
            65\,{\rm m^3/h}  & \text{if $t\geq 4\,{\rm h}$,} 
           \end{cases}
\qquad
 \Qu(t) = \begin{cases}
            30\,{\rm m^3/h}  & \text{if $0\,{\rm h}\leq t < 2\,{\rm h}$,} \\
            100\,{\rm m^3/h} & \text{if $2\,{\rm h}\leq t < 4\,{\rm h}$,} \\
            35\,{\rm m^3/h}  & \text{if $4\,{\rm h}\leq t < 7\,{\rm h}$,}\\
            50\,{\rm m^3/h}  & \text{if $t \geq 7\,{\rm h}$,} 
           \end{cases}
\end{align*}
and $\Qe$ according to $\Qe(t) = \Qf(t) -\Qu(t)$. The solids feed concentrations  are  taken as
\begin{align*}
\bCf(t) = X_{\mathrm{f}}(t) \begin{pmatrix} 
                        5/7\\ 
                        2/7
            \end{pmatrix},\quad\mbox{where}\quad
 X_{\mathrm{f}}(t) = \begin{cases}
            1.0\,{\rm kg/m^3} & \text{if $0\,{\rm h} \leq t < 2\,{\rm h}$,} \\
            0.5\,{\rm kg/m^3} & \text{if $2\,{\rm h} \leq t < 4\,{\rm h}$,}\\
            3.0\,{\rm kg/m^3} & \text{if $4\,{\rm h} \leq t < 7\,{\rm h}$,}\\
            4.0\,{\rm kg/m^3} & \text{if $t \geq 7\,{\rm h}$,}
           \end{cases}  
           \end{align*} 
            and the initial conditions have been chosen as 
            \begin{gather*} 
\bC_0(z) = X_0(z) \begin{pmatrix} 
                       5/7\\ 
                       2/7
                     \end{pmatrix}, \quad \mbox{where}\quad
      X_0(z) = \begin{cases}
                0    \,{\rm m}    &  \text{if  $z<0.5 \,{\rm m}$,} \\
                3.8z+1.6\,{\rm m} &  \text{if  $z\geq 0.5 \,{\rm m},$} 
               \end{cases}  \\ 
\bS_0(z) =\begin{cases}  (0.006, 0, 0)^{\rm T} & \text{if $z < 0.5\,\rm m $,} \\
   (0,  0.12(z-0.5), 0.006)^{\rm T} & \text{if $z\geq 0.5\,\rm m $.}  \end{cases} 
   \end{gather*} 
Here and in the next examples, the initial condition for all variables is taken constant outside the vessel. The value at the respective boundary is extended,  i.e.,  we set 
 $\bC_{0}(z) = \bC_{0}(-H) $ for $z\leq -H$ and $\bC_{0}(z) = \bC_{0}(B) $ for $z\geq B$ and analogously for~$\bS_0$.

  We have computed a reference solution with $N =N_{\rm ref} := 4096$ for a simulated time of $T=9\,$h with  Method~CS, see Figure~\ref{fig:refsol}. 
The approximate numerical error~$\smash{ e_N^{\rm rel}(t)}$ of an approximate solution (with respect to the reference solution) at a simulated time point~$t$ and the estimated rate of convergence~$\theta(t)$  for two $N$-values are defined 
 as follows: 
\begin{align}\label{eq:error}
 e_N^{\rm rel}(t)& := \sum_{k = 1}^{k_{\bC}} \dfrac{\lVert C^{(k)}_N-C^{(k)}_{N_{\rm ref}} (\cdot,t)\rVert_{L^1(-H,B)}}{\lVert C^{(k)}_{N_{\rm ref}} (\cdot,t)\rVert_{L^1(-H,B)}}
 +
  \sum_{k = 1}^{k_{\bS}} \dfrac{\lVert S^{(k)}_N-S^{(k)}_{N_{\rm ref}} (\cdot,t)\rVert_{L^1(-H,B)}}{\lVert S^{(k)}_{N_{\rm ref}} (\cdot,t)\rVert_{L^1(-H,B)}},  \\ 
  \label{eq:theta}
  \theta(t) & := - \frac{\log (e^{\rm rel}_{N_1}(t)/e^{\rm rel}_{N_2}(t) )}{\log\left(N_1/N_2\right)}, 
\end{align}
Table \ref{table:error} shows these estimations in this example. As expected, both methods have order of convergence close to one.
The errors produced by Method~XP are only slightly smaller than those of Method~CS and the CPU times are about the same for both methods. 
In Figure~\ref{fig:profiles} (first row), we compare some numerical solutions for $X_{\rm OHO}$ and $S_{\rm N_2}$ at a fixed time point and for different $N$. In the second row of Figure~\ref{fig:profiles}, we visualize the convergence of numerical solutions to the reference solution, all with Method~CS.

\begin{figure}[t]
\centering
 \includegraphics[scale=0.5]{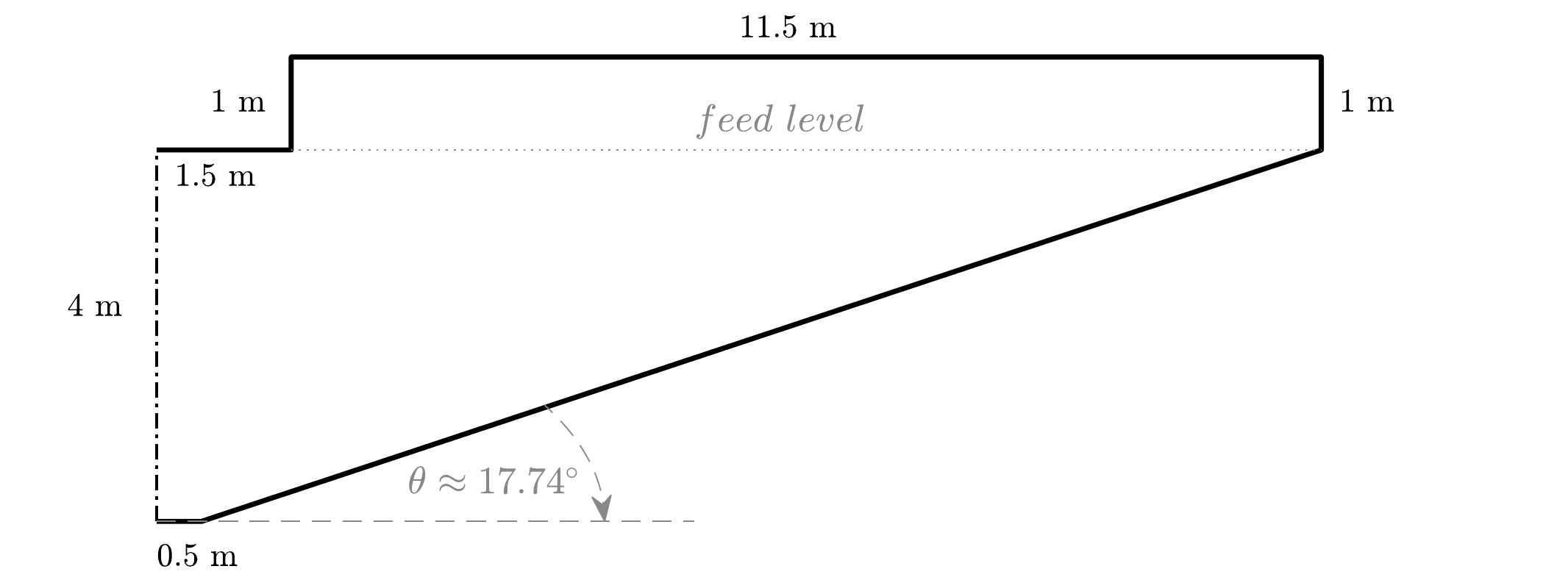} 
 \caption{Schematic of half of the vertical cross-sectional area of the axisymmetric vessel used in Examples 2 to 5. The dash-dotted line represents the axis of rotation.} \label{fig:vessel}
\end{figure}%

\begin{figure}[th] 
\centering 
 \begin{tabular}{cc}
 \includegraphics[scale=0.4]{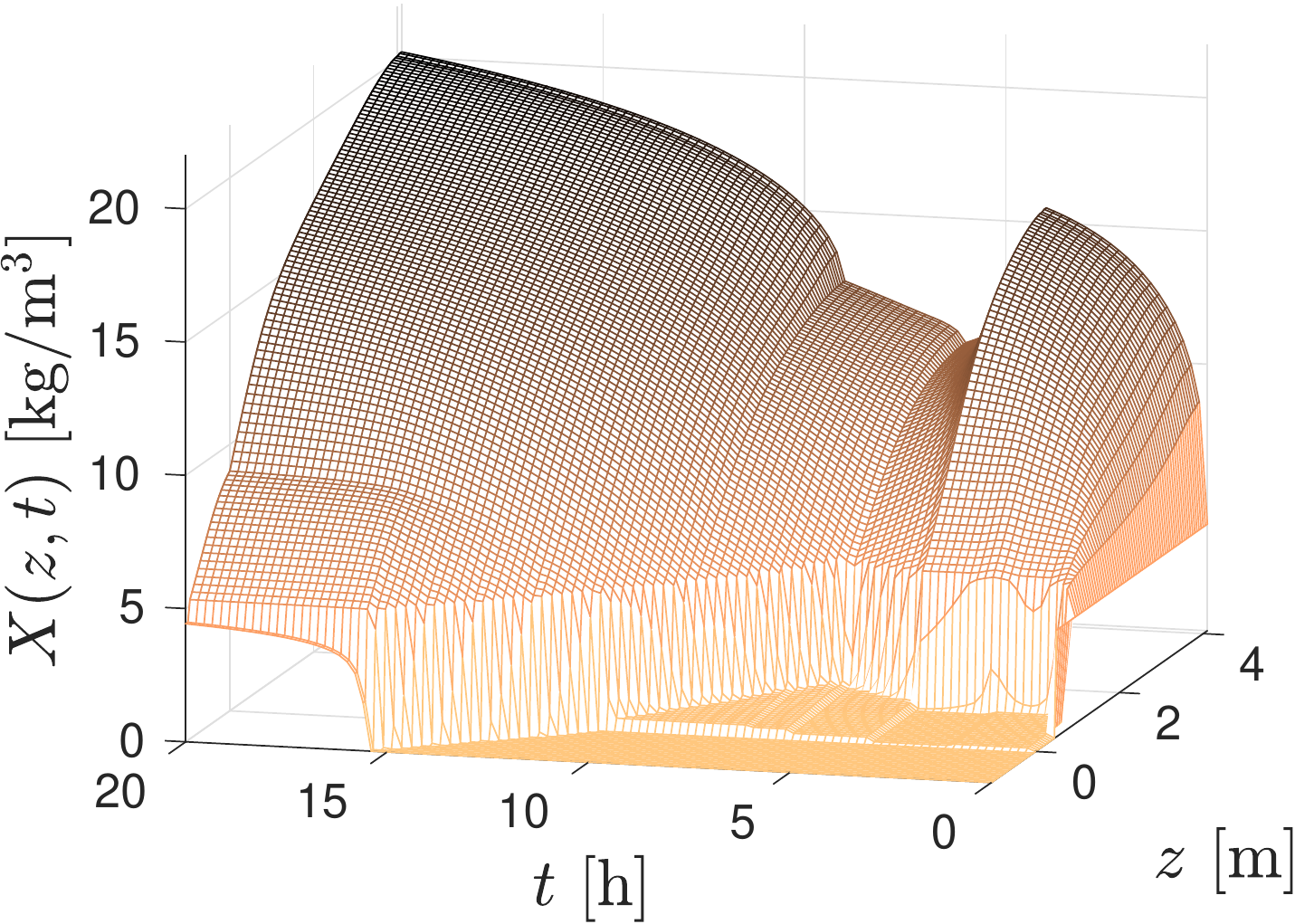} &
 \hspace{-1cm} \includegraphics[scale=0.4]{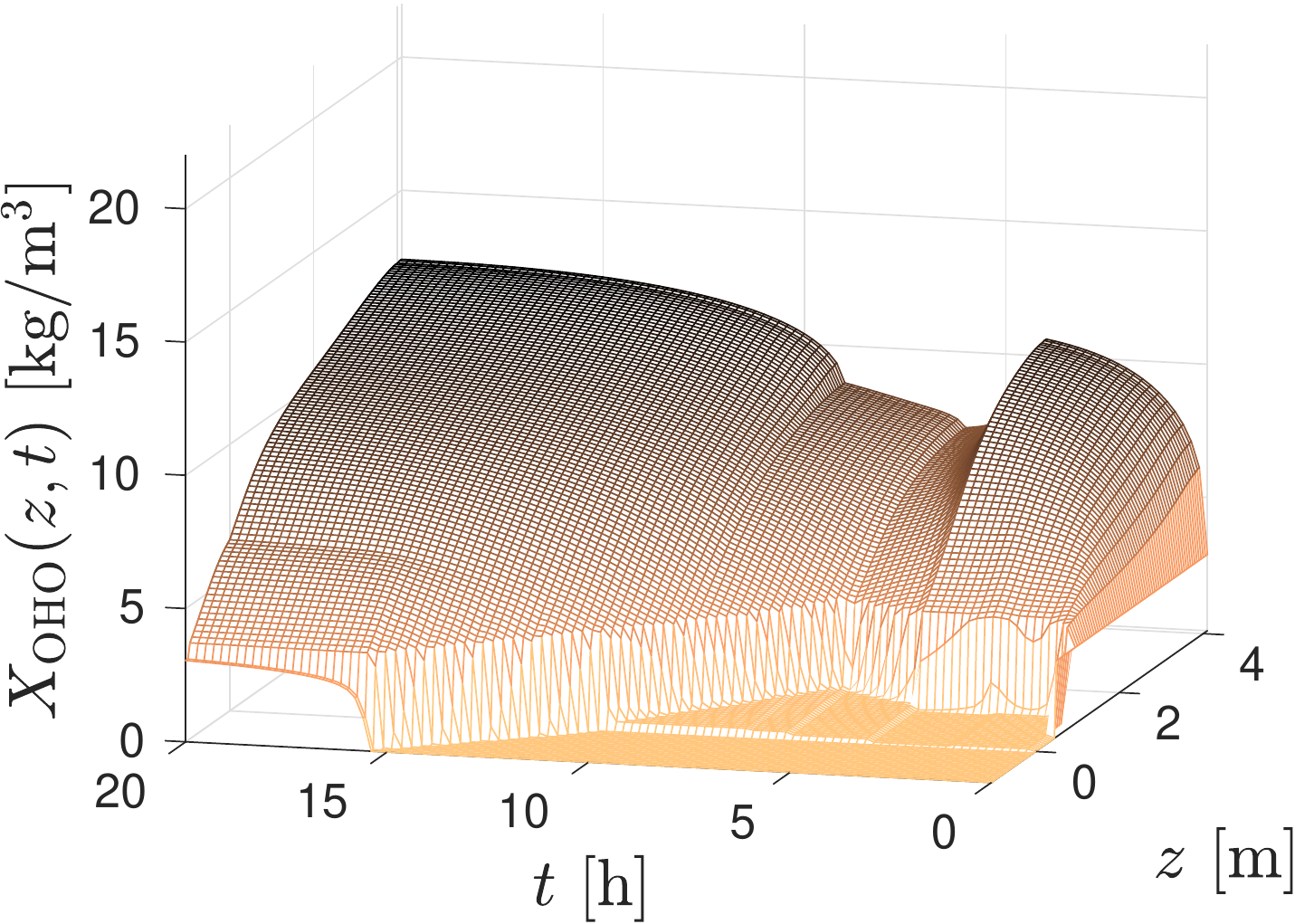} \\
 \includegraphics[scale=0.4]{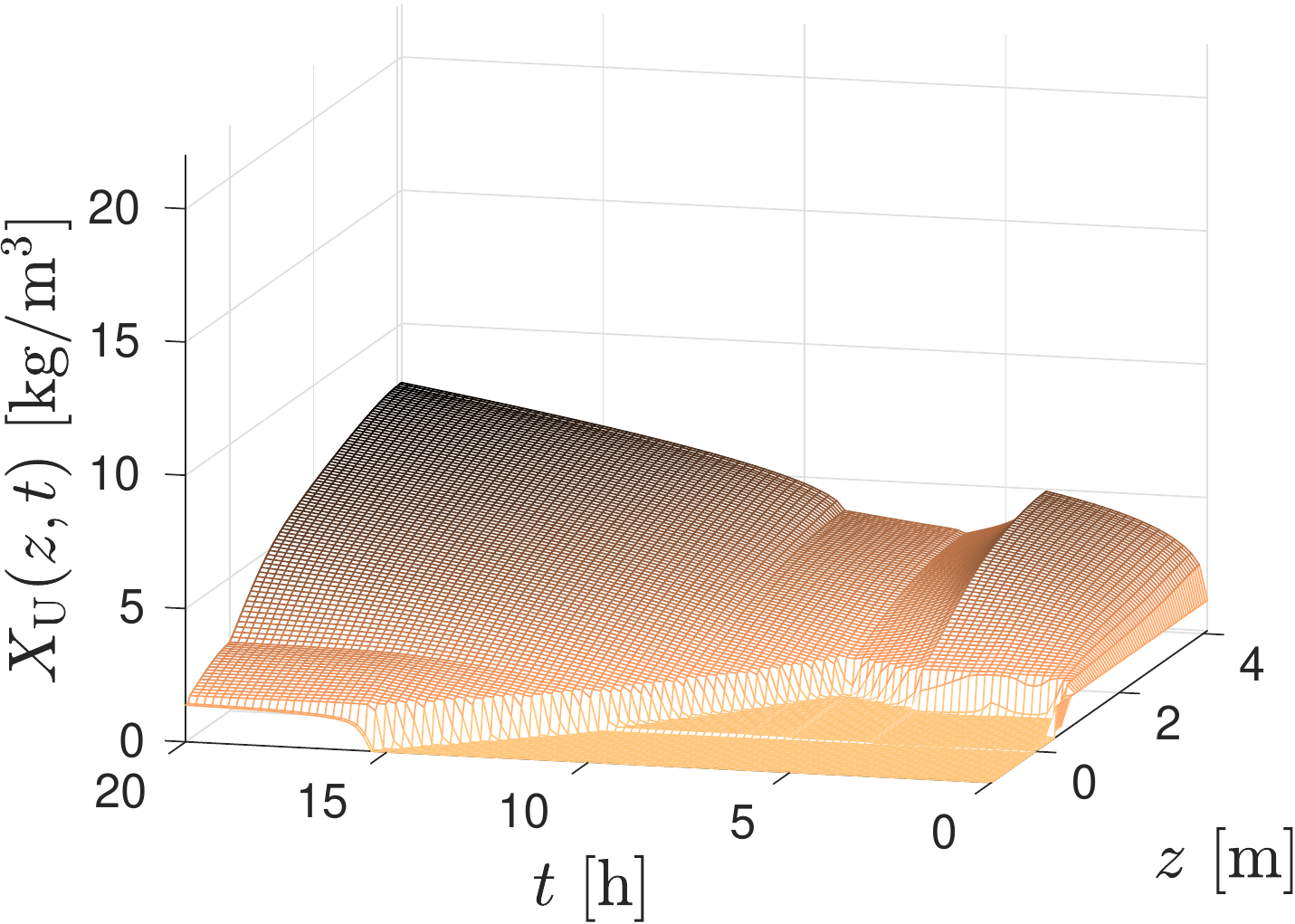} &
 \hspace{-1cm} \includegraphics[scale=0.4]{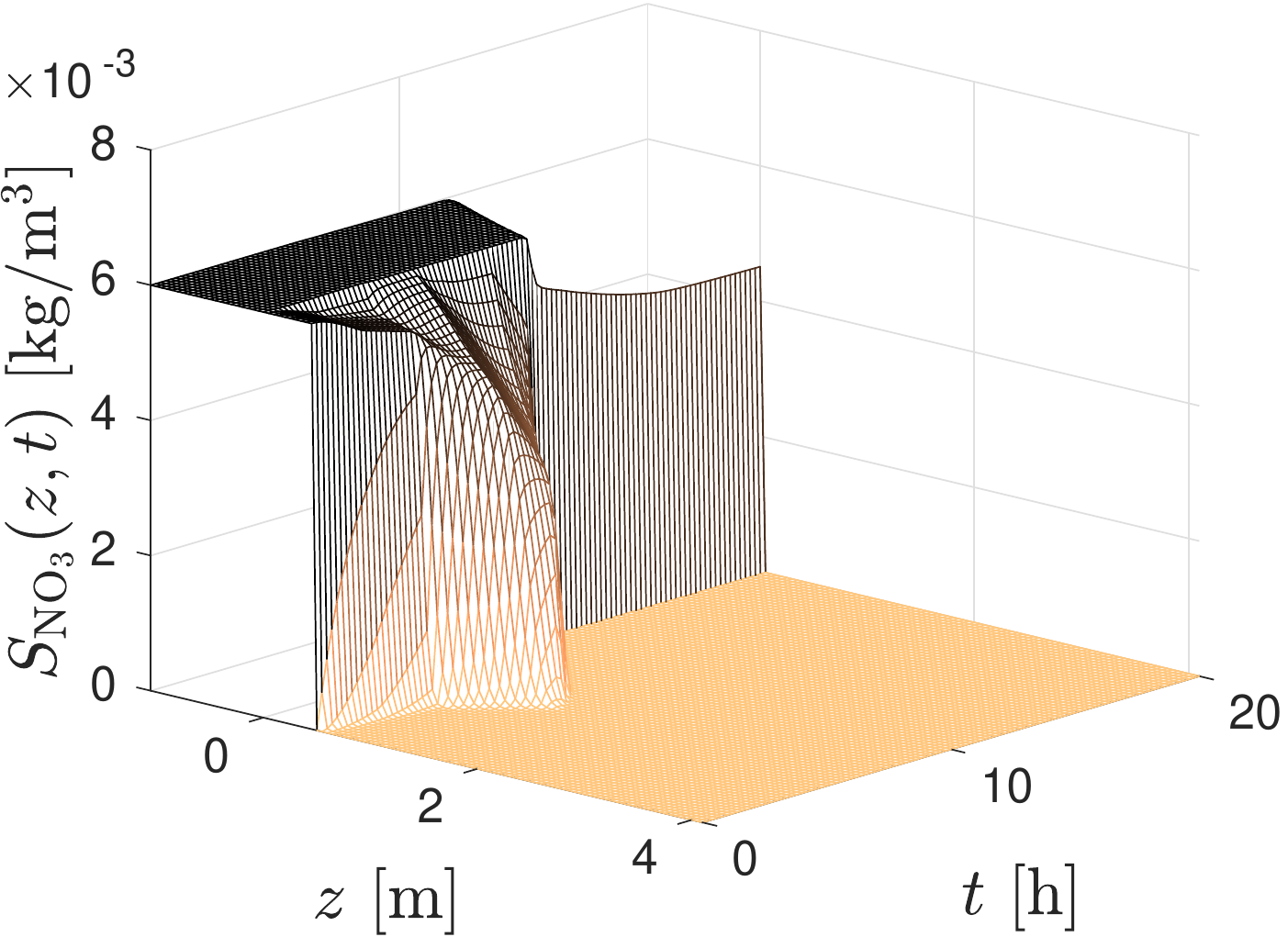} \\
 \includegraphics[scale=0.4]{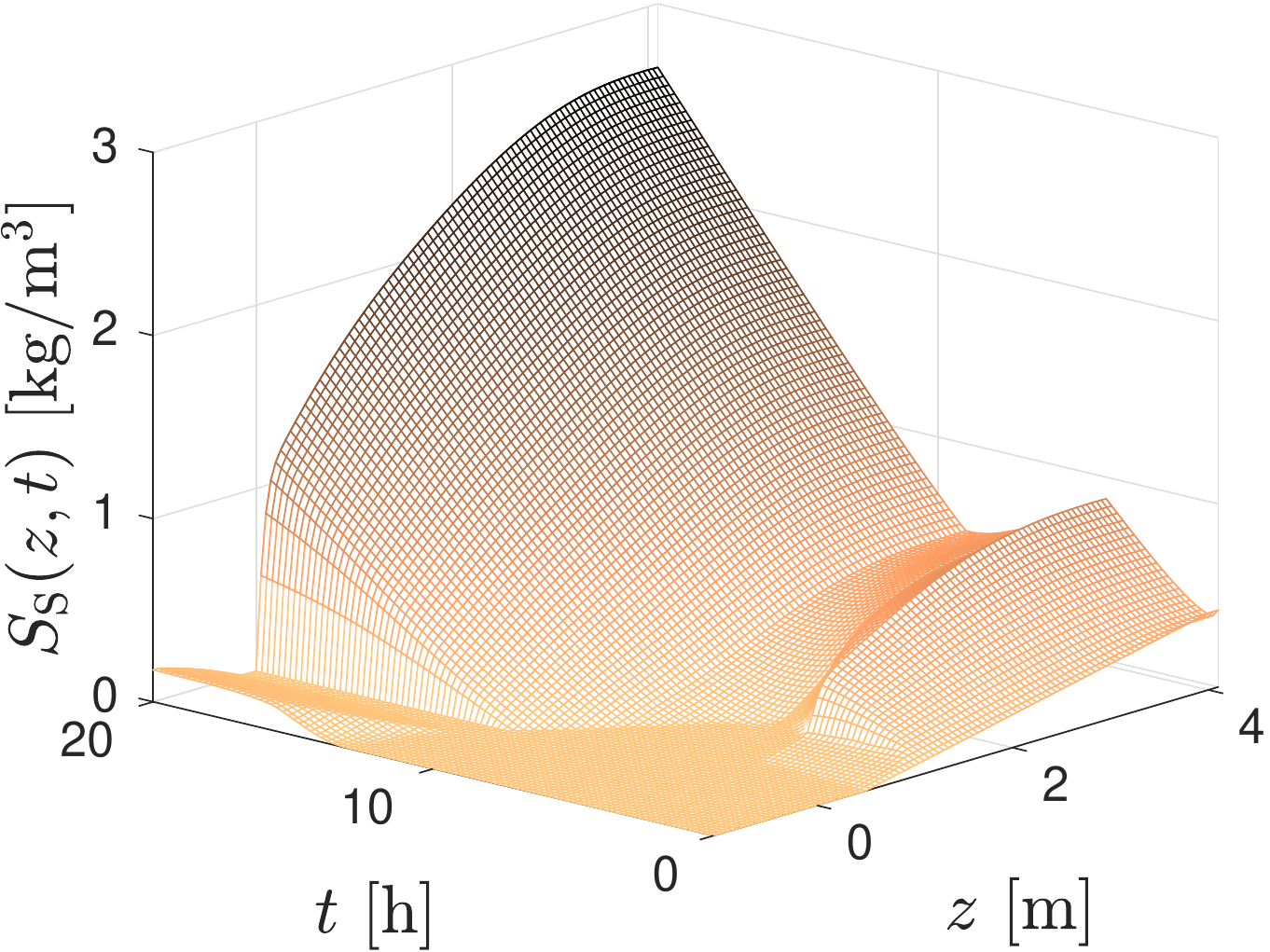} &
 \hspace{-1cm}   \includegraphics[scale=0.4]{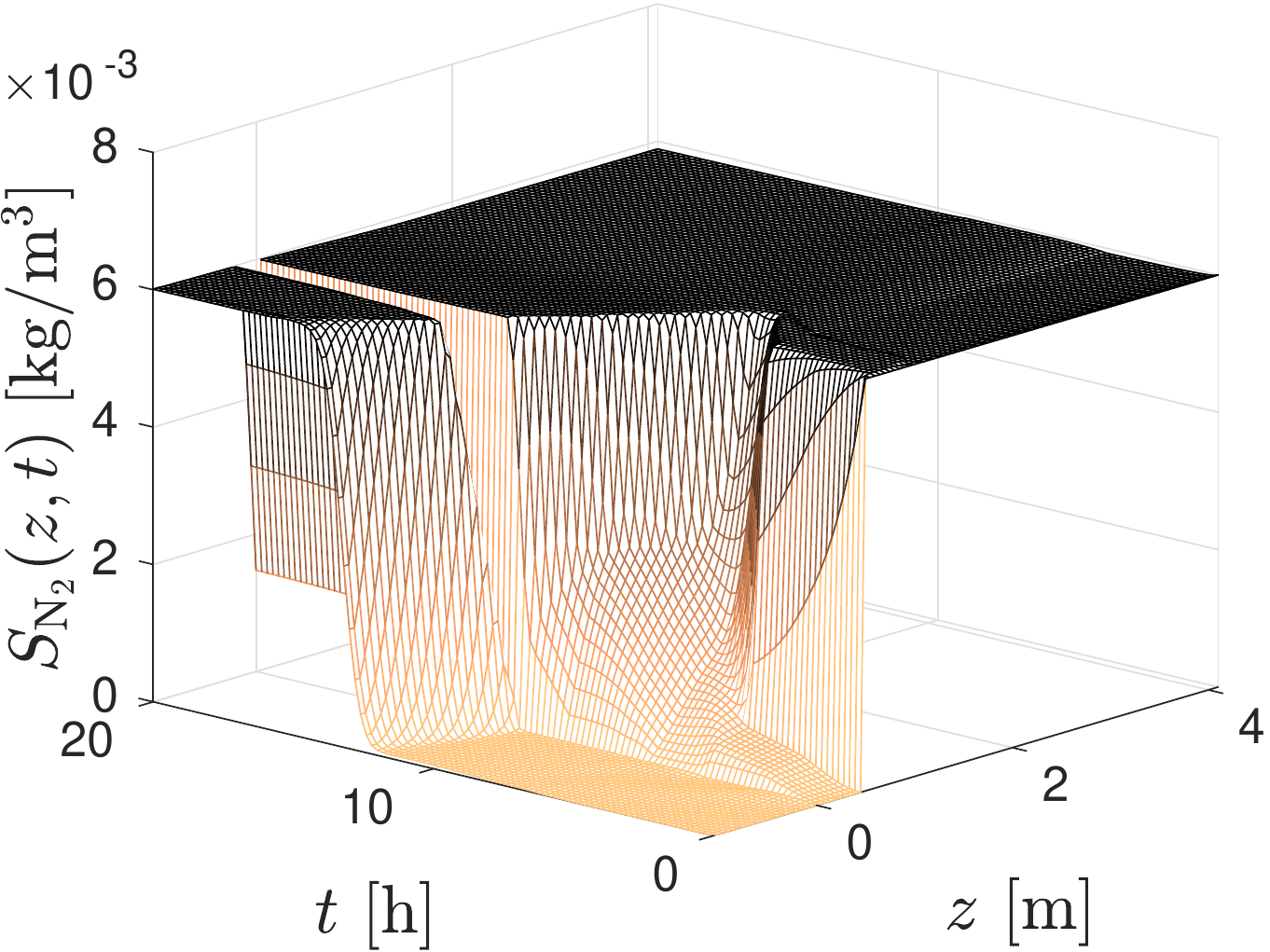}
 \end{tabular}
 \caption{Example 2: Numerical simulation with $N = 100$ until $T = 20\,\rm h$.} \label{fig:simtwo}
\end{figure}%

\subsection{Example 2} \label{subsec:ex2} 
Here and in Examples~3  to~5,   we use a non-constant function $A = A(z)$ that  describes the 
 axisymmetric, non-cylindrical  tank in Figure~\ref{fig:vessel} (cf.\ vessel V7 in \cite{SDcec_varyingA}), where $H = 1\, \mathrm{m}$ and $B = 4\, \mathrm{m}$.
In this example we use different feed and underflow bulk flows than in the previous example:
\begin{align*}
 \Qf(t) = \begin{cases}
            100\,{\rm m^3/h} & \text{if $ 0\,{\rm h}\leq t < 4 \,{\rm h}$,}  \\
            150\,{\rm m^3/h} & \text{if $ 4\,{\rm h}\leq t < 6 \,{\rm h}$,}\\
            250\,{\rm m^3/h}  & \text{if $ t\geq 6\,{\rm h}$,}
           \end{cases}
\qquad
 \Qu(t) = \begin{cases}
            10\,{\rm m^3/h}  & \text{if $ 0\,{\rm h}\leq t < 4\,{\rm h}$,}\\
            100\,{\rm m^3/h} & \text{if $ 4\,{\rm h}\leq t < 6\,{\rm h}$,} \\
            50\,{\rm m^3/h}  & \text{if $ 6\,{\rm h}\leq t < 9\,{\rm h}$,}\\
            5\,{\rm m^3/h}  & \text{if $ t \geq 9\,{\rm h}$.}
           \end{cases}
\end{align*}
The solids feed concentrations are  given by
\begin{align*}
\bCf(t) = X_{\mathrm{f}}(t) \begin{pmatrix} 
                        5/7\\ 
                        2/7
            \end{pmatrix},\quad\mbox{where}\quad
 X_{\mathrm{f}}(t) = \begin{cases}
            4.0\,{\rm kg/m^3} & \text{if $ 0\,{\rm h} \leq t < 2\,{\rm h}$,} \\
            2.0\,{\rm kg/m^3} & \text{if $ 2\,{\rm h} \leq t < 4\,{\rm h}$,}\\
            5.0\,{\rm kg/m^3} & \text{if $ 4\,{\rm h} \leq t < 7\,{\rm h}$,}\\
            6.0\,{\rm kg/m^3} & \text{if $ t \geq 7\,{\rm h}$.}
           \end{cases} 
\end{align*}
The initial condition for the solids is chosen as the step function
\begin{align*}
\bC_0(z) = \chi_{\{ z \geq 0.5 \}} 
\begin{pmatrix} 
                       20/7\\ 
                       8/7
                       \end{pmatrix}, 
\end{align*}
and for the soluble components we use the same initial condition as in Example~1.
As the simulation in Figure~\ref{fig:simtwo} shows,  the numerical scheme handles the discontinuous cross-sectional area function without any problem.

\begin{figure}[t]
\centering 
\begin{tabular}{cc}
\includegraphics[scale=0.37]{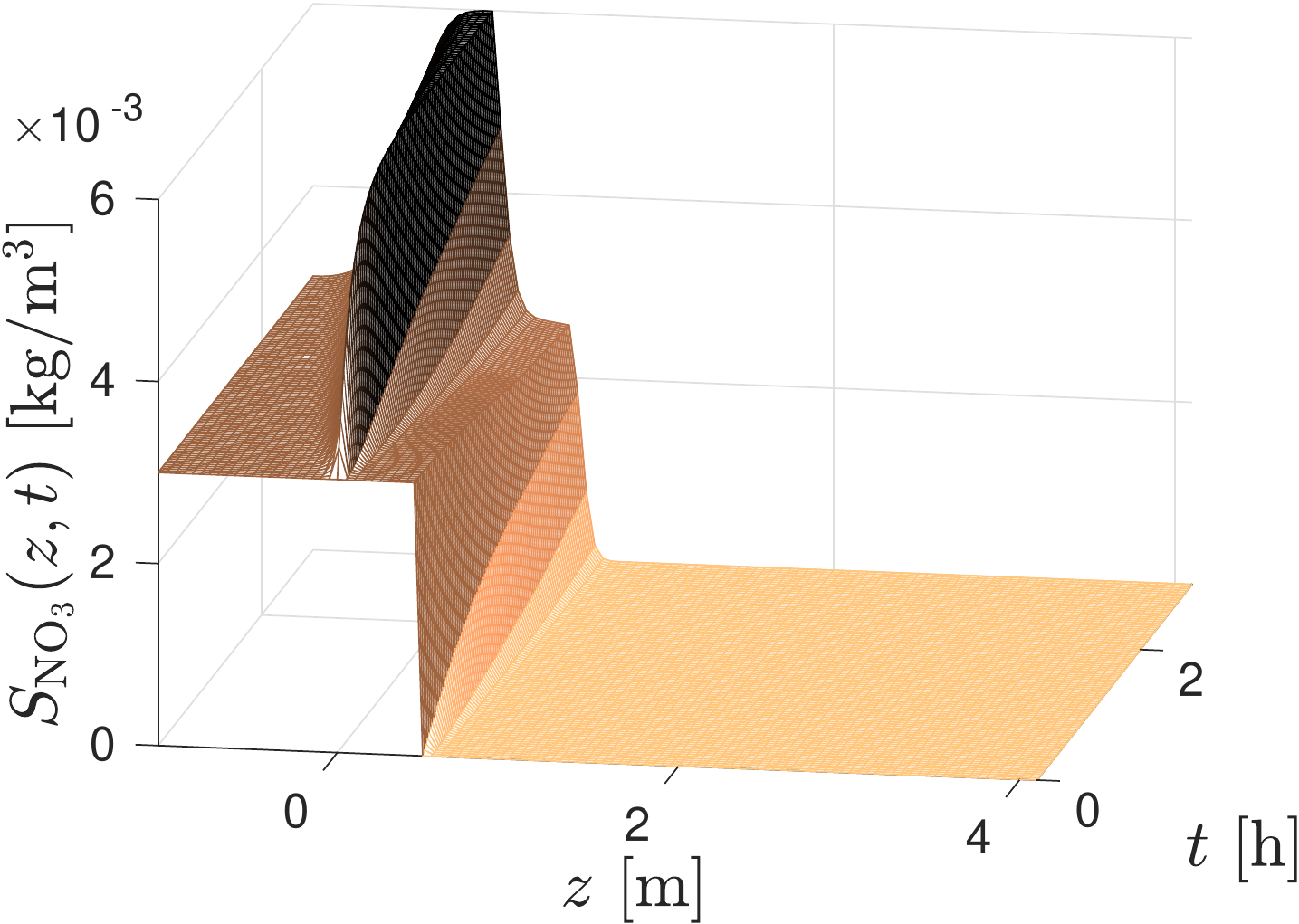} &
\hspace{-1cm}\includegraphics[scale=0.37]{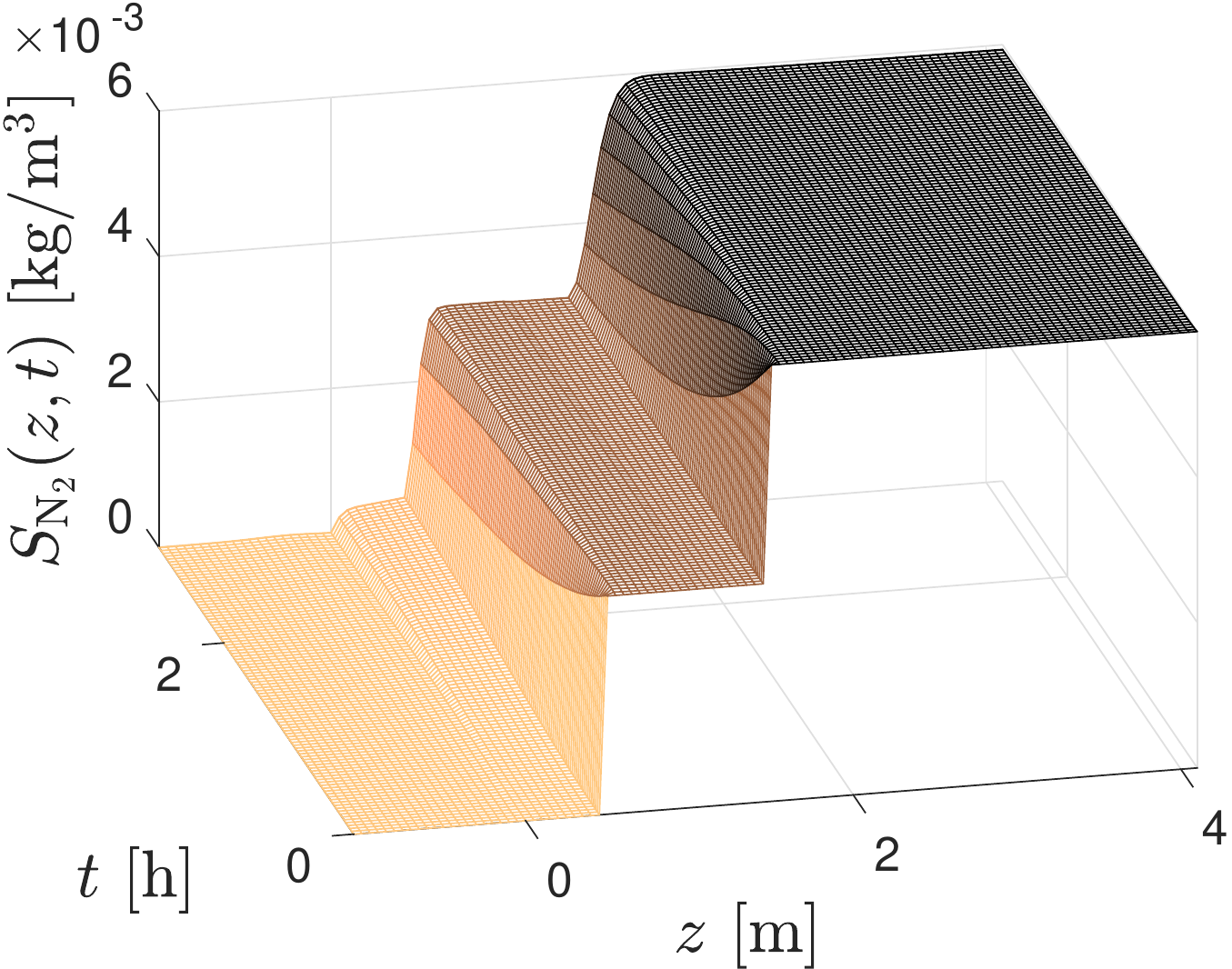} \\
\includegraphics[scale=0.37]{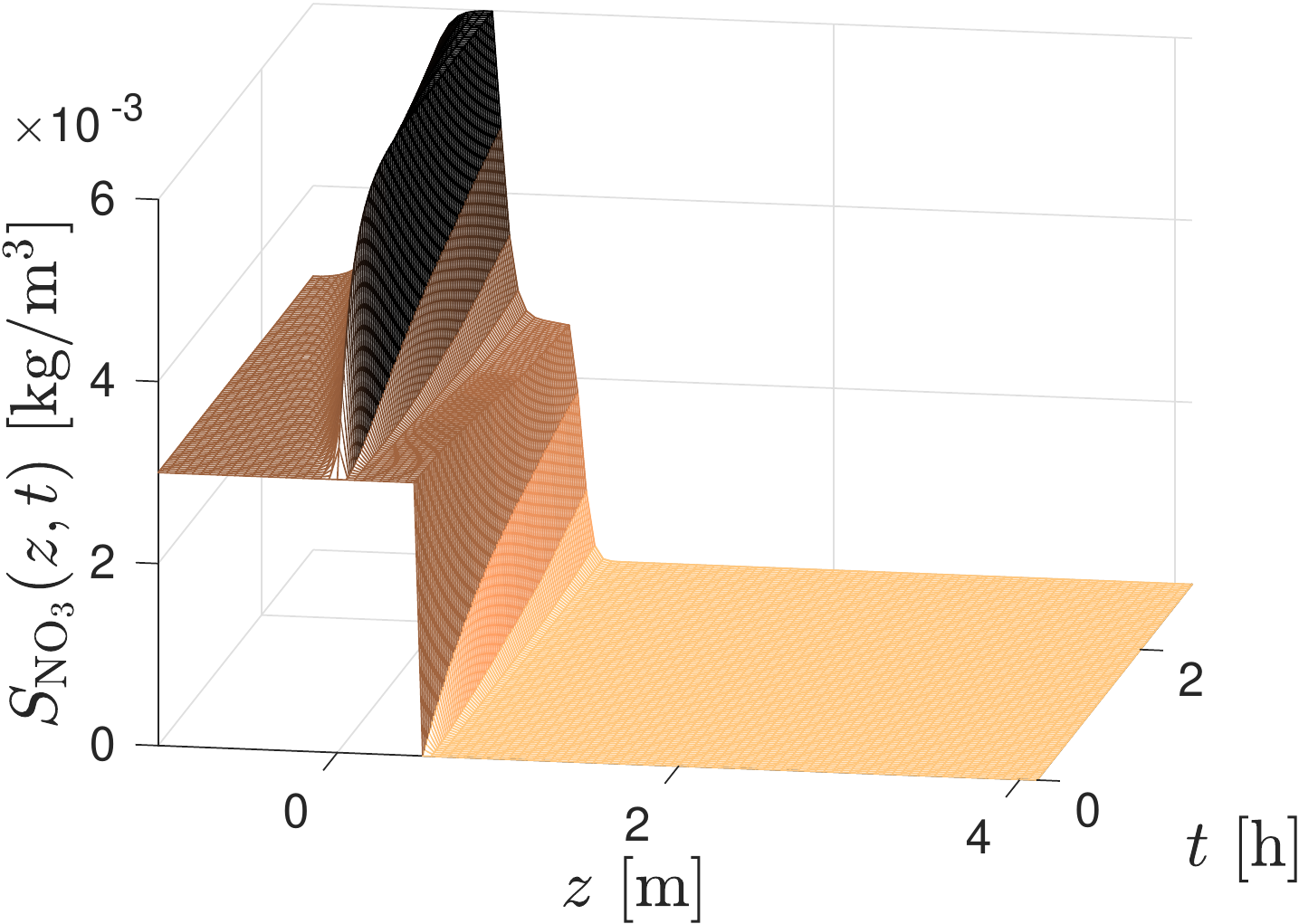} &
\hspace{-1cm}\includegraphics[scale=0.37]{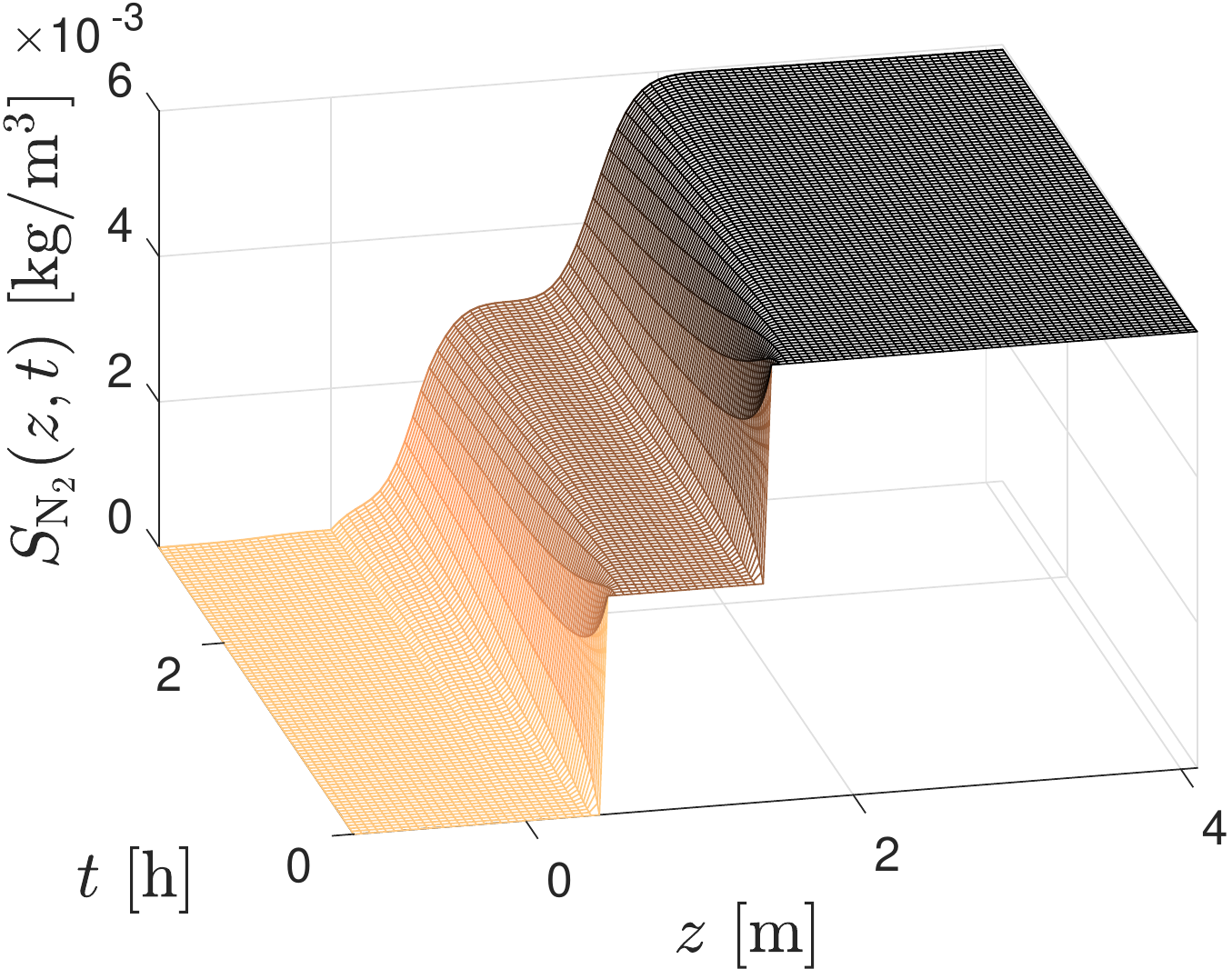} \\
\includegraphics[scale=0.37]{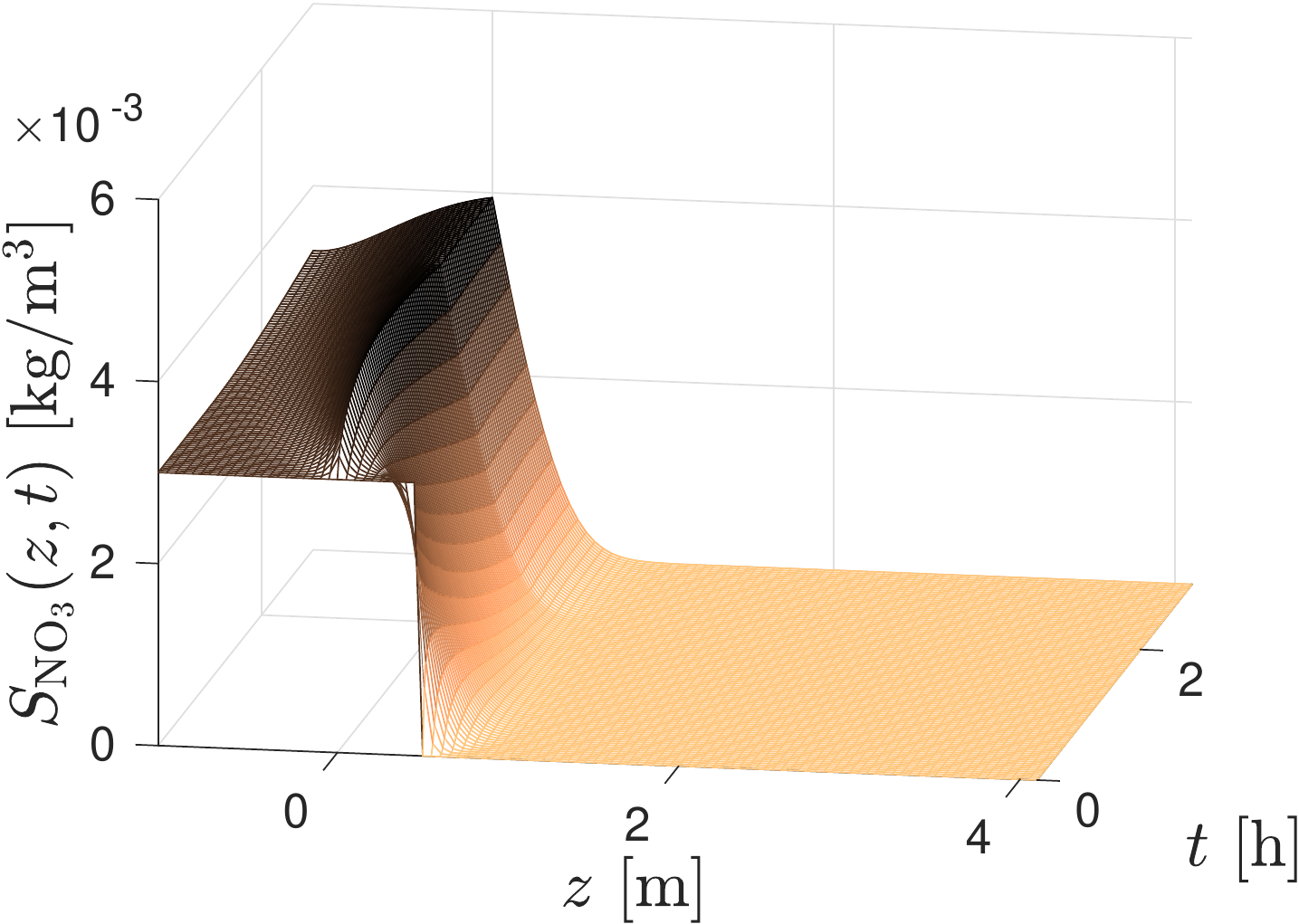} &
\hspace{-1cm}\includegraphics[scale=0.37]{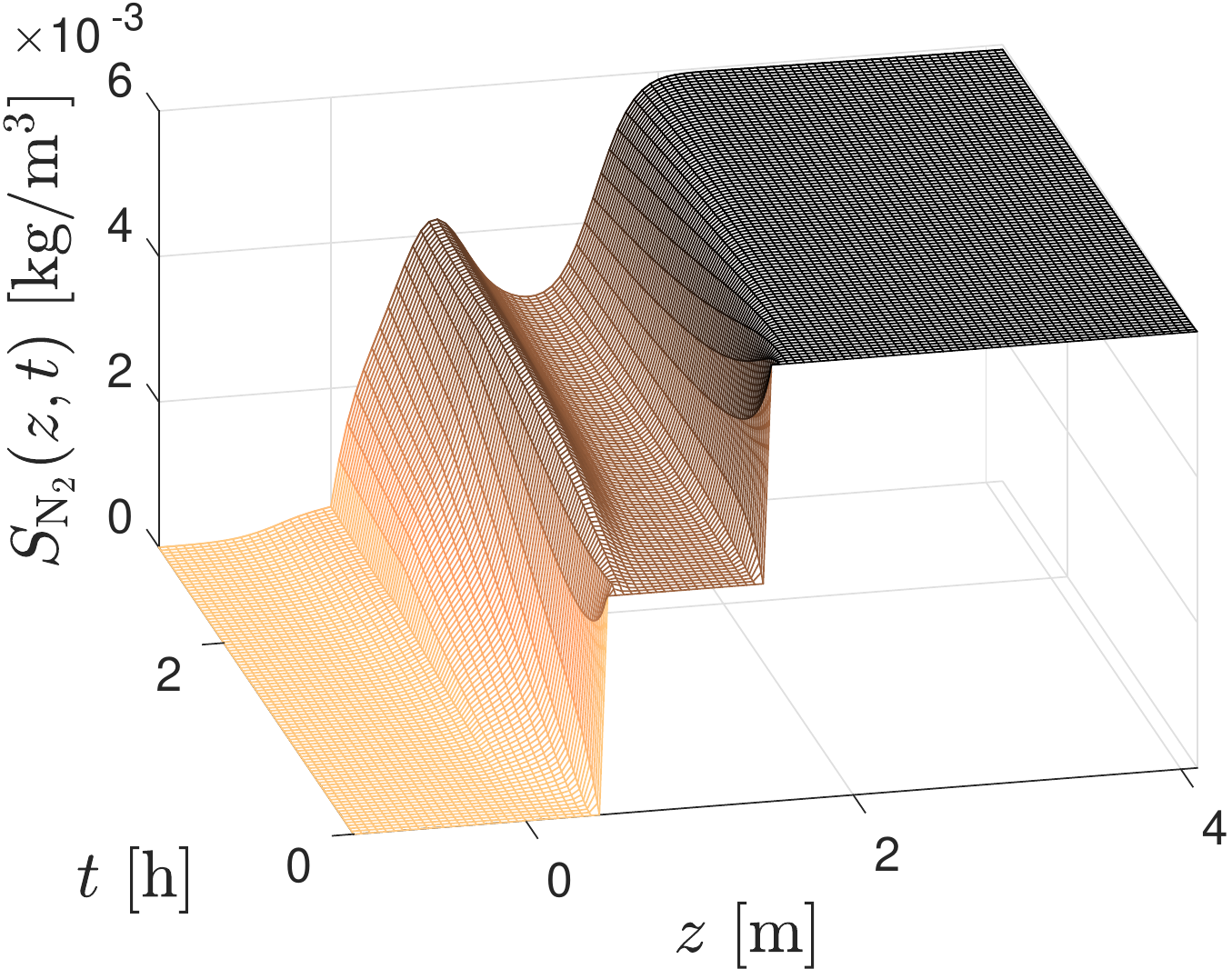}  
\end{tabular} 
\caption{Examples~3--5: The soluble components $S_{\rm NO_3}$ (first column) and $S_{\rm N_2}$ (second column) obtained with $N = 100$ until $T = 3\,\rm h$. First row: no diffusion; second row: diffusion only in $S_{\rm N_2}$; third row: diffusion in all three soluble components.}\label{fig:diff}
\end{figure}%

\subsection{Examples 3--5} \label{subsec:ex3to6} 
In this group of examples we explore the inclusion of the last ingredient of our model and numerical scheme, namely the diffusion terms in the equation  for~$\bS$. We use the same cross-sectional area, bulk flows and feed concentrations for the solid and liquid phases  as in Example~2, also the same initial condition for $\bC$. For the soluble components we consider 
\begin{align*}
S_{\rm NO_3}^0(z) = 0.006\chi_{\{ z \leq 0.5\}}, \quad S_{\rm S}^0(z) =0.12(z-0.5)\chi_{\{ z \geq 0.5\}}, \quad 
 S_{\rm N_2}^0 = \begin{cases}
                     0 & \text{if $ z<0.5\,\rm m$,}  \\
                  0.003& \text{if $ 0.5\,{\rm m}\leq z <1.5\,\rm m$,}\\
                  0.006& \text{if $ z\geq 1.5\,\rm m$.}
                 \end{cases}
\end{align*}
For Example~3 we set all diffusion coefficients to zero, in Example~4 we let $d^{(1)} = d^{(2)} = 0\,\rm m^2/s$ and $d^{(3)} = 3\times 10^{-6}\,\rm m^2/s$, and for Example~5 we have $d^{(1)} = 10^{-5}\,\rm m^2/s$, $d^{(2)} = 5\times 10^{-5}\,\rm m^2/s$ and $d^{(3)} = 3\times 10^{-6}\,\rm m^2/s$. 

Figure~\ref{fig:diff} shows the $S_{\rm NO_3}$ and $S_{\rm N_2}$ components for Examples~3 (first row) to~5 (third row),  where we can observe the  effect of different diffusion coefficients. As expected, the inclusion of diffusion in the third component $S_{\rm N_2}$ (second row, Ex.~4) smoothes out the solution without diffusion (first row, Ex.~3).
Nevertheless, the influence of this diffusion on the other components is not very accentuated. The inclusion of diffusion in all soluble components (third row, Ex.~5) shows the effect of cross diffusion with a wave created near the discontinuity at $z = 0.5\, \mathrm{m}$ in the solution of $S_{\rm N_2}$.

\begin{figure}[t]
\centering 
  \includegraphics[scale=0.5]{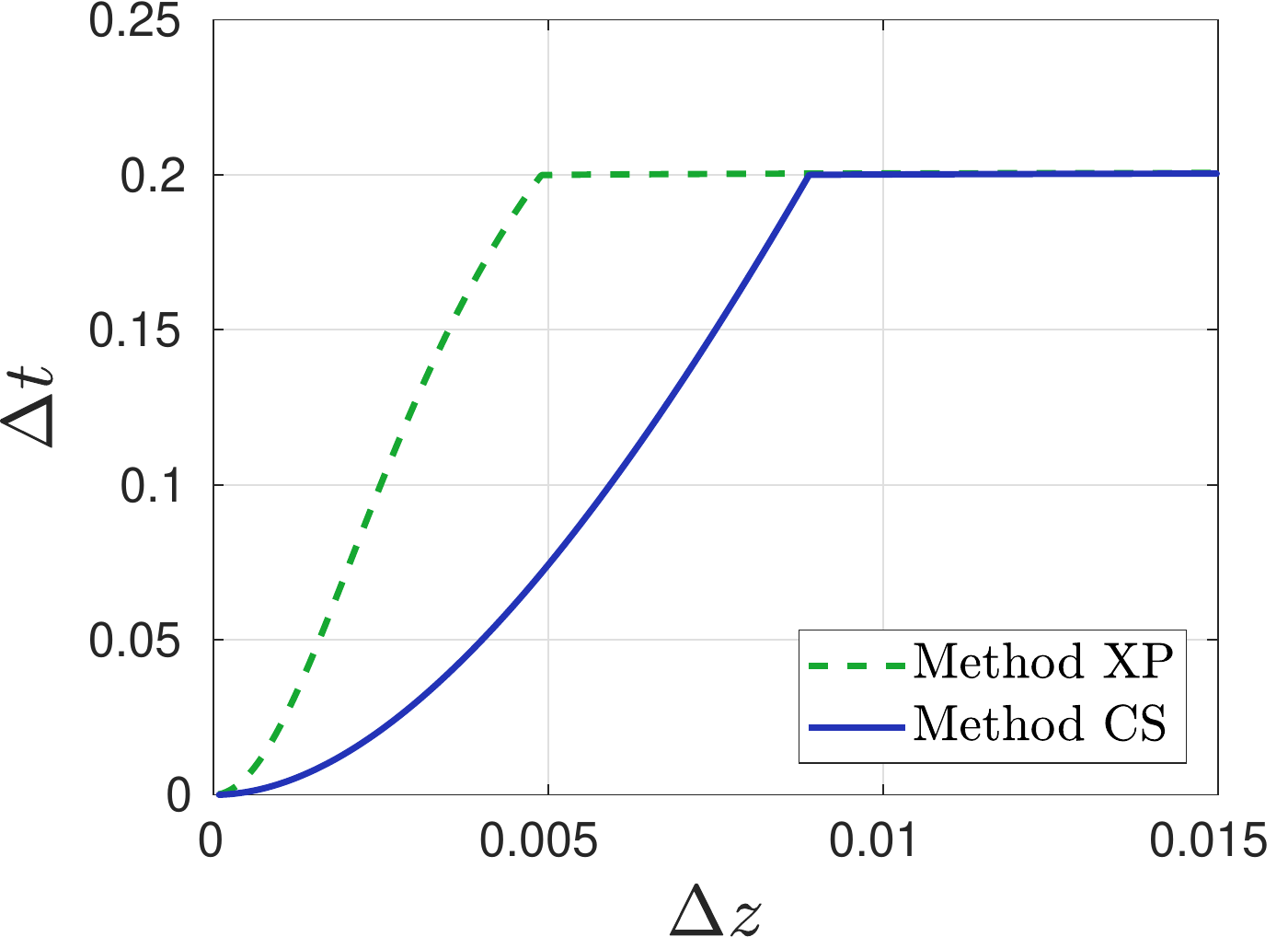}
  \caption{Graphs of the time step $\Delta t$ as function of $\Delta z$ given by the respective CFL conditions of Method~XP (dashed green) and Method~CS (solid blue).
For larger $\Delta z$, the graphs are approximately constant.\label{fig:cfl}}
\end{figure}%
 
\section{Conclusions} \label{sec:conc} 

The main novelty and advantage of the new numerical method (Method~CS) 
  is  its  formulation in method-of-lines (MOL) form. This property  
    makes it possible to implement Method~CS within commercial simulators together with other submodels  
     of WRRFs, which are mostly posed in ODE form. In fact, 
Method~CS only contains   easily implemented explicit formulas, in contrast to previously published methods \citep{SDcace_reactive,SDm2an_reactive} that involve the Godunov numerical flux, which on the other hand is expected to yield slightly more accurate solutions. 
Other advantages of the present model 
 in comparison with 
  previous efforts  \citep{SDcace_reactive,SDm2an_reactive} 
  include  the  incorporation of 
  diffusion or dispersion of each soluble component within the liquid 
   and the variation of the cross-sectional area~$A(z)$ with depth. The cross-sectional area  may even vary   discontinuously,  
    which may be useful for the appropriate  description of  the feed inlet. Thus, the 
     model may handle   realistic rotationally symmetrical shapes of SSTs.

A reformulation of the equivalent (for constant $A$ and without diffusion) model of \cite{SDm2an_reactive} made it possible to derive the MOL equations \eqref{eq:MOL}.
For the fully discrete scheme~\eqref{eq:numscheme}, we have proved an invariant-region property under the condition~\eqref{cfl}; see Theorem~\ref{thm}.
This means positivity of all concentrations and boundedness above of the solids concentrations; however, we have not been able to bound the substrate concentrations from above.

With respect to the numerical  results, we mention that 
Example~1 demonstrates that there is no substantial difference in performance between Method~CS and the previous Method~XP \citep{SDm2an_reactive}, which are both of first order; see Table~\ref{table:error}.
This holds for discretizations that are normally used (roughly,  $\Delta z \geq  0.01\, \mathrm{m}$, which for a tank of height 4\,m corresponds to $N \leq 400$ cells). 
The similar CPU times in Table~\ref{table:error} can be explained by the plot of the respective CFL conditions for the two methods; see Figure~\ref{fig:cfl}.
For small $\Delta z$ that figure reveals the expected parabolic behaviour of $\Delta t$ as a function of $\Delta z$.
For $\Delta z\approx 0.005\, \mathrm{m}$ ($N\approx 800$), Method~XP is the faster one.
The reason for the almost constant values ($\Delta t\approx 0.2\,$s) for large $\Delta z$ is the reaction terms contribution in the CFL conditions. Example~2 shows that the numerical scheme can handle non-constant cross-sectional area functions even having discontinuities. 
Example~3 exhibits the versatility of soluble diffusion effects, which includes cross diffusion between the soluble components.

Future research related to the present model should be conducted in at least three directions.  One of them is related to the well-posedness 
 (existence, uniqueness, and continuous dependence on data of solution) 
   of the underlying mathematical model. Specifically, while the well-posedness of general hyperbolic systems 
    and in particular strongly degenerate parabolic systems is essentially unavailable, an effort should be made to analyze whether 
     the well-posedness of the present model can possibly be reduced to that of a single degenerate parabolic equation for~$X$ 
      plus first-order transport equation for the solid concentrations, akin to the 
       formulation that led to  Method~XP (see the Appendix). In particular, it remains to elaborate  an analytical counterpart, based 
        on PDE theory,  of   the invariant region principle 
       (Lemmas~\ref{lem:Cjbound} to~\ref{lem:Sjbound}) established herein for discrete solutions. 
       
  With respect to numerical schemes, we  mention that Method~CS has been developed  under the aspect of ease of implementation, preference of an MOL formulation,  and satisfaction 
   of a (partial) invariant-region principle. The options of improving the method to make it computationally more efficient have not yet been explored. As a monotone scheme including a first-order time discretization, the method is only first-order accurate and could be upgraded 
    to formal second or higher order accuracy by standard techniques such as monotone upstream centered (MUSCL-type) variable extrapolation  
    or high-order weighted essentially non-oscillatory (WENO) reconstructions in combination, for instance, 
      with strong stability-preserving (SSP) Runge-Kutta time schemes for time integration. All these techniques 
       are treated, for instance, by~\cite{hest18}. Another potential improvement could be to treat certain contributions, 
        for example the discretizations of diffusive terms,  
        in the MOL formulation \eqref{eq:MOL} implicit in time, in the spirit of implicit-explicit (IMEX) schemes for
         time-dependent PDE (see, e.g.,  \cite{bosc15} and references cited in that work). However, such partitioned schemes 
          are not compatible with the  preferred MOL form. In addition,  while these schemes are  devised to achieve 
           a less restrictive  CFL condition (allowing larger time steps), a real gain  in CPU is achieved only for such problems 
             where the strongest time step restriction comes from the discretization of  diffusive terms. 
              However, Figure~\ref{fig:cfl} indicates  that for the present model discretized by Methods~CS or~XP, 
               such gains are likely to accrue for very fine  discretizations only.

       Finally, we comment that it would be very desirable to compare the present model with experimental evidence  and 
        to calibrate the material specific functions, such as $v_{\mathrm{hs}}$ and $ \sigma_{\mathrm{e}}$, properly to make the 
         model usable for prediction, control and simulation of real-world scenarios. However, while  
          data for the  non-reactive model of sedimentation with compression are available (see, e.g., 
           \cite{DeClercq2003,DeClercq2008}) 
            and the reaction kinetics come from standardized models in wastewater treatment \citep{metcalf}, information that combines 
             both ingredients is scarce but includes recent work by \cite{Kirim2019}.

\section*{Acknowledgements} 

RB~is  supported by  CONICYT/PIA/AFB170001; CRHIAM, Proyecto ANID/FONDAP/15130015;   Fon\-de\-cyt project 1170473; and by the INRIA Associated Team ``Efficient numerical schemes for non-local transport phenomena'' (NOLOCO; 2018--2020).
SD~acknowledges support from the Swedish Research Council (Vetenskapsr\aa det, 2019-04601).

\bibliographystyle{apalike}

\section*{Appendix. Method XP}

For easy of reference, we here summarize Method~XP developed by \cite{SDm2an_reactive}. 
We use the same notation as in Section~\ref{sec:scheme} when there is no ambiguity, but also functions and constants defined in Section~\ref{sec:two}. 
Let $j = -1,\dots,N+1$, $\Delta z$ and the nodes $z_j$, $z_{j+1/2}$ taken as in Section~\ref{sec:scheme}.  

The total concentrations are denoted by $X_j$ and $L_j$ in both methods, while in Method~XP we define the percentage vectors $\boldsymbol{P}_{X,j}\in \mathbb{R}^{k_{\bC}}$ and $\boldsymbol{P}_{L,j}\in\mathbb{R}^{k_{\bS}}$ of the subcomponents of the solid and liquid phases in the cell~$j$, respectively.
Note that here the dimension of $\boldsymbol{P}_{L,j}$ is $k_{\bS}$, which means that we do not include the percentage of water.
The concentrations of the subcomponents are then computed by $\bC_j = \boldsymbol{P}_{X,j}X_j$ and $\bS_j = \boldsymbol{P}_{L,j}L_j$.

Method~XP uses Godunov's numerical flux of the unimodal flux function $f(X) := X\vhs(X)$:
\begin{align*}
 \mathcalold{G}_j(X_j,X_{j+1}) & := \min\big\{f(\min\{X_j,\hat{X}\}), f(\max\{X_{j+1},\hat{X}\})\big\},
\end{align*}
and the function
\begin{align*}
 D(X) := \dfrac{\rho_X}{g\Delta \rho} \int_{X_{\rm c}}^X \vhs(s) \sigma'_{\rm e}(s) {\rm d}s.
\end{align*}
For the appoximation of the cell boundary fluxes, we define
\begin{align*}
 \tilde{F}_{X,j+1/2} & := q_{j+1/2}^+X_j + q_{j+1/2}^-X_{j+1} + \gamma_{j+1/2}\mathcalold{G}_j(X_j,X_{j+1/2}) - \gamma_{j+1/2}\left(D(X_{j+1})-D(X_{j})\right)/\Delta z,\\
 \tilde{F}_{L,j+1/2} & := \rho_Lq_{j+1/2}-r\tilde{F}_{X,j+1/2},\\
(\boldsymbol{P}_{E} \tilde{F}_E)_{j+1/2} &:= \tilde{F}_{E,j+1/2}^+\boldsymbol{P}_{E,j} +   \tilde{F}_{E,j+1/2}^- \boldsymbol{P}_{E,j+1},\quad E\in\{X,L\}.
\end{align*}
With
\begin{align*}
\Psi_{X,j}^n &:= \boldsymbol{P}_{X,j}^nX_j^n + \dfrac{\Delta t}{\Delta z}\left(-[\Delta (\boldsymbol{P}_{X}^n\tilde{F}_X^n)]_j + \delta_{j,j_{\rm f}}\boldsymbol{C}_{\rm f}^n q_{\rm f}^n\right) + \Delta t\, \gamma_j \boldsymbol{R}_{\bC,j}^n,\\
\Psi_{L,j}^n &:= \boldsymbol{P}_{L,j}^nL_j^n + \dfrac{\Delta t}{\Delta z}\left(-[\Delta (\boldsymbol{P}_{L}^n\tilde{F}_L^n)]_j + \delta_{j,j_{\rm f}}\boldsymbol{S}_{\rm f}^n q_{\rm f}^n\right) + \Delta t\, \gamma_j \boldsymbol{R}_{\bS,j}^n,
\end{align*}
the marching formulas are given by
\begin{align*}
 X_j^{n+1} &= X_j^n + \dfrac{\Delta t}{\Delta z}\left(-[\Delta \tilde{F}_X^n]_j + \delta_{j,j_{\rm f}}X_{\rm f}^n q_{\rm f}^n\right) + \Delta t \gamma_j\tilde{R}_{C,j}^n,\\
 \boldsymbol{P}_{X,j}^{n+1} &= \begin{cases}
                               \boldsymbol{P}_{X,j}^n &\mbox{if } X_{j}^{n+1}=0,\\
                               \Psi_{X,j}^n/X_j^{n+1} & \mbox{if } X_{j}^{n+1}>0,
                              \end{cases}\\
 L_j^{n+1} &  = \rho_L - rX_j^{n+1},\\
 \boldsymbol{P}_{L,j}^{n+1}& = \Psi_{L,j}^n/L_j^{n+1}.
\end{align*}
The CFL condition is given by
\begin{align*}
 \Delta t \left( \dfrac{\lVert q \rVert_{\infty}}{\Delta z} + \max\{\beta_X,\beta_{\boldsymbol{P}_X},\beta_{\boldsymbol{P}_L}\} \right)\leq 1,
\end{align*}
where
\begin{align*}
 \beta_X:= \dfrac{\lVert f' \rVert_{\infty}}{\Delta z} + \dfrac{\lVert D' \rVert_{\infty}}{\Delta z^2} + \tilde{M}_{\bC} + r\tilde{M}_{\bS}, \quad 
 \beta_{\boldsymbol{P}_X} := \beta_X-(\tilde{M}_{\bC} + r\tilde{M}_{\bS})+M_{\bC},\\
 \beta_{\boldsymbol{P}_L} := \left(\dfrac{\lVert f \rVert_{\infty}}{\Delta z} + \dfrac{D(\Xmax)}{\Delta z^2}\right)/(\rho_X-\Xmax) + M_{\bC}, \quad \tilde{M}_{\bS} := \sup_{\boldsymbol{\mathcal{U}}\in\Omega\atop 1\le k\le k_{\bS}}
\left|\pp{\tilde{R}_{\bC}^{(k)}}{S^{(k)}}\right|.
\end{align*}
The norm and constants presented here are defined in Subsection~\ref{subsec:scheme}.
\end{document}